\documentclass[a4paper,11pt]{article}

\usepackage[margin=0.7in]{geometry}
\usepackage[english]{babel}
\usepackage{amsfonts}
\usepackage{amssymb}
\usepackage{amsthm, amsmath}
\usepackage{eucal}
\usepackage{mathrsfs}
\allowdisplaybreaks[0]
\usepackage[hypertexnames=false]{hyperref}
\usepackage{bm}
\usepackage{graphicx}
\usepackage{caption}
\usepackage{subcaption}
\usepackage{algorithm,algorithmic}
\usepackage{framed}
\usepackage{diagbox}
\usepackage[titletoc,title]{appendix}

\usepackage[usenames,dvipsnames]{color}

\newcommand{\R}{\mathbb{R}}

\newtheorem{theorem}{Theorem}[section]

\newtheorem{proposition}[theorem]{Proposition}
\newtheorem{corollary}[theorem]{Corollary}
\newtheorem{definition}[theorem]{Definition}
\newtheorem{remark}[theorem]{Remark}

\DeclareMathOperator*{\argmin}{arg\,min}
\DeclareMathOperator*{\argmax}{arg\,max}

\graphicspath{{./images/}}

\let\oldFootnote\footnote
\newcommand\nextToken\relax

\renewcommand\footnote[1]{%
    \oldFootnote{#1}\futurelet\nextToken\isFootnote}

\newcommand\isFootnote{%
    \ifx\footnote\nextToken\textsuperscript{,}\fi}

\numberwithin{equation}{section}

\begin{document}

\thispagestyle{empty}

\begin{center}
{\Large \bf 
Infimal convolution of data discrepancies for mixed noise removal} \\
\end{center}

\begin{center}
     \footnotesize{            {\sc Luca Calatroni}\\
 Centre de Math\'{e}matiques Appliqu\'{e}es (CMAP)\\
 \'{E}cole Polytechnique CNRS, \\
 Route de Saclay, 91128 Palaiseau Cedex, France\\
 (luca.calatroni@polytechnique.edu) \\[0.4cm]
 {\sc Juan Carlos De Los Reyes}\\
 Research Center on Mathematical Modelling (MODEMAT), Escuela Polit\'ecnica Nacional\\
Ladron De Guevara E11-253, 2507144, Quito, Ecuador \\
(juan.delosreyes@epn.edu.ec) \\[0.4cm]
 {\sc Carola-Bibiane Sch\"{o}nlieb} \\
 Department of Applied Mathematics and Theoretical Physics (DAMTP)\\
 University of Cambridge, Wilberforce Road, CB3 0WA, Cambridge, UK \\
 (cbs31@cam.ac.uk)}
       \end{center}

\begin{abstract}
We consider the problem of image denoising in the presence of noise whose statistical properties are a combination of two different distributions. We focus on noise distributions that are frequently considered in applications, in particular mixtures of salt \& pepper and Gaussian noise, and Gaussian and Poisson noise. We derive a variational image denoising model that features a total variation regularisation term and a data discrepancy that features the mixed noise as an infimal convolution of discrepancy terms of the single-noise distributions. We give a statistical derivation of this model by joint Maximum A-Posteriori (MAP) estimation, and discuss in particular its interpretation as the MAP of a so-called infinity convolution of two noise distributions. Moreover, classical single-noise models are recovered asymptotically as the weighting parameters go to infinity. The numerical solution of the model is computed using second order Newton-type methods. Numerical results show the decomposition of the noise into its constituting components. The paper is furnished with several numerical experiments and comparisons with other existing methods dealing with the mixed noise case are shown.
\end{abstract}

\section{Introduction}  \label{sec:intro}

Let $\Omega\subset\mathbb R^2$ open and bounded with Lipschitz boundary and $f:\Omega\rightarrow \mathbb R$ be a given noisy image. The image denoising problem can be formulated as the task of retrieving a denoised image $u:\Omega\rightarrow \mathbb R$ from $f$. In its general expression, assuming no blurring effect on the observed image, the denoising inverse problem assumes the following form:
\begin{equation}  \label{invprob}
\text{find }\qquad u\qquad\text{s.t.}\qquad f=\mathcal{T}(u),
\end{equation}
where $\mathcal{T}$ stands for the degradation process generating random noise that follows some statistical distribution. The problem \eqref{invprob} is often regularised and reformulated as the following minimisation problem  
\begin{equation}   \label{tik:regularis}
\text{find }\quad u\quad\text{ s.t. }\quad u\in\text{argmin}_{v\in V} \Bigl\{ J(v):= R(v) + \lambda~\Phi(v,f) \Bigr\}, 
\end{equation}
where $f$ and $u$ are elements in suitable function spaces, and the desired reconstructed image is computed as a minimiser $u$ of the energy functional $J$ in the Banach space $V$. The functional $J$ is the sum of two different terms: the regularising energy $R$ which drives the reconstruction process encoding \emph{a-priori} information about the desired image $u$ and the data fidelity function $\Phi$ modelling the relation between the data $f$ and the reconstructed image $u$. The positive weighting parameter $\lambda$ balances the action of the regularisation against the trust in the data. In this paper we consider
\begin{equation}\label{eq:TV}
R(u) = |Du|(\Omega)  = \sup_{\xi\in C_c^\infty(\Omega,\mathbb R^2); ~ \|\xi\|_\infty\leq 1} \int u~\mathrm{div}(\xi)~ dx
\end{equation}
the Total Variation (TV) as image regulariser, which is a popular choice since the seminal works of Rudin, Osher and Fatemi \cite{rudinosherfatemi}, Chambolle and Lions \cite{chambollelions1997} and Vese \cite{vese2001study} due to its edge-preserving and smoothing properties. Under this choice, the minimisation problem \eqref{tik:regularis} is formulated in a subspace of $BV(\Omega)$, the space of functions of bounded variation, \cite{AmbrosioBV}.

Having fixed the regularisation term, depending on the imaging task and the specific application considered, several choices for the fidelity term $\Phi$ and the balancing parameter $\lambda$ can be made. The correct mathematical modelling of the data fidelity term $\Phi$ in \eqref{tik:regularis} is crucial for the design of an image reconstruction model fitting appropriately the given data. Its choice is typically driven by physical and statistical considerations on the noise corrupting the data $f$ in \eqref{invprob}, cf. \cite{idier2013bayesian,stuartInversebook,bovikBook2000}.  In this spirit, the general model \eqref{tik:regularis} is derived as the MAP estimate of the likelihood distribution. In the simplest scenario, the noise is assumed to be an additive random component $w$, Gaussian-distributed with zero mean and variance $\sigma^2$ determining the noise intensity. In this case \eqref{invprob} can be simply written as
$$
f=u+w.
$$
Gaussian noise is often used to model the noise statistics in many applications such as in medical imaging, as a simple approximation of more complicated noise models.  Other additive noise distributions, such as the more heavy-tailed  Laplace distribution can alternatively be considered. Another possibility --  which is appropriate for modelling transmission errors affecting only a percentage of the pixels in the image --  is to consider a type of noise where the intensity value of only a fraction of pixels in the image is switched to either the maximum/minimum value of its dynamic range or to a random value within it with positive probability. This type of noise is called ``salt \& pepper'' noise or ``impulse" noise, respectively. In some other cases, a different, signal-dependent property is assumed to conform to the actual physical application considered. For instance, a Poisson distribution of the noise is used for astronomical and microscopy imaging applications. 

When Gaussian noise is assumed, an $L^2$-type data fidelity
\begin{equation} \label{gaussian:fidelity}
\Phi(u,f)=\int_\Omega (u-f)^2~dx,
\end{equation}
can be derived for $f\in L^2(\Omega)$ as the MAP estimate of the Gaussian likelihood function \cite{bovikBook2000,rudinosherfatemi,chambollelions1997}. Similarly, in the case of additive Laplace noise, the statistically consistent data fidelity term for $f\in L^1(\Omega)$ reads
\begin{equation}  \label{impulse:fidelity}
\Phi(u,f)=\int_\Omega |u-f|~dx,
\end{equation}
see, e.g. \cite{laplacenoiseMarks,benningFidelities}. Furthermore, the same data fidelity is considered in \cite{nikolovaoutremov,duvaltvL1} to model the sparse structure of the salt \& pepper and impulse noise distributions. Variational models where a Poisson noise distribution is assumed are often approximated by weighted-Gaussian distributions through variance-stabilising techniques \cite{starckmurtagh1994, vaggelisPET}. In \cite{alexTVpoisson,Saw11} a statistically-consistent analytical modelling has been derived: for $f\in L^\infty(\Omega)$ the resulting data fidelity term is a Kullback-Leibler-type functional $\Phi$ of the form
\begin{equation}  \label{poisson:fidelity}
\Phi(u,f)=\int_\Omega \left( u - f\log u \right)~dx.
\end{equation}

As a result of different image acquisition and transmission factors, very often in applications the observed image is corrupted by a mixture of noise statistics. For instance, mixed noise distributions can be observed when faults in the acquisition of the image are combined with transmission errors to the receiving sensors. In this case, the modelling of \eqref{tik:regularis} shall encode a combination of salt \& pepper and Gaussian noise in the choice of $\Phi$. In other applications, specific tools (such as fluorescence and/or high-energy beams) are used before the signal is actually acquired. This process is typical, for instance, in microscopy and astronomy and may result in a combination of a Poisson noise with an additive Gaussian noise component \cite{sarder2006,snyder93}. 

From a modelling point of view, the presence of multiple noise distributions has been translated in the relevant literature in the combination of the data fidelity terms \eqref{gaussian:fidelity}, \eqref{impulse:fidelity} and \eqref{poisson:fidelity} in different ways. In \cite{HintermuellerLanger2013}, for instance, a combined $L^1+L^2$ data fidelity with TV regularisation model is considered for joint impulsive and Gaussian noise removal. There, the image is spatially decomposed in two disjoint domains in which the image is separetely corrupted by pure salt \& pepper and Gaussian noise, respectively. A two-phase approach is considered in \cite{impulsegauss2008} where two sequential steps with $L^1$ and $L^2$ data fidelity are performed to remove the impulsive and the Gaussian component of the mixed noise, respectively.  Mixtures of Gaussian and Poisson noise have also been considered. In \cite{poissongauss2013,Jezierska2012}, for instance, the exact log-likelihood estimator of the mixed noise model is derived and its numerical solution is computed via a primal-dual splitting. A similar model has been considered in \cite{Benvenuto2008} where a scaled gradient semi-convergent algorithm is used to solve the combined model. In \cite{Foigausspoiss} the discrete-continuous nature of the model (due to the different support of the Poisson-Gaussian distribution on the set of natural and real numbers, respectively) is approximated by an additive model, using homomorphic variance-stabilising transformations and weighted-$L^2$ approximations. In \cite{Luisier2010a} a non-Bayesian framework to differentiate Poisson intensities from the Gaussian ones in the Haar wavelet domain is considered. In \cite{lanza2013} a Gaussian-Poisson model similar to the one we consider in this paper is derived in a discrete setting. A more recent approach featuring a linear combination of different data fidelities of the type \eqref{gaussian:fidelity}, \eqref{impulse:fidelity} and \eqref{poisson:fidelity} has been considered in \cite{langerl1l2} for a combination of Gaussian and impulse noise and in \cite{noiselearning,lucasampling,bilevellearning} in the context of bilevel learning approaches for the design of optimal denoising models.

In this work, we present an alternative variational model for denoising of images corrupted by mixed noise distributions that is based on an infimal convolution of the data fidelity terms \eqref{gaussian:fidelity}, \eqref{impulse:fidelity} and \eqref{poisson:fidelity}. For the case of Gaussian-Poisson noise a similar model is discussed in \cite{lanza2013} in a finite-dimensional setting. Our model is derived from statistical assumptions on the data and can be studied rigorously in function spaces using standard tools in calculus of variations and functional analysis. The simple infimal convolution nature of the model derived makes its numerical solution amenable for standard first or second order optimisation methods. In this paper we use a semi-smooth Newton (SSN) method for the numerical realisation of our model. By using the classical operation of infimal convolution, our variational model combines different data fidelities $\Phi$, each associated to the corresponding noise component in the data, allowing for splitting the noise into its constituting elements. In what follows, we fix the regularisation term in \eqref{tik:regularis} to be the TV energy. This is of course just a toy example and extensions to higher-order regularisations such as TV-TV$^2$ \cite{kostasTV2} or TGV \cite{TGV} can be considered as well. We refer to our model as TV-IC to highlight the infimal convolution combination of data fidelities. This should not be confused with the ICTV model \cite{chambollelions1997,HollerKunischIC2014} where the same operation is used to combine TV regularisation with second-order regularisation terms. 

\subsection{The reference models} \label{sec:reference:models} 

In the following, we consider two exemplar problems which extend \eqref{invprob} to the case of multiple noise distributions. We consider the following problem encoding noise mixtures
\begin{equation}   \label{invprob1}
\text{find}\quad u \quad \text{such that}\quad f=\mathcal{T}(u) + w,
\end{equation}
where $\mathcal{T}$ models a general noising process possibly depending on $u$ in a non-linear way, while $w$ is an additive, Gaussian-distributed, noise component independent of $u$. We will focus in the following on two particular cases of \eqref{invprob1}.

\medskip

\paragraph{Salt \& pepper and Gaussian noise} For this case, we suppose that the observed noisy image $f$ is corrupted by a mixture of salt \& pepper and  Gaussian noise, i.e. we consider the following instance of the noising model \eqref{invprob1},
\begin{equation} \label{laplacegaussianoise} 
f=(1-s)u+sc+w
\end{equation}
where, following \cite[Section 1.2.2]{chanshenbook},  the salt \& pepper component is modelled by considering two independent random fields $s$ and $c$ defined, for every pixel $x\in\Omega$ defined as
$$
c(x)=\begin{cases}
0,\quad \text{with probability }p=1/2 \\
1,\quad \text{with probability }q=1/2
\end{cases},\quad 
s(x)=\begin{cases}
0,\quad \text{with probability }p \\
1,\quad \text{with probability }1-p
\end{cases} .
$$
Following \cite[Section 2.2]{rodriguezReview2013} and \cite{laplacenoiseMarks}, we will approximate the nonlinear model \eqref{laplacegaussianoise} by the following additive model:
\begin{equation} \label{additivenoise}
f=u+v+w,
\end{equation}
where $v$ is now the realisation of a Laplace distributed random variable, independent of $u$. This approximation will be made more clear in Section \ref{sec:statconst}.

\medskip

\paragraph{Poisson and Gaussian noise} In this case, we consider the problem of a mixture of noise distributions where a \emph{signal-dependent} component is combined with an additive Gaussian noise component. In particular, we will focus on a Poisson distributed component $z$ having $u$ as Poisson parameter. In other words, we will consider the noising model \eqref{invprob1} with
\begin{equation}  \label{gausspoissnoise}
 f=z+w,\qquad \text{where }z\sim Pois(u).
\end{equation}

\medskip

\paragraph{Organisation of the paper} 
In Section \ref{sec:model} we present the infimal-convolution data fidelity terms considered for the reference models above and show that they are well-defined. Motivations for the TV-IC denoising model are given in Section \ref{sec:statconst} where some considerations based on the use of Bayes' formula are given to support our modelling. In particular, we make precise how the mixed data fidelity terms can be derived in two different ways, using a joint MAP estimation strategy as in \cite{lanza2013}, and by deriving a joint posterior distribution using and $L^\infty$ variant of the classical convolution operation of two probability distributions.  In Section \ref{subsec:wellposednessBV} well-posedness results of the TV-IC model by means of classical tools of calculus of variations and functional analysis are given. In Section \ref{sec:asympt} we show how classical single noise models can be recovered asymptotically from our model by letting the fidelity parameters go to infinity. In Section \ref{sec:numres} we confirm the validity of our model with several different numerical experiments using a semi-smooth Newton method for its efficient numerical solution. Finally, in Section \ref{sec:learning_mot} we report some preliminary results in the spirit of recent advances in noise model learning via bilevel optimisation \'a la \cite{bilevellearning}.

\section{The variational model} \label{sec:model}

We consider a open and bounded image domain $\Omega\subset\R^2$ with Lipschitz boundary and denote by $f$ the given, noisy image. Further assumptions on the function spaces where $f$ lies  will be specified in the following.  

Let $R(u):=|Du|(\Omega)$ the total variation semi-norm defined in \eqref{eq:TV} and let $BV(\Omega)$ denote the space of functions of bounded variation, \cite{AmbrosioBV} . For positive weights $\lambda_1,~ \lambda_2>0$ and admissible sets of functions $\mathcal{A}$ and $\mathcal{B}$, the general proposed TV-IC denoising model for mixed noise distributions reads: 

\begin{subequations}  \label{generalinfconvproblem}
\begin{equation}  \label{infconvproblemu}
\min_{u\in BV(\Omega)\cap\mathcal{A}} \left\{|Du|(\Omega) + \Phi^{\lambda_1,\lambda_2}(u,f) \right\}   \tag{TV-ICa}
\end{equation} 
where the data fidelity $\Phi^{\lambda_1,\lambda_2}(u,f)$ has the following infimal convolution-type structure:
\begin{equation}  \label{infconvproblem2}
\Phi^{\lambda_1,\lambda_2}(u,f):=\inf_{v\in L^2(\Omega)\cap\mathcal{B}} \left\{ \mathscr{F}^{\lambda_1,\lambda_2}(f,u,v):=\lambda_1~ \Phi_1(v) + \lambda_2 ~\Phi_2(u,f-v) \right\},   \tag{TV-ICb}
\end{equation}
\end{subequations}
for two data fidelity functions $\Phi_1$ and $\Phi_2$ defined over $L^2(\Omega)\cap\mathcal{B}$. The size of the weighting parameters $\lambda_1$ and $\lambda_2$ in \eqref{infconvproblem2} balances the trust in the data against the smoothing effect of the regularisation as well as the fitting with respect to the intensity of each single noise distribution in $f$. 

We now consider the two particular noise combinations introduced in Section \ref{sec:reference:models} and demonstrate how these fit into the general model above. 

\subsection{Salt \& pepper-Gaussian fidelity}   \label{subsec:gaussimp}
For the case of mixed salt \& pepper and Gaussian noise \eqref{additivenoise}, we specify the assumptions for the model in \eqref{infconvproblemu}-\eqref{infconvproblem2} as follows. We consider the noisy image $f\in L^2(\Omega)$, $\Phi_1(v)=\| v \|_{L^1(\Omega)}$ for the impulsive noise component and $\Phi_2(u,f-v)=\frac{1}{2}\| f-v-u\|_{L^2(\Omega)}^2$ for the Gaussian noise component. The admissible sets $\mathcal A, \mathcal B$ are simply $\mathcal{A}=\mathcal{B}=L^2(\Omega)$. In this case, the infimal convolution data fidelity in \eqref{infconvproblem2} reads
\begin{equation}   \label{impgaussinfconvfidelity}
\Phi^{\lambda_1,\lambda_2}(u,f)= \inf_{v\in L^2(\Omega)} \left\{\mathscr{F}^{\lambda_1,\lambda_2}(f,u,v)= \lambda_1 \| v\|_{L^1(\Omega)} + \frac{\lambda_2}{2}\| f-u-v\|_{L^2(\Omega)}^2 \right\}.
\end{equation}
Since the set $\Omega$ is bounded, $L^2(\Omega)\subset L^1(\Omega)$ and both terms in \eqref{impgaussinfconvfidelity} are well-defined.

\begin{remark} \label{remark:prox_map} 
Note that \eqref{impgaussinfconvfidelity} can be rewritten as the \emph{Moreau envelope} or as a weighted \emph{proximal map} of the $L^1$-norm in terms of the ratio between the parameters $\lambda_1$ and $\lambda_2$ (see \cite[Section 12.4]{bauschkecombettes}). Namely,  we have:
$$
\Phi^{\lambda_1,\lambda_2}(u,f)= \mbox{prox}_{\frac{\lambda_1}{\lambda_2}\Phi_1}(f-u) = \frac{\lambda_1}{\lambda_2}~M^{\frac{\lambda_1}{\lambda_2}}_{\Phi_1},
$$
where $\mbox{prox}_\Psi$ denotes the proximal map of a generic function $\Psi$ while $M^{\gamma}_\Psi$ indicates the \emph{Moreau envelope} of $\Psi$ with parameter $\gamma$. Therefore, $\Phi^{\lambda_1,\lambda_2}$ inherits in this case several good properties of proximal maps such the firmly non-expansiveness (\cite[Proposition 12.27]{bauschkecombettes}) which will be useful in the following analysis.
\end{remark}

The following proposition asserts that the minimisation problem \eqref{impgaussinfconvfidelity} is well-posed.

\begin{proposition}  \label{wellposednimpgauss}
Let  $f\in L^2(\Omega)$, $u\in \text{BV}(\Omega)\subset L^2(\Omega)$ and $\lambda_1,\lambda_2>0$. Then the minimum in the minimisation problem \eqref{impgaussinfconvfidelity} is uniquely attained.
\end{proposition}

\begin{proof}
The proof of this proposition is based on the use of standard tools of calculus of variations. We report it in Appendix \ref{append:proofs}.
\end{proof}

\subsection{Gaussian-Poisson fidelity}   \label{subsec:gausspoiss}

For the case of mixed Gaussian and Poisson noise  \eqref{gausspoissnoise} some additional assumptions need to be specified. The given image $f$ is assumed to be bounded, that is $f\in L^\infty(\Omega)$ and the data fidelities $\Phi_1$ and $\Phi_2$, encoding the Gaussian and the Poisson component of the noise, respectively, are then chosen as $\Phi_1(v)=\frac{1}{2}\| v \|_{L^2(\Omega)}^2$ and $\Phi_2$ as the Kullback-Leibler (KL) divergence between $u$ and the ``residual" $f-v$, that is
\begin{equation} \label{fullKL}
\Phi_2(u,f-v)=D_{KL}(f-v,u)=\int_\Omega \left(u - (f-v) +(f-v)\log \left(\frac{f-v}{u}\right)\right)~d\mu.
\end{equation} 
We refer the reader to Appendix \ref{append:KLfunct} where more properties of the KL functional are given.
The variational fidelity $\Phi^{\lambda_1,\lambda_2}$ in \eqref{infconvproblem2} then reads
\begin{equation}   \label{gausspoissinfconvfidelity}
\Phi^{\lambda_1,\lambda_2}(u,f)= \inf_{v\in L^2(\Omega)\cap\mathcal{B}} \left\{ \mathscr{F}^{\lambda_1,\lambda_2}(f,u,v)=\frac{\lambda_1}{2} \| v\|^2_{L^2(\Omega)} + \lambda_2~D_{KL}(f-v,u) \right\}
\end{equation}
with the following choice of the admissible sets:
 \begin{equation}    \label{admissiblesetsgausspoiss}
 \mathcal{A}=\left\{ u\in L^1(\Omega),~ \log u \in L^1(\Omega)\right\}\quad\text{ and }\quad\mathcal{B}=\left\{ v\in L^2(\Omega): v\leq f\ \text{ a. e.} \right\}.
 \end{equation}
We note that for $u\in\mathcal{A}$ we have $u\geq0$ almost everywhere. Together with the definition of $\mathcal{B}$, this ensures that the $D_{KL}$ functional is well defined (using the standard convention $0\log 0 = 0$).

\begin{remark}
Our choice for the Poisson data fidelity \eqref{fullKL} differs from previous works on Poisson noise modelling such as \cite{Saw11,noiselearning} where a reduced KL-type data fidelity
\begin{equation*}   
\tilde{\Phi}_2(u,f)=\int_\Omega \left(u - f~\log u\right)~d\mu
\end{equation*} 
is preferred. This choice is motivated through Bayesian derivation and MAP estimation (compare  \cite{Saw11,vaggelisPET,benningFidelities}), where the terms that do not depend on the de-noised image $u$ are neglected since they are not part of the optimisation argument. In our combined modelling, however, those quantities have to be taken into account to incorporate the additional variable $v$ encoding the Gaussian noise component.
\end{remark}

Similarly as in Proposition \ref{wellposednimpgauss}, we guarantee that the minimisation problem \eqref{gausspoissinfconvfidelity} is well posed in the following proposition. 

\begin{proposition}  \label{wellposednessgausspoiss}
Let $f\in L^\infty(\Omega)$, $u\in BV(\Omega)\cap\mathcal{A}$ and $\lambda_1,\lambda_2>0$. Let $\mathcal{A}$ and $\mathcal{B}$ be as in  \eqref{admissiblesetsgausspoiss}. Then, the minimum in the minimisation problem \eqref{gausspoissinfconvfidelity} is uniquely attained.
\end{proposition}

\begin{proof}
We report the proof in Appendix \ref{append:proofs} It is based on standard tools of calculus of variations and on the properties of Kullback-Leibler divergence recalled in Appendix \ref{append:KLfunct}. 
\end{proof}

\section{Statistical derivation and interpretation}  \label{sec:statconst}

In this section we want to give a statistical motivation for the choice of the IC data fidelity term introduced in the previous section. We do this by switching from the continuous setting with $f$ in an infinite dimensional function space, to the discrete setting, where $f=(f_i)_{i=1}^M$ is a vector in $\mathbb R^M$ and the noise distribution of each $f_i$ is independent of the others. Correspondingly, we consider an appropriate discretisation of the continuous TV-IC variational model\eqref{infconvproblemu}-\eqref{infconvproblem2}. In this setting, we will show, how the IC data fidelity can be derived in two ways: in terms of a joint MAP estimate for $u$ and $v$ in Section \ref{sec:mapderivation}, and by considering a modified likelihood for the mixed noise distribution  -- which consists of an infinity convolution of the two noise distributions -- and computing its MAP estimate with respect to $u$ in Section \ref{sec:infinityconv}.

\subsection{Joint maximum-a-posterior estimation}\label{sec:mapderivation}

A standard approach used to derive variational regularisation methods for inverse problems which reflect the statistical assumptions on the data is the Bayesian approach \cite{chanshenbook,stuartInversebook}. In this framework, the problem is formulated in terms of the maximisation of the posterior probability density $P(u|f)$, i.e. the probability of observing the desired de-noised image $u$ given the noisy image $f$. This approach is commonly known as Maximum A Posteriori (MAP)  estimation and relies on the simple application of the Bayes' rule
\begin{equation}   \label{bayes}
P(u|f)=\frac{P(f|u)P(u)}{P(f)},
\end{equation}
by which maximising $P(u|f)$ for $u$ translates in maximising the ratio on the right hand side of \eqref{bayes}. Classically, the probability density $P(u)$ is called prior probability density since it encodes statistical a priori information on $u$. Frequently, a Gibbs model of the form  
\begin{equation} \label{gibbs_model}
P(u)=e^{-\alpha R(u)},~ \alpha>0,
\end{equation}
 where $R(u)$ is a convex regularising energy functional is assumed. In what follows, we take $R$ to be a discretisation of the total variation \eqref{eq:TV}, that is denoting by $k,l$ the pixel-coordinates and reordering within a one-dimensional vector with index $i$, thus considering $u=(u_i)_{i=1}^M=(u_{k,l})\in\mathbb R^M$, we have
 $$
 R(u)= \|\nabla u\|_{2,1} = \sum_i |\nabla u_i| = \sum_{k,l} \sqrt{(u_{k+1,l}-u_{k,l})^2 + (u_{k,l+1}-u_{k,l})^2},
$$ 
where $\nabla$ denotes the forward-difference approximation of the two-dimensional gradient. The conditional probability density $P(f|u)$, also called likelihood function, relates to the statistical assumptions on the noise corrupting the data $f$. Its expression depends on the probability distribution of the noise assumed on the data. The denominator of \eqref{bayes} plays the role of a normalisation constant for $P(u|f)$ to be a probability density.
 
Clearly, maximising \eqref{bayes} in our setting of a Gibbs prior is equivalent to minimising the negative logarithm of $P(u|f)$. Thus, the MAP estimation problem equivalently corresponds to the problem of finding 
\begin{equation*}   
u_{MAP}\in \argmin_u \left\{ -\log P(u|f) \right\} = \argmin_u \left\{ -\log P(f|u) + \alpha R(u) \right\},
\end{equation*}
where the term $\log P(f)$ has been neglected since it does not affect the minimisation over $u$. 

We now want to use MAP estimation to derive a discretisation of the TV-IC variational model \eqref{infconvproblemu}-\eqref{infconvproblem2} for the mixed noise scenario. From a Bayesian point of view, the general model \eqref{invprob1} can be written in terms of random variables $F$ and $U$ (for which the noisy image $f$ and the noise-free image $u$ are realisations, respectively) as follows
\begin{equation}   \label{structureMAP}
F=Z+W\quad Z\sim P^U_Z,\quad U\sim P_U, \quad W\sim P_W:=\mathcal{N}(0,\sigma^2),
\end{equation}
Here, $F$ is the sum of the two random variables $Z$ and $W$. The random variable $Z$ is distributed according to a probability distribution $P^U_Z$ which may in general depend on $U$ (for instance, $U$ could be one of its parameters), whereas $W$ is fixed to be a Gaussian random variable independent of $U$ and $Z$, which combined with $Z$ through an additive modelling. Here, we denote by $P_U$  the Gibbs model \eqref{gibbs_model} on $U$, i.e. the prior distribution on $U$, the quantity of interest, so that $P_U(u)=P(u)$.

For our derivation we consider a joint MAP estimation for a pair of realisations $(u,w)$ of the corresponding random variables $U$ and $W$. Using Bayes' rule and mutual independence between $W$ and $U$, we have:  
\begin{multline}\label{jointMAP_u_n}
(\hat{u},\hat{w}) = \argmax_{(u,w)}~P(u,w|f)=\argmax_{(u,w)}~P(f|u,w)P(u,w)\\ =  \argmax_{(u,w)}~P^U_Z(f-w)P_W(w)P_U(u)
\end{multline}
which holds for general noise distributions $P_Z^U$. However, in the particular case where  $Z=U+V$ with $V$ being a random variable distributed as $V\sim P_V$ and independent of $U$ (in particular, when $U$ is not a parameter of the probability distribution of $Z$ and is combined additively with another random variable $V$), \eqref{structureMAP} becomes
\begin{equation}  \label{MAPsum}
F=U+V+W.
\end{equation}
In this case $P_Z^U(z)=P_V(z-u)$, and \eqref{jointMAP_u_n} reduces to:
\begin{equation}  \label{jointMAP1}
(\hat{u},\hat{w})  = ~\argmax_{(u,w)}~P^U_Z(f-w)P_W(w)P_U(u) =  \argmax_{(u,w)} P_V(f-u-w)P_W(w)P_U(u).
\end{equation}

After these general considerations, let us now carry on with the detailed MAP derivation, first for the additive mixture of Laplace noise noistribution (as an approximation for salt \& pepper distribution) and Gaussian noise distribution \eqref{additivenoise}, and then for the mixed Gaussian-Poisson case \eqref{gausspoissnoise}.

\medskip

\paragraph{The additive case} Let us focus first on the additive model \eqref{MAPsum}  in the case of a mixture of Laplace- and Gaussian-distributed noise. Following \cite{rodriguezReview2013}, we will use the Laplace distribution as an approximation for the one describing salt \& pepper noise. More precisely, such distribution is an example of heavy-tailed distribution whose behaviour can be well approximated by the heavy-tailed Laplace noise distribution \cite{laplacenoiseMarks,laplacenoiseWu}. We start from formulating the problem \eqref{MAPsum} in a finite-dimensional statistical setting. For each image pixel $i=1,\ldots,M$, our problem is 
\begin{equation} \label{finite_dim_additive}
\text{find}\quad u_i\quad\text{s.t.}\quad f_i=u_i+v_i+w_i,
\end{equation}
where $v_i$ and $w_i$ are realisations of two mutually independent random variables distributed according to Laplace and Gaussian distribution, respectively. More precisely, the equation above describes the structure of the realisations $(f_1,\ldots, f_M)$ of the components of the random vector $F:=(F_1,\ldots,F_M)$ in terms of the ones of the random vectors $ U:=(U_1,\ldots,U_M), ~V:=(V_1,\ldots,V_M)$ and $W:=(W_1,\ldots,W_M)$. The quantities $v_i$ and $w_i$ are, for every $i=1,\ldots,M$, independent realisations of identically distributed Laplace and Gaussian random variables $V_i$ and $W_i$, respectively, i.e. $V_i\sim Lapl(0,\tau)=:P_{V_i}$ and $W_i\sim\mathcal{N}(0,\sigma^2)=:P_{W_i}$. 
 
We recall that the Laplace probability density with parameter $\tau>0$ is defined as
\begin{equation} \label{probdensities:imp}
P_{V_i}(s)=\frac{e^{-|s|/\tau}}{2\tau},\quad s\in\R,
\end{equation}
and the Gaussian probability density with zero-mean and variance $\sigma^2$ is defined as
\begin{equation}   \label{probdensities:gaussian}
P_{W_i}(s)=\frac{e^{-|s|^2/2\sigma^2}}{\sqrt{2\pi\sigma^2}},\quad s\in\R.
\end{equation}
Let us now take the negative logarithm in the joint MAP estimate \eqref{jointMAP1} for all the realisations. By independence, we have that:
\begin{align} 
(\hat{u},\hat{w}) & = ~\argmin_{\substack{u=(u_1,\ldots,u_M),\\ w=(w_1,\ldots,w_M)}}~-\log \Bigl(\prod_{i=1}^M P_{V_i}(f_i-u_i-w_i)P_{W_i}(w_i)P_{U_i}(u_i)\Bigr) \notag \\
& = ~ \argmin_{\substack{u=(u_1,\ldots,u_M),\\ w=(w_1,\ldots,w_M)}}- \sum_{i=1}^M ~\log \Bigl( P_{V_i}(f_i-u_i-w_i)P_{W_i}(w_i)P_{U_i}(u_i)\Bigr) \label{loglikel:prop1}\\
& = ~ \argmin_{\substack{u=(u_1,\ldots,u_M),\\ w=(w_1,\ldots,w_M)}} \sum_{i=1}^M \Bigl( \frac{|f_i-u_i-w_i|}{\tau}+\frac{|w_i|^2}{2\sigma^2} +\alpha |\nabla u_i| \Bigr).  \notag
\end{align}
where the constant terms which do not affect the minimisation over $u$ and $w$ have been neglected.

To pass from a discrete to a continuous expression of the model we follow \cite{Saw11}. We interprete the elements in the space $\R^M$ as samples of functions defined on the whole image domain $\Omega$. For convenience, we use in the following the same notation to indicate the corresponding, continuous quantities. Introducing the indicator function 
\[
\chi_{D_i}(x) = 
\begin{cases}
1~,\quad\text{if }x\in D_i, \\
 0~,\quad\text{else},
\end{cases}
\]
where $D_i$ is the region in the image occupied by the $i$-th detector, we have that any discrete data $h_i$ can be interpreted as the mean value of a function $h$ over the region $D_i$ as follows
\[
h_i = \int_{D_i} h(x)~dx = \int_{\Omega} \chi_{D_i}(x)h(x)~dx.
\]
In this way, we can then express \eqref{loglikel:prop1} as the following continuous model 
\begin{equation*}  
\min_{u,w:~\Omega\to\R}~\int_\Omega \left( \frac{|f(x)-u(x)-w(x)|}{\tau}+\frac{|w(x)|^2}{2\sigma^2} + 
\alpha |\nabla u(x)|\right) ~ d\mu(x)
\end{equation*}
where $d\mu(x)=\sum_{i=1}^M \chi_{D_i}dx$ with $dx$ being the usual Lebesgue measure in $\R^2$. We observe that at this level, the function spaces (i.e. the regularity) where the minimisation problem is posed still need to be specified. Defining  $\lambda_1:=1/\alpha\tau$ and $\lambda_2:=1/\alpha\sigma^2$ and replacing $\int_\Omega |\nabla u(x)|~ dx$ by $|Du|(\Omega)$, we derive from \eqref{commuteinfint} the following variational model for noise removal of Laplace and Gaussian noise mixture
\begin{equation*}
\min_{\substack{u\in BV(\Omega)\\ w\in L^2(\Omega)}}~|Du|(\Omega) + \lambda_1 \| f-u-w \|_{L^1(\Omega)} + \frac{\lambda}{2}\| w\|_{L^2(\Omega)}^2,
\end{equation*}
where the function spaces for $u$ and $w$ have been chosen so that all the terms are well defined. By a simple change of variables, we can equivalently write the model above as
\begin{equation}   \label{TV_IC_additive2}
\min_{\substack{u\in BV(\Omega)\\ v\in L^2(\Omega)}}~|Du|(\Omega) + \lambda_1 \|v \|_{L^1(\Omega)} + \frac{\lambda}{2}\|  f-u-v\|_{L^2(\Omega)}^2,
\end{equation}
where the minimisation is taken over the de-noised image $u$ and the salt \& pepper noise component $v$. 

\medskip

\paragraph{The signal-dependent case} Let us now consider \eqref{structureMAP} in the case when the non-Gaussian noise component in the data $Z$ follows a Poisson probability distribution with parameter $U$. In this case, 
the model \eqref{structureMAP} can be formulated in a finite-dimensional setting as
\begin{equation} \label{poisson_gauss_strcture}
f_i = z_i + w_i,\quad\text{with } z_i\sim Poiss(u_i)
\end{equation}
at every pixel $i=1,\ldots,M$. Similarly as before, the equation above describes the structure of the realisations of the random vector $F=(F_1,\ldots,F_M)$ in terms of the sum of two mutually independent random vectors $Z:=(Z_1,\ldots,Z_M)$ and  $W:=(W_1,\ldots,W_M)$, having as components independent and identically distributed Poisson and Gaussian random variables $Z_i\sim Pois(u_i)$ and  $W_i\sim \mathcal{N}(0,\sigma^2)$, respectively, for every $i=1,\ldots,M$. Again, the values $u_i$ are  random realisations of the random variables $U_i$ which are distributed according to the Gibbs model \eqref{gibbs_model} introduced before.

The Poisson noise density with parameter $u_i$ is defined as
\begin{equation} \label{poissondistribution}
P^{U_i}_{Z_i}(z_i)=\frac{{u_i}^{z_i} e^{-u_i}}{z_i!},\qquad z_i\in\R,
\end{equation}
where the factorial function is classically extended to the whole real line by using the Gamma function. We note that \eqref{poissondistribution} can be rewritten equivalently as:
\begin{equation*}
P^{U_i}_{Z_i}(z_i)=\exp\Bigl ( -u_i + \log\Bigl( \frac{u_i^{z_i}}{z_i!} \Bigr) \Bigr)\qquad z_i\in\R.
\end{equation*}

Similarly as above, let us now take the negative logarithm in the general joint MAP estimate \eqref{jointMAP1}, for all the realisations. By independence, we have:
\begin{align} 
(\hat{u},\hat{w}) & = ~\argmin_{\substack{u=(u_1,\ldots,u_M),\\ w=(w_1,\ldots,w_M)}}~-\log \Bigl(\prod_{i=1}^M P^{U_i}_{Z_i}(f_i-w_i)P_{W_i}(w_i)P_{U_i}(u_i)\Bigr) \notag \\
& = ~ \argmin_{\substack{u=(u_1,\ldots,u_M),\\ w=(w_1,\ldots,w_M)}} - \sum_{i=1}^M \log \Bigl( P^{U_i}_{Z_i}(f_i-w_i)P_{W_i}(w_i)P_{U_i}(u_i)\Bigr)  \notag\\
& = ~ \argmin_{\substack{u=(u_1,\ldots,u_M),\\ w=(w_1,\ldots,w_M)}} \sum_{i=1}^M \Bigl( u_i - \log\Bigl( \frac{u_i^{f_i-w_i}}{z_i!}\Bigr) +\frac{|w_i|^2}{2\sigma^2} +\alpha|\nabla u_i|\Bigr), \notag
\end{align}
where the constant terms which do not affect the minimisation over $u$ and $w$ have been neglected. Passing from a discrete to a continuous modelling similarly as described in the discussion for the additive case, we obtain the following continuous model
\begin{equation}   \label{modellinggausspoiss1}
\min_{u,w}~ |Du|(\Omega)+\frac{\lambda_1}{2} \| w \|_{L^2(\Omega)}^2 + \lambda_2\int_\Omega \left( u(x) -\log\Bigl( \frac{u(x)^{f(x)-w(x)}}{(f(x)-w(x))!} \Bigr) \right)~d\mu(x),
\end{equation}
where we have set $\lambda_1:=1/\alpha\sigma^2$ and $\lambda_2:=1/\alpha$ and we have still to specify the function spaces where the minimisation takes place. By a closer inspection on the third term in the model above, we have:
\begin{align*}  \label{modellinggausspoiss2}
&  \int_\Omega \left( u(x) -\log\Bigl( \frac{u(x)^{f(x)-w(x)}}{(f(x)-w(x))!} \Bigr) \right)~d\mu(x) \\
&= \int_\Omega \left(u(x) - \log\Bigl( u(x)^{f(x)-w(x)} \Bigr) + \log\left( \Bigl( f(x) -w(x) \Bigr)! \right) \right) ~d\mu(x)\\
&= \int_\Omega \left(u(x) - (f(x)-w(x))\log\left( u(x) \right) + \log\left( \Bigl( f(x) -w(x) \Bigr)! \right) \right)~d\mu(x) \\& \approx \int_\Omega \left(u(x) - (f(x)-w(x))\log\left( \frac{f(x)-w(x)}{u(x)} \right) -(f(x)-w(x))\right)~d\mu(x) = D_{KL}(f-w,u),
\end{align*}
where we have used the standard Stirling approximation of the logarithm of the factorial function. The functional we end up with is the well-known Kullback-Leibler (KL) functional which has been used for the design of imaging models used for Poisson noise removal in previous works, such as \cite{alexTVpoisson,Saw11,lechartrandTVpoisson}.  We refer the reader also to Appendix \ref{append:KLfunct} where some  properties of $D_{KL}$ are recalled. 
The variational Poisson-Gaussian data fidelity \eqref{modellinggausspoiss1} can then be written in a more compact form as
\begin{equation}  \label{infconvpoissgauss}
\min_{\substack{u\in BV(\Omega)\mathcal{A} \\w\in L^2(\Omega)\cap\mathcal{B}}}~|Du|(\Omega) +\frac{\lambda_1}{2}~ \| w \|_{L^2(\Omega)}^2 + \lambda_2~ D_{KL}(f-w,u),
\end{equation}
where the admissible sets $\mathcal{A}$ and $\mathcal{B}$ are chosen as specified in \eqref{admissiblesetsgausspoiss} to guarantee that all the terms in \eqref{infconvpoissgauss} are well-defined. We point out here again that in \cite{lanza2013} a similar joint MAP estimation for a mixed Gaussian and Poisson modelling was considered.

\medskip

In both the general model \eqref{structureMAP} and in its  simplified additive version \eqref{MAPsum} we have derived through joint MAP estimation the two variational models \eqref{TV_IC_additive2} and \eqref{infconvpoissgauss} which,
consistently with our statistical assumptions, model the single noise components by norms classically used for the corresponding single noise models and combined together for the mixed noise case through a joint minimisation procedure. The fidelity terms are weighted against each other by the parameters $\lambda_1$ and $\lambda_2$, whose size depends on the intensity of each noise component, and weight the fidelity against the TV regularisation.

\medskip

In the next section, we discuss a different interpretation of the TV-IC model, as the MAP estimate for a posterior distribution in which the two noise distributions are combined by a so-called infinity convolution, instead of the classical convolution of probability distributions. 

\subsection{Infinity convolution} \label{sec:infinityconv}

From a probabilistic point of view, convolution-type data fidelities \eqref{infconvproblem2} can be alternatively derived via a modification of the expression of the probability density of a sum of random variables. Given two independent real-valued random variables $V$ and $W$ with associated probability densities $f_V$ and $f_W$, it is well-known that the random variable $Z:=V+W$ has probability density $f_Z$ given by the convolution of $f_V$ and $f_W$, that is
\begin{equation}  \label{probdensZ:classconv}
f_Z(z)=\int_\R f_V(v)f_W(z-v)~dv=(f_V\ast f_W)(z).
\end{equation}
Following \cite[Remark 2.3.2]{hiriarturruty}, since $f_V$ and $f_W$ are nonnegative, we can define for $p>0$ the convolution of order p as follows:
\begin{equation} \label{probdensZ:convordp}
f_Z^p(z):=(f_V\ast_p f_W)(z)=\left( \int_\R \Bigl (f_V(v)f_W(z-v) \Bigr )^{p}~dv \right)^{1/p},
\end{equation}
which clearly corresponds to the classical convolution \eqref{probdensZ:classconv} if $p=1$.
By letting $p\to\infty$ it is easy to show that \eqref{probdensZ:convordp} converges to the infinity convolution of $f_V$ and $f_W$
\begin{equation} \label{probdensZ:convordinf}
f_Z^\infty(z):=((f_V\ast_\infty f_W)(z)=\sup_{v\in\R} ~ f_V(v)f_W(z-v).
\end{equation}
We note that $f_Z(z)$ in \eqref{probdensZ:classconv} and $f^\infty$ have the same domain. Also, note that in order for $f^\infty$ to be a probability density the following normalisation
$$
\tilde f_Z^\infty(z):=\frac{1}{\int f_Z^\infty(z) dz} ~f_Z^\infty(z)
$$ 
is needed.
Then,  $\tilde f_Z^\infty$ is a probability density, having the same domain of $f_Z$, but potentially featuring heavier tails. 

In the special case when the probability densities $f_V$ and $f_W$ have an exponential form of the type $f_V(\cdot)= c_V e^{-g_V(\cdot)}$ and $f_W(\cdot)= c_W e^{-g_W(\cdot)}$, with $c_V$ and  $c_W$ being positive constants and $g_V, g_W:\R\to\R^+$ continuous and convex positive functions, \eqref{probdensZ:convordinf} can be equivalently rewritten as:
\begin{equation} \label{probdensZ:convordinfexp}
(f_V\ast_\infty f_W)(z)=c_V c_W~ e^{-\inf\limits_{v\in\R} \left( g_V(v)+g_W(z-v)\right)}.
\end{equation}
Hence, the infinity convolution of $f_V$ and $f_W$ corresponds to the infimal convolution of $g_V$ and $g_W$ through a negative exponentiation and up to multiplication by positive constants.

\medskip

Let us now recall the additive model in the finite dimensional setting \eqref{finite_dim_additive}. At every pixel, $i=1,\ldots,M$ the random variable $Y_i:=F_i-U_i=V_i+W_i$ is the sum of the two independent random variables $V_i$ and $W_i$ having Laplace \eqref{probdensities:imp} and Gaussian \eqref{probdensities:gaussian} probability densities, respectively. In contrast to the previous Section \ref{sec:mapderivation} where we have computed a joint MAP estimate for $u$ and $v$, here we want to compute the MAP estimate for the modified, infinity convolution likelihood \eqref{probdensZ:convordinfexp}, optimising over $u$ only.  We will see that doing so we derive the same infimal-convolution models \eqref{TV_IC_additive2} and \eqref{infconvpoissgauss} as before. By independence of the realisations, we have that the probability density $f_Y$ of the random vector $Y=(Y_1,\ldots,Y_M)$ in correspondence with the realisation $y=(y_1,\ldots,y_M)$ reads
\begin{equation}  \label{productdensity}
f_Y(y)=\prod_{i=1}^M f_{Y_i}(y_i).
\end{equation}
Each probability density $f_{Y_i}(y_i)$ is formally given by the convolution between the Laplace and Gaussian probability distribution. However, if we replace this convolution with the infinity convolution defined in \eqref{probdensZ:convordinf} and we normalise appropriately in order to get a probability distribution, we get that the probability density of a single $Y_i$ evaluated in correspondence of one realisation $y_i$ can be expressed as
\begin{align}  
& f_{Y_i}(y_i)=\frac{1}{2\tau\sqrt{2\pi\sigma^2}}e^{-\inf\limits_{v_i\in\R} ~\left( |v_i|/\tau~+~|y_i-v_i|^2/2\sigma^2\right)},\qquad &i=1,\ldots,M. \notag\\
& \tilde{f}_{Y_i}(y_i):=\frac{1}{\sum_{i=1}^M f_{Y_i}(y_i)} f_{Y_i}(y_i),\qquad& i=1,\ldots,M. \notag
\end{align}
Plugging this expression in \eqref{productdensity} and computing the negative log-likelihood of $P(y|u)$, we get:
\begin{equation*}
 -\log P(y|u)= -\log \prod_{i=1}^M \tilde{f}_{Y_i}(y_i) = -\sum_{i=1}^M \log (\tilde{f}_{Y_i}(y_i))=-\sum_{i=1}^M \log ( \tilde{f}_{Y_i}(f_i-u_i)).
 \end{equation*}
Thus, we have that the log-likelihood we intend to minimise is
\begin{equation*}
\sum_{i=1}^M~ \inf_{v_i\in\R} ~\left( \frac{|v_i|}{\tau}+\frac{|f_i-u_i-v_i|^2}{2\sigma^2}\right),
\end{equation*} 
where the constant terms which do not affect the minimisation over $u$ have been neglected. 

Similar to before, passing from a discrete to a continuous representation we get
\begin{multline}   \label{commuteinfint}
\int_\Omega \inf_{v(x)\in\R}\left( \frac{|v(x)|}{\tau}+\frac{|f(x)-u(x)-v(x)|^2}{2\sigma^2}\right) ~ d\mu(x) \\ = \inf_{v \in L^(\Omega)} \int_\Omega \left( \frac{|v(x)|}{\tau}+\frac{|f(x)-u(x)-v(x)|^2}{2\sigma^2}\right) ~ d\mu(x),\qquad\qquad
\end{multline}
where the infimum and the integral operators commute by assuming that the function space for $v$ is $L^2(\Omega)$ as stated in \cite[Theorem 3A]{rockafell76} (see also \cite{HafsaInfInt2003} for further details). This also ensures that the integrand terms are both well defined. Similarly as above, we now define  $\lambda_1:=1/\tau$ and $\lambda_2:=1/\sigma^2$ and derive from \eqref{commuteinfint} the same following data fidelity term as in \eqref{TV_IC_additive2}.

The same computation can be done for the signal dependent case of the Poisson-Gaussian noise mixture \eqref{poisson_gauss_strcture} with \eqref{poissondistribution}, thus obtaining the same model \eqref{infconvpoissgauss}.

\subsection{Connection with existing approaches}  \label{subsect:connectexistappr}

Several variational models for image denoising in the presence of combined noise distributions have been considered in the literature. 

In the case of a mixture of impulsive and Gaussian noise these models can be roughly divided in two categories. The first category considers additive, $L^1$+$L^2$ data fidelities. In \cite{HintermuellerLanger2013}, for instance, the image domain is decomposed in two parts, with impulsive noise in one and Gaussian noise in the other, modelled by the sum of an $L^1$ and $L^2$ data fidelity supported on the respective domain with Gaussian or impulse noise. For the numerical solution an efficient domain decomposition approach is used. In \cite{noiselearning,lucasampling,bilevellearning} semi-smooth Newton's methods are employed to solve a denoising model where $L^1$ and $L^2$ data fidelities are combined in an additive fashion. A second category of methods renders the removal of impulsive and Gaussian noise in a two-step procedure (see, e.g. \cite{impulsegauss2008}). In the first phase the pixels damaged by the impulsive noise component are identified via an outlier-removal procedure and removed making use of a variational regularisation model with $L^1$ data fitting. Then a variational denoising model with an $L^2$ data fidelity is employed for removing the Gaussian noise in all other pixels. These methods are often presented as image inpainting strategies where the pixels corrupted by impulsive noise are first identified and then filled in using the information nearby with a variational regularisation approach. 

When a mixture of Gaussian and Poisson noise is assumed, an exact log-likelihood model has been considered in \cite{Jezierska2012,poissongauss2013}. In the same discrete setting described above, still denoting by $M$ the total number of pixels, the expression of the negative log-likelihood reads
\begin{equation}\label{eq:exactGaussPoiss}
\Phi_{PG}(u,f)=\sum_{i=1}^M \Bigl( -\log \sum_{\pi=0}^\infty \frac{{u_i}^{\pi} e^{-u_i}}{\pi!} \frac{e^{-\Bigl ( \frac{f_i-\pi}{\sqrt{2}\sigma} \Bigr)^2}}{\sqrt{2\pi}\sigma}\Bigr).
\end{equation}
In comparison with our derivation presented in Section \ref{sec:infinityconv}, here the log-likelihood $\Phi_{PG}(u,f)$ is derived via MAP estimation for the logarithm of the convolution \eqref{probdensZ:classconv} of the Gaussian probability density \eqref{probdensities:gaussian} with the Poisson probability density \eqref{poissondistribution}. Instead, the model we propose is derived equivalently either by replacing the convolution with the infinity convolution or by a joint MAP estimation. The resulting variational model \eqref{infconvpoissgauss} is much simpler since, in particular, it makes easier dealing with the infinite sum appearing in  \eqref{eq:exactGaussPoiss} which is related to the discrete Poisson density function. In order to design an efficient optimisation strategy, the authors first split the expression above into the sum of two different terms, the former being a convex Lipschitz-differentiable function, the latter being a proper, convex and lower semi-continuous function. In particular, the authors are able to design a competitive first-order optimisation method based on the use of primal-dual splitting algorithms exploiting the closed-form of the proximal operators associated to the convex Lipschitz-differentiable function. The advantage of such algorithms is that it is only based on the iterative application of the proximal mapping operations, without requiring any matrix inversion. Furthermore, the approximation error accumulating throughout the iterations of the algorithm are shown to be absolutely summable sequences, a property which is essential in this framework due to the presence of infinite sums.

Another approach for mixed Gaussian and Poisson noise is considered in \cite{lanza2013}, where the authors design a data fidelity term similar to \eqref{infconvpoissgauss} combining it with TV regularisation for image denoising. Despite the analogies between the joint MAP estimation of their and our model, in their work no well-posedness results in fuction species nor properties of noise decompositions are discussed. Nonetheless, in \cite{lanza2013} the good performance of the combined model is observed in terms of improvements of the Peak Signal to Noise Ratio (PSNR) for several synthetic and microscopy images.

Differently from most aforementioned works, the simple structure of our models  \eqref{TV_IC_additive2}  and \eqref{infconvpoissgauss} allows for the design of efficient first and second-order numerical schemes. In this paper we focus on the latter. Our approach is further able to `decompose' the noise into its different statistical components, each corresponding to one particular noise distribution in the data. The authors are not aware of any existing method dealing with such a feature.

\section{Well-posedness of the TV-IC model}  \label{subsec:wellposednessBV}

Thanks to Proposition \ref{wellposednimpgauss} and Proposition \ref{wellposednessgausspoiss}, we can conclude that in both cases \eqref{impgaussinfconvfidelity} and \eqref{gausspoissinfconvfidelity}, the function $\Phi^{\lambda_1,\lambda_2}$ is well-defined and that for every $u\in BV(\Omega)\cap\mathcal{A}$ there exists a unique element $v^*(u)\in L^2(\Omega)\cap\mathcal{B}$ minimising the functional $ \mathscr{F}^{\lambda_1,\lambda_2}(f,u,\cdot)$. Problem \eqref{infconvproblemu} can then be rewritten as
\begin{equation}  \label{minimisprobuoptv2}
\min_{u\in BV\Omega)\cap\mathcal{A}} \left\{ J(u):=|Du|(\Omega) +  \lambda_1~ \Phi_1(v^*(u)) + \lambda_2 ~\Phi_2(u,f-v^*(u))  \right\},
\end{equation}
where $v^*(u)\in L^2(\Omega)\cap\mathcal{B}$ is the unique solution of \eqref{infconvproblem2} in one of the two cases \eqref{impgaussinfconvfidelity} and \eqref{gausspoissinfconvfidelity}. In particular, for every $u\in BV(\Omega)\cap\mathcal{A}$, there is a positive finite constant $C=C(u)$ such that
\begin{equation}  \label{estimatevstar}
\| v^*(u) \|_{L^2(\Omega)} \leq C(u).
\end{equation}
Note that the constant in \eqref{estimatevstar} may depend on $u$ and hence does not necessarily bound $v$ uniformly. For the following existence proof, therefore, we restrict the admissible set of solutions for $v$ by intersecting it with the closed ball in $L^2(\Omega)$ with fixed radius $R>0$, centred in $f$. That is, in \eqref{infconvproblem2} we consider a new admissible set
\begin{equation}   \label{admissible_sets_ball}
\hat{\mathcal{B}}:= L^2(\Omega)\cap \mathcal{B} \cap B_R(f),
\end{equation}
where $B_R(f):=\left\{ z\in L^2(\Omega): \|z-f\|_{L^2(\Omega)} \leq R \right\}$ for some $R>0$ and $\mathcal{B}$ is defined as before in each case. Since $B_R(f)$ is compact and convex, the well-posedness properties studied in Section \ref{sec:model} still hold true. In addition, one can now easily compute the constant $C$ of \eqref{estimatevstar} by Young's inequality. Namely, for every $u\in BV(\Omega)\cap \mathcal{A}$ one has now
\begin{equation}  \label{estimatevar_unif}
\| v^*(u) \|_{L^2(\Omega)}\leq C,
\end{equation}
with $C:=\sqrt{2(R^2+\|f\|_{L^2(\Omega)}^2)}$. This additional assumption is reasonable for our applications since we want the noise component $v$ to preserve some similarities with the given image $f$ in terms of its noise features.


After this modification, we can now state and prove the main well-posedness result or the salt \& pepper-Gaussian noise combination described in Section \ref{subsec:gaussimp} and the Gaussian-Poisson model in Section \ref{subsec:gausspoiss} by means of standard tools of calculus of variations. We refer the reader also to  \cite{vese2001study} and \cite{Bru10} where similar results are proved for $L^1$ and Kullback-Leibler-type data fidelities, respectively.

\begin{theorem}  \label{wellposednessBV}
Let $\lambda_1,\lambda_2>0$ and let us denote with $v^*(u)\in \hat{\mathcal{B}}$ defined in \eqref{admissible_sets_ball} the solution of the minimisation problem \eqref{infconvproblem2} in one of the two cases \eqref{impgaussinfconvfidelity} and \eqref{gausspoissinfconvfidelity} provided by Propositions \ref{wellposednimpgauss} and \ref{wellposednessgausspoiss}, respectively, for every $u\in BV(\Omega)\cap\mathcal{A}$. Then, there exists a unique solution  of the minimisation problem \eqref{minimisprobuoptv2}.
\end{theorem}

\begin{proof}
Let $\left\{ u_n \right\}\subset BV(\Omega)\cap \mathcal{A}$ be a minimising sequence for $J$. Such sequence exists since the functional $J$ is nonnegative. Neglecting the positive contribution given by $\Phi_1$, we have:
\begin{equation}    \label{coercivity1}
| Du_n |(\Omega)+\lambda_2 ~\Phi_2(u_n,f-v^*(u_n))\leq |Du_n|(\Omega) + \Phi^{\lambda_1,\lambda_2}(u_n,f) \leq M,\quad\text{ for all }n
\end{equation}
for some finite constant $M>0$. To show the uniform boundedness of the sequence $\left\{ u_n \right\}$ in $BV(\Omega)$ we first observe that using the positivity of $\Phi_2$, \eqref{coercivity1} we have
\begin{equation}  \label{TV_bound}
| Du_n |(\Omega)\leq M,\quad\text{ for all }n.
\end{equation}
Next, in order to get appropriate bounds for $\{u_n\}$ in $L^1(\Omega)$ we need to differentiate the two cases considered.
\begin{itemize}
\item[-] \underline{\smash{Gaussian-salt \& pepper case}}: For $\Phi_2(u_n,f-v^*(u_n))=\frac{1}{2}\| f-v^*(u_n)-u_n\|^2_{L^2(\Omega)}$, we get from \eqref{coercivity1} that by Young's inequality:
\begin{equation*} 
M\geq \Phi_2(u_n,f-v^*(u_n))=\frac{1}{2}\| f-u_n-v^*(u_n)\|_{L^2(\Omega)}^2\geq C_1\| u_n \|_{L^1(\Omega)}^2 -C_2,
\end{equation*}
where $C_1:=\frac{1}{4|\Omega|}$ and $C_2:=\frac{1}{2}R^2$ is finite by the uniform bound \eqref{admissible_sets_ball}. Hence, $\left\{ u_n \right\}$ is bounded in $L^1(\Omega)$ in this case, which, combined with \eqref{TV_bound} gives us that the sequence $\left\{ u_n \right\}$ is bounded in $BV(\Omega)$.
\item[-] \underline{\smash{Gaussian-Poisson case}}: If $\Phi_2(u_n,f-v^*(u_n))=D_{KL}(f-v^*(u_n),u_n)$, then by using the uniform bound  \eqref{admissible_sets_ball} on $g:=f-v^*(u_n)$ we can apply directly the $BV$-coercivity result for the standard TV-KL functional proved in \cite[Lemma 6.3.2]{Saw11} to \eqref{coercivity1}:
\begin{equation*}
C\| u_n \|_{BV(\Omega)} \leq | Du_n |(\Omega)+\lambda_2 ~D_{KL}(g,u_n) \leq J(u_n) \leq M,\quad\text{for every }n
\end{equation*}
to conclude that also in this case $\left\{ u_n \right\}$ is bounded in $BV(\Omega)$.
\end{itemize}

\medskip

Thanks to the uniform boundedness of the sequence $\left\{ u_n \right\}$ in $BV(\Omega)$ and since $BV(\Omega)$ is embedded in $L^1(\Omega)$ with compact embedding, we have that, up to a non-relabelled subsequence, there is $u\in BV(\Omega)$ such that:
\begin{align*}
 u_n \stackrel{*}{\rightharpoonup} u  \text{ in }BV(\Omega),  \qquad u_n\to u \text{ in }L^1(\Omega).
\end{align*}  
Moreover, since the image domain $\Omega$ is bounded, a further non-relabelled subsequence $\left\{u_n\right\}$ converging pointwise to $u$ a.e. in $\Omega$ can be extracted. We claim now that $J$ is weakly lower semicontinuous in $L^1(\Omega)$. To show that, we combine the lower semicontinuity property of $|D u|(\Omega)$
 with respect to the strong topology $L^1$ with continuity properties for the data discrepancies \eqref{impgaussinfconvfidelity} and \eqref{gausspoissinfconvfidelity}. 
\begin{itemize}
\item[-] \underline{\smash{Gaussian-salt \& pepper case}}: Recalling Remark \ref{remark:prox_map} and \cite[Proposition 12.27]{bauschkecombettes}, we have that the data fidelity $\Phi^{\lambda_1,\lambda_2}(u,f)$ is in this case a proximal mapping and, as such, is firmly nonexpansive, thus nonexpansive, i.e. Lipschitz continuous with Lipshitz constant equal to one.
\item[-] \underline{\smash{Gaussian-Poisson case}}: Recalling Proposition \ref{propos:KLproperties} in the Appendix \ref{append:KLfunct}, we observe that the functional $\Phi_2(u_n,f-v^*(u_n))=D_{KL}(f-v^*(u_n),u_n)$ is weakly lower semi-continuous in $L^1(\Omega)$ in both arguments. In fact, by assumptions on the admissible sets $\mathcal{A}$ and $\hat{\mathcal{B}}$ (see \eqref{admissiblesetsgausspoiss} and \eqref{admissible_sets_ball}), we have that for every $n$ the first argument $f-v_n$ is nonnegative and integrable and the second argument $u_n$ is nonnegative and bounded in $L^1$ from what shown above. Let us now define for every element $u_n\in BV(\Omega)\cap\mathcal{A}$ the corresponding unique solution $v_n$ of the minmisation problem \eqref{gausspoissinfconvfidelity} provided by Proposition \ref{wellposednessgausspoiss}, so let $v_n:=v^*(u_n)$ for every $n$. By the uniform estimate \eqref{estimatevar_unif}, we have 
$$
\| v_n \|_{L^2(\Omega)} \leq C\quad\text{for every } n.
$$
We can then extract a non-relabelled subsequence and find an element $\hat{v}\in L^2(\Omega)$ such that $v_n\to \hat{v}$  weakly in $L^2(\Omega)$, whence weakly in $L^1(\Omega)$. We can then write:
\begin{multline}
\Phi^{\lambda_1,\lambda_2}(u,f)= \min_{v\in\hat{\mathcal{B}}} ~ \frac{\lambda_1}{2}\|v\|_{L^2(\Omega)}^2 + \lambda_2~D_{KL}(f-v,u)\\ \leq  \frac{\lambda_1}{2}\|\hat{v}\|_{L^2(\Omega)}^2 + \lambda_2~D_{KL}(f-\hat{v},u) \leq
\liminf_{n\to\infty}~\frac{\lambda_1}{2}\|v_n\|_{L^2(\Omega)}^2 + \liminf_{n\to\infty}~\lambda_2~D_{KL}(f-v_n,u_n)\\
\leq\liminf_{n\to\infty}~ \frac{\lambda_1}{2}\|v_n\|_{L^2(\Omega)}^2+\lambda_2 ~D_{KL}(f-v_n,u_n) =\liminf_{n\to\infty}~\Phi^{\lambda_1,\lambda_2}(u_n,f),  \notag
\end{multline}
which is the required weak lower semicontinuity property. 
\end{itemize}  

Furthermore, in both cases the functional $J$ is strictly convex in both cases because of the the presence of the quadratic $L^2$ term. These properties, guarantee existence and uniqueness of the minimiser $u$ in $BV(\Omega)$. Moreover, $u$ is indeed an element of the admissible set $\mathcal{A}$ as a consequence of the fact that $\mathcal{A}$ is a closed and convex subspace of $BV(\Omega)$ in both cases considered and, as such, weakly closed by Mazur's  lemma. 
\end{proof}

\begin{remark}   \label{unif_bound_l1l2}
In the case of $L^1$-$L^2$ TV-IC model \eqref{impgaussinfconvfidelity},  we can alternatively show well-posedness using a similar argument as in \cite[Lemma 2.4, Proposition 3.2]{VaggelisIC2015} by considering the IC $1$-\emph{homogeneous} functional $\Phi^{\lambda_1,\lambda_2}_{1-hom}$ defined as:
$$
\Phi^{\lambda_1,\lambda_2}_{1-hom}(u,f):= \min_{v\in L^2(\Omega)}~\lambda_1 \| v \|_{L^1(\Omega)} + \lambda_2 \| f- v- u\|_{L^2(\Omega)}.
$$
By standard lower semi-continuity arguments, one can show that the minimum above is attained.  By considering the analogous TV-type problem:
\begin{equation}  \label{TV_IC_1hom}
\min_{u\in BV(\Omega)\cap\mathcal{A}}~\left\{ |Du|(\Omega) + \Phi^{\lambda_1,\lambda_2}_{1-hom}(u,f)\right\},
\end{equation}
one can further show that for every minimising sequence $\left\{u_n\right\}\subset BV(\Omega)\cap\mathcal{A}$, an estimate similar to \eqref{coercivity1} implies that for every $n$ and every $w\in L^2(\Omega)$:
\begin{multline}
\|u_n\|_{L^1(\Omega}-\|f\|_{L^1(\Omega)} \leq \|u_n-f\|_{L^1(\Omega)} \leq \|u_n-f-w\|_{L^1(\Omega)} + \|w\|_{L^1(\Omega)} \notag \\ \leq \|u_n-f-w\|_{L^1(\Omega)} + |\Omega|^{1/2}\| w \|_{L^2(\Omega)} \leq C_1\left(\|u_n-f-w\|_{L^1(\Omega)} +\| w \|_{L^2(\Omega)}\right). \notag
\end{multline}
where $C_1:=\max\left\{\Omega|^{1/2},1\right\}$. Hence, for an appropriate choice of $C_2$ we can get
$$
C_2\|u_n\|_{L^1(\Omega)} \leq  \lambda_1\|u_n-f-w\|_{L^1(\Omega)} +\lambda_2\| w \|_{L^2(\Omega)} \leq M + C_2\|f\|_{L^1(\Omega)}<\infty\quad\text{for every }n,
$$
for every $w\in L^2(\Omega)$ and for $M>0$. In this case, we do not need to restrict the set $\hat{\mathcal{B}}$ as in \eqref{admissible_sets_ball}, since the uniform bound of $\left\{u_n\right\}$ in $L^1(\Omega)$ is uniform by standard $L^p$ inclusions. To conclude, one can then apply \cite[Proposition 3.2]{VaggelisIC2015} which ensures that the $1$-\emph{homogeneous} problem \eqref{TV_IC_1hom} and its $2$-homogeneous version with data fidelity \eqref{impgaussinfconvfidelity} have in fact the same minimisers.
We remark here that since the same argument does not apply to the case of Kullback-Leibler functional, the restriction to the set $\hat{\mathcal{B}}$ \eqref{admissible_sets_ball} is needed to prove the theorem \ref{wellposednessBV} for the TV-IC $L^2$-KL case \eqref{gausspoissinfconvfidelity}.
\end{remark}

\subsection{Well-posedness of Huber-regularised TV-IC}  \label{subsec:wellposedness}

In view of the numerical realisation of the TV-IC via a gradient-based method presented in Section \ref{sec:numres} we smooth the TV energy to avoid the multivaluedness of its subdifferential and, consequently, to be able to deal with unique gradients. In particular, we consider a standard Huber-type smoothing of TV  depending on a parameter $\gamma\gg 1$.
For a general function $z:\Omega\to \R^\ell$, the Huber regularisation of $|z|$ is defined as:
\begin{equation}   \label{huberregular}
|z|_\gamma := \begin{cases}
|z| - \frac{1}{2\gamma},\quad&\text{if }|z|\geq \frac{1}{\gamma} \\
\frac{\gamma}{2}|z|^2,\quad&\text{if } |z|<\frac{1}{\gamma}
\end{cases}.
\end{equation}
In words, in the proximity of the points where $|z|$ is small a quadratic smoothing is used, while essentially the function  is kept the same for large values of $|z|$. Clearly, the non-smooth version is recovered letting $\gamma\to\infty$. Denoting by $\mathcal{L}^2$ the usual Lebesgue measure in $\R^2$ and by $\mathcal{B}(\Omega)$ the $\sigma$-algebra of $\Omega$ , let $Du=\nabla u~\mathcal{L}^2+ D_s u$ be the Lebesgue decomposition of the two-dimensional distributional gradient $Du$ into its absolutely continuous $\nabla u~\mathcal{L}^2$ and  singular $D_s u$ parts. Then, the Huber-regularisation of the total variation measure $|Du|$ is defined as
\begin{equation}   \label{huberregularTV}
| Du |_\gamma (V) := \int_V |\nabla u|_\gamma~dx + \int_V |D_s u|,\quad V\in\mathcal{B}(\Omega),
\end{equation}
i.e. we regularise the absolutely continuous part using \eqref{huberregular}. We then consider the following Huber-regularised version of \eqref{minimisprobuoptv2}:
\begin{equation}  \label{minimisprobuoptvhuber}
\min_{u\in BV(\Omega)\cap\mathcal{A}} \left\{ J_\gamma(u):=|Du|_\gamma(\Omega) + \lambda_1\Phi_1(v^*(u))+\lambda_2 ~\Phi_2(u,f-v^*(u))  \right\},
\end{equation}
where, as before, for every $u\in L^2(\Omega)\cap\mathcal{B}$ the element $v^*(u)\in \hat{\mathcal{B}}$ is the unique solution of \eqref{infconvproblem2} and $\hat{\mathcal{B}}$ is defined as in \eqref{admissible_sets_ball}. As a Corollary of Theorem \ref{wellposednessBV} we show now that also the regularised TV-IC problem is well-posed.

\begin{corollary}
Let $\lambda_1,\lambda_2>0$ and let  $v^*(u)\in \hat{\mathcal{B}}$ be the solution of the minimisation problem \eqref{infconvproblem2} for every $u\in BV(\Omega)\cap\mathcal{A}$. Then, there exists a unique solution $u\in BV(\Omega)\cap\mathcal{A}$ of the Huber-regularised minimisation problem \eqref{minimisprobuoptvhuber}.
\end{corollary}

\begin{proof}
The Huber regularisation function is coercive and has at most linear growth (cf. \cite[Theorem 2.1]{noiselearning}). Forgetting the contribution coming from the positive term $\Phi_1$, we have:
\begin{multline}    
| Du_n |(\Omega)+\lambda_2 ~\Phi_2(u_n,f-v^*(u_n))  \notag \\ \leq  |D u_n|_\gamma(\Omega)+\lambda_1\Phi_1(v^*(u_n))+\lambda_2 ~\Phi_2(u_n,f-v^*(u_n)) \leq M,\quad\text{ for all }n. \notag
\end{multline}
for every minimising sequence $\left\{u_n\right\}$ in $BV(\Omega)\cap \mathcal{A}$.  To conclude, we can simply use the result proved in Theorem \ref{wellposednessBV} for the non-smooth case combined with the lower-semicontinuity property of $|Du|_\gamma(\Omega)$ with respect to the strong topology of $L^1(\Omega)$ (see, e.g., \cite{demengeltemam}).
\end{proof}

We now connect the solution of the regularised minimisation problem \eqref{minimisprobuoptvhuber} to a solution of the non-smooth problem \eqref{minimisprobuoptv2} via $\Gamma$-convergence \cite{dalmasoGamma}.

\begin{theorem}   \label{gammaconvergenceTV}
The sequence of functionals $J_\gamma:BV(\Omega)\times\mathcal{A}\to\R$ defined in \eqref{minimisprobuoptvhuber} $\Gamma$-converges to the functional $J:BV(\Omega)\times\mathcal{A}\to\R$ defined in \eqref{minimisprobuoptv2} as $\gamma\to \infty$. Hence, the minimiser of $J_\gamma$ converges to a minimiser of $J$  as $\gamma\to\infty$ in $BV(\Omega)\cap\mathcal{A}$.
\end{theorem}

\begin{proof}
The proof is a standard result based on relaxation techniques for measures, see e.g. \cite{buttazzorelaxation1,buttazzorelaxation2,buttazzorelaxation3}. We observe that as $\gamma\to \infty$ the functional $J_\gamma$ converges pointwise, for every $u\in BV(\Omega)\cap\mathcal{A}$, to
\[ 
J(u) = \int_\Omega |\nabla u|~dx + \int_\Omega |D_s u|~dx+ \lambda_1\Phi_1(v^*(u))+ \lambda_2~\Phi_2(u,f-v^*(u))
\]
since for the absolutely continuous part $\nabla u$ there holds $|\nabla u|_\gamma\to |\nabla u|$ pointwise as $\gamma\to\infty$. Since the convergence is monotonically increasing, we have that $J_\gamma$ $\Gamma$-converges to $J$ (\cite[Remark 5.5]{dalmasoGamma}). 
\end{proof}

Therefore, thanks to the results above, the general infimal convolution model \eqref{infconvproblemu}-\eqref{infconvproblem2} is well-posed in both cases described in Section \ref{subsec:gaussimp} and \ref{subsec:gausspoiss}. For simplicity,  we will focus in the following on an equivalent formulation of the nonsmooth problem \eqref{minimisprobuoptv2} which reads:
\begin{equation}   \label{jointminimisation}
\min_{\substack{u\in BV(\Omega)\cap\mathcal{A}\\
v\in L^2(\Omega)\cap\mathcal{B}}} \left\{ J(u,v):=|Du|(\Omega) + \lambda_1~\Phi_1(v) + \lambda_2~ \Phi_2(u,f-v) \right\},
\end{equation}
where the two components of the model, i.e. the reconstructed image $u$ and the noise component $v$ are treated jointly. Analogously, in Section \ref{sec:numres} we will consider the corresponding Huber-regularised version
\begin{equation*}
\min_{\substack{u\in BV(\Omega)\cap\mathcal{A}\\
v\in L^2(\Omega)\cap\mathcal{B}}} \left\{ J_\gamma(u,v):=|Du|_\gamma(\Omega) + \lambda_1~\Phi_1(v) + \lambda_2~ \Phi_2(u,f-v) \right\},
\end{equation*}
and use it for the design of efficient gradient-based numerical schemes.

\section{Recovery of the single noise models} \label{sec:asympt}

The TV-IC model \eqref{infconvproblemu}-\eqref{infconvproblem2} (or equivalently \eqref{jointminimisation}) combines data fidelities classically used in the literature for single-noise models to deal with the combined case. The reader may ask whether and under which conditions the single noise models can be recovered from our model. In this section we show that this is possible by looking at the behaviour of solutions $(u^*, v^*)$ of \eqref{jointminimisation} as the parameters $\lambda_1$ and $\lambda_2$ become infinitely large. Similarly as we have done so far, in the following analysis we discuss the Gaussian- salt \& pepper (see Section \ref{subsec:gaussimp}) and the Gaussian-Poisson case (see Section \ref{subsec:gausspoiss}) separately for more clarity.

\subsection{The Gaussian-salt \& pepper case}   \label{subsec:asymptgaussimp}

The following proposition asserts essentially that TV-$L^2$ \cite{rudinosherfatemi} and TV-$L^1$ \cite{nikolovaoutremov,duvaltvL1} -type models can be recovered `asymptotically' from \eqref{jointminimisation} by letting the $L^1$/$L^2$ fidelity weight become infinitely large, respectively. Moreover, when both parameters become infinitely large, a full recovery of the data is obtained. The proof is based on standard energy estimates for the minimisation problem \eqref{jointminimisation} in the case when the variational fidelity is chosen as in \eqref{impgaussinfconvfidelity}.

\begin{proposition}    \label{prop:impgaussasympt}
Let $(u^*,v^*)\in BV(\Omega)\times L^2(\Omega)$ the optimal pair for \eqref{jointminimisation} in the Gaussian-salt \& pepper case described in Section \ref{subsec:gaussimp}. If $f\in L^2(\Omega)$ is not identically zero, then the following asymptotical convergences hold:

\begin{enumerate}
\item[i)] If $\lambda_2$ is finite, then $v^*\to 0$ in $L^1(\Omega)$ as $\lambda_1\to+\infty$.
\item[ii)] If $\lambda_1$ is finite, then $v^*\to f-u^*$ in $L^2(\Omega)$ as $\lambda_2\to+\infty$.
\item[iii)] If, additionally, $f\in BV(\Omega)$ and is not a constant, then the pair $(u^*,v^*)$ converges to $(f,0)$ in $L^1(\Omega)\times L^1(\Omega)$ as $\lambda_1,\lambda_2\to+\infty$. 
\end{enumerate} 
\end{proposition}

\begin{proof}
We start writing the variational inequality satisfied by the optimal pair $(u^*,v^*)\in BV(\Omega)\times L^2(\Omega)$ in this case. This reads:

\begin{multline}   \label{impgauss:variationineq}
|Du^*|(\Omega) + \lambda_1~\| v^* \|_{L^1(\Omega)}  + \frac{\lambda_2}{2}~ \| f-u^*-v^*\|_{L^2(\Omega)}^2 \leq \\ |Du|(\Omega) + \lambda_1~\| v \|_{L^1(\Omega)}  + \frac{\lambda_2}{2}~ \| f-u-v\|_{L^2(\Omega)}^2,\quad\text{for all }u\in BV(\Omega),~ v\in L^2(\Omega).
\end{multline}
For the proof of \textit{i)} we choose in \eqref{impgauss:variationineq} $u=v=0$. We have:
\begin{equation*}
 \lambda_1~\| v^* \|_{L^1(\Omega)} \leq |Du^*|(\Omega) + \lambda_1~\| v^* \|_{L^1(\Omega)}  + \frac{\lambda_2}{2}~ \| f-u^*-v^*\|_{L^2(\Omega)}^2 \leq \frac{\lambda_2}{2}~ \| f\|_{L^2(\Omega)}^2,
\end{equation*}
which implies:
\begin{equation*}
\| v^* \|_{L^1(\Omega)} \leq \frac{\lambda_2}{2\lambda_1} \| f\|_{L^2(\Omega)}^2.
\end{equation*}
Since $\lambda_2$ is positive and finite, by letting $\lambda_1$ go to infinity we have that $v^*$ converges to $0$ in $L^1(\Omega)$.

Similarly, for \textit{ii)} we choose $u=0$ and $v=f$ in \eqref{impgauss:variationineq} and get:
\begin{equation*}
\| f-u^*-v^* \|_{L^2(\Omega)}^2  \leq \frac{2\lambda_1}{\lambda_2}~ \| f\|_{L^1(\Omega)}.
\end{equation*}
The conclusion follows analogously by letting $\lambda_2$ go to infinity.

For the proof of \textit{iii)} we choose in \eqref{impgauss:variationineq} $u=f\in BV(\Omega)$ and $v=0$. Since $f$ is not a constant, $|Df|(\Omega)\neq 0$. We have:
\begin{multline}
C_{\lambda_1,\lambda_2}\left(\| v^* \|_{L^1(\Omega)} + \| f-u^*-v^*\|_{L^2(\Omega)}^2\right) \notag \\ \leq |Du^*|(\Omega) + \lambda_1~\| v^* \|_{L^1(\Omega)}  + \frac{\lambda_2}{2}~ \| f-u^*-v^*\|_{L^2(\Omega)}^2 \leq \frac{\lambda_2}{2}~ |Df|(\Omega), \notag
\end{multline}
where $C_{\lambda_1,\lambda_2}:=\min\left\{\lambda_1,\frac{\lambda_2}{2}\right\}$. We have:
\begin{equation*}
\| v^* \|_{L^1(\Omega)} + \| f-u^*-v^*\|_{L^2(\Omega)}^2 \leq \frac{1}{C_{\lambda_1,\lambda_2}}|Df|(\Omega).
\end{equation*}
The right hand side of the inequality above goes to zero as $\lambda_1,\lambda_2$ tend to $+ \infty$ since in this case $C_{\lambda_1,\lambda_2}$ goes to infinity. Thus, as $\lambda_1,\lambda_2\to + \infty$  we have the following convergences:
\begin{align*}
&v^*\to f-u^*\quad\text{in } L^2(\Omega), \\
&v^*\to 0\quad\qquad~\text{in } L^1(\Omega).
\end{align*}
By uniqueness of the limit in $L^1(\Omega)$ we conclude that $f-u^*=0$ and, consequently \textit{iii)} holds.
\end{proof}

Using the same energy estimates it is also possible to show that similar convergence results hold whenever one of the two fidelity weights goes to zero.

\begin{corollary}  \label{corollary:convergence}
Let $(u^*,v^*)\in BV(\Omega)\times L^2(\Omega)$ the optimal pair for \eqref{jointminimisation} in the Gaussian-salt \& pepper case described in Section \ref{subsec:gaussimp}. If $f\in L^2(\Omega)$ is not identically zero, then:
\begin{enumerate}
\item If $\lambda_1$ is finite, then $|Du^*|(\Omega)\to 0$ and $v^*\to 0$ in $L^1(\Omega)$ as $\lambda_2\to 0$.
\item If $\lambda_2$ is finite, then $|Du^*|(\Omega)\to 0$ and $v^*\to f-u^*$ in $L^2(\Omega)$ as $\lambda_1\to 0$. 
\end{enumerate}
\end{corollary}

\begin{proof}
The proof is similar to the points \textit{i)} and \textit{ii)} of the Proposition \ref{prop:impgaussasympt}. Starting from the variational inequality \eqref{impgauss:variationineq} one in fact gets similar estimates as above and the conclusion simply follows by letting the parameters go to zero.
\end{proof}

\subsection{The Gaussian-Poisson case} \label{subsec:asymptpoissgauss}

In the Gaussian-Poisson framework described in Section \ref{subsec:gausspoiss} similar results can be proved. They rely on analogous energy estimates and, essentially, on the estimate \eqref{KLestimate} for the KL fidelity term \eqref{fullKL} recalled in Appendix \ref{append:KLfunct}. Analogously as before, by letting the Gaussian and Poisson weights go to infinity, solutions of the TV-KL \cite{alexTVpoisson,lechartrandTVpoisson} and the TV-$L^2$ \cite{rudinosherfatemi} model can be recovered as well as the full data $f$.

\begin{proposition}    \label{prop:gausspoissasympt}
Let $\mathcal{A}$ and $\mathcal{B}$ be the admissible sets in \eqref{admissiblesetsgausspoiss}. Let $(u^*,v^*)\in (BV(\Omega)\cap\mathcal{A}) \times (L^2(\Omega)\cap\mathcal{B})$ the optimal pair for \eqref{jointminimisation} in the Gaussian-Poisson case  described in Section \ref{subsec:gausspoiss}. If $f\in L^\infty(\Omega)$ is not identically zero, then the following asymptotical convergences hold:

\begin{enumerate}
\item[i)] If $\lambda_2$ is finite and $f$ is not identically equal to one, then $v^*\to 0$ in $L^2(\Omega)$ as $\lambda_1\to+\infty$.
\item[ii)] If $\lambda_1$ is finite, then $v^*\to f-u^*$ in $L^1(\Omega)$ as $\lambda_2\to+\infty$.
\item[iii)] If additionally $f\in BV(\Omega)$ and is not a constant, then the pair $(u^*,v^*)$ converges to $(f,0)$ in $L^1(\Omega)\times L^2(\Omega)$ as $\lambda_1,\lambda_2\to+\infty$. 
\end{enumerate} 
\end{proposition}

\begin{proof}
Again, we start writing explicitly the variational inequality satisfied by the optimal pair $(u^*,v^*)\in (BV(\Omega)\cap\mathcal{A}) \times (L^2(\Omega)\cap\mathcal{B})$ for the Gaussian-Poisson combined case. This reads:
\begin{multline}   \label{gausspoiss:variationineq}
|Du^*|(\Omega) + \frac{\lambda_1}{2}~\| v^* \|_{L^2(\Omega)}^2  +\lambda_2~ D_{KL}(f-v^*,u^*) \\ \leq |Du|(\Omega) + \frac{\lambda_1}{2}~\| v \|_{L^2(\Omega)}^2  +\lambda_2~D_{KL}(f-v,u),
\end{multline}
for all $u\in (BV(\Omega)\cap\mathcal{A})$ and $v\in (L^2(\Omega)\cap\mathcal{B})$.

For the proof of \textit{i)} we choose in \eqref{gausspoiss:variationineq} $u=1_\Omega$, the constant function identically equal to one on $\Omega$ and $v=0$. We deduce:
\begin{equation*}
\frac{\lambda_1}{2}~\| v^* \|_{L^2(\Omega)}^2  \leq \lambda_2 D_{KL}(f,1_\Omega).
\end{equation*}
Since by assumption the function $f$ is not identically equal to one, the right hand side of the inequality above is strictly positive and bounded as:
\begin{equation*}
0 < D_{KL}(f,1_\Omega)\leq K:= \Bigl( \| f \|_{L^\infty(\Omega)} \| \log f \|_{L^1(\Omega)} +|\Omega| \Bigr)< \infty,
\end{equation*}
by H\"older inequality and assumptions on $f$. Thus, we have:
\begin{equation*}
\| v^* \|_{L^2(\Omega)}^2 \leq \frac{2\lambda_2}{\lambda_1} K, 
\end{equation*}
which implies the convergence $v^*\to 0$ in $L^2(\Omega)$ as $\lambda_1\to+\infty$. 

To prove \textit{ii)}, we consider in \eqref{gausspoiss:variationineq} $u=0$ and $v=f$. For such a choice we have $D_{KL}=0$. Hence, inequality \eqref{gausspoiss:variationineq} reduces to:
\begin{equation*}
\lambda_2~D_{KL}(f-v^*,u^*)\leq \frac{\lambda_1}{2} \| f \|_{L^2(\Omega)}^2,
\end{equation*}
which implies that $D_{KL}(f-v^*,u^*)\to 0$ as $\lambda_2\to +\infty$. Thanks to Corollary \ref{corollaryKL}, we deduce that $v^*\to f-u^*$ in $L^1(\Omega)$.

We can show that \textit{iii)}  holds by choosing $v=0$ and $u=f\in BV(\Omega)$ with $|Df|(\Omega)\neq 0$ by assumption. Proceeding similarly as before, the convergence result follows immediately after applying once again Corollary \ref{corollaryKL}.
\end{proof}

Again, similar results of convergence hold in the case where the fidelity parameters go to zero individually.

\begin{corollary}  \label{corollary:convergence2}
Let $\mathcal{A}$ and $\mathcal{B}$ be the admissible sets in \eqref{admissiblesetsgausspoiss}. Let $(u^*,v^*)\in (BV(\Omega)\cap\mathcal{A}) \times (L^2(\Omega)\cap\mathcal{B})$ the optimal pair for \eqref{jointminimisation} in the Gaussian-Poisson case \ref{subsec:gausspoiss}. If $f\in L^\infty(\Omega)$ is not identically zero, then:
\begin{enumerate}
\item If $\lambda_1$ is finite and $f$ is not identically equal to one, then $|Du^*|(\Omega)\to 0$ and $v^*\to 0$ in $L^2(\Omega)$ as $\lambda_2\to 0$.
\item If $\lambda_2$ is finite, then $|Du^*|(\Omega)\to 0$ and $v^*\to f-u^*$ in $L^1(\Omega)$ as $\lambda_1\to 0$. 
\end{enumerate}
\end{corollary}

\section{Numerical results} \label{sec:numres}

In this section we report on the numerical realisation of the general mixed noise model \eqref{infconvproblemu}-\eqref{infconvproblem2} (in the form \eqref{jointminimisation}) for the two frameworks described in Sections \ref{subsec:gaussimp} and \ref{subsec:gausspoiss}. 

We consider a discretised image domain $\Omega=\left\{ (x_i,y_j): i=1,\ldots,N,~ j=1,\ldots,M\right\}$ with cardinality $|\Omega|=NM$. Standard finite difference discretisation schemes are used. In particular, forward and backward finite differences are considered for the discretisation of the divergence and gradient operators, respectively, thus preserving their mutual adjointness property, compare \cite{chambolle2004algorithm}.

For the numerical realisation of the model, we use a \textbf{S}emi\textbf{S}mooth \textbf{N}ewton (SSN) type algorithm with a primal-dual strategy. To do that, we regularise the nonsmooth TV term by Huber-regularisation using a parameter $\gamma\gg 1$, thus dealing with unique-gradients. Relations with the original problem are guaranteed by Theorem 
\ref{gammaconvergenceTV}.  
Other numerical approaches, such as first-order convex optimisation methods \cite{chambolle2016introduction} could alternatively be used. However, in this paper we chose the second-order SSN scheme in view of parameter learning via bilevel optimisation as outlined in Section \ref{sec:learning_mot}.
We use a combined stopping criterion which stops the iterations either when the norm of the difference between two different iterates is below a given tolerance or when a maximum number of iterations (typically, $35$) is attained.

Our numerical results confirm the property of the infimal-convolution model \eqref{jointminimisation} to capture the different noise components in the image. Indeed, we will see that \eqref{jointminimisation} allows to decompose the noise into its single noise components.

\medskip

\paragraph{Test images and parameters} For our computational tests we consider different images selected either from the Berkeley database\footnote{\url{https://www.eecs.berkeley.edu/Research/Projects/CS/vision/bsds/BSDS300/html/dataset/images.html}}, see Figure \ref{orig:images}, or from some other public availble website, see Figure \ref{orig:images2}. For each experiment, the ground truth image $u_0$ is artificially corrupted with mixed noise distributions of different intensities which are specified in each case.  For simplicity, we consider square $N\times N$ pixel images (corresponding to a step size $h=1/N$) and fix the Huber-regularisation parameter to be $\gamma=1e5$. 

\begin{figure}[h!]
\begin{center}
\includegraphics[height=2.5cm]{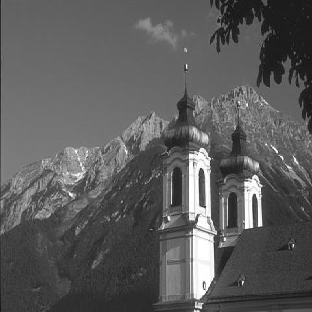}\quad 
\includegraphics[height=2.5cm]{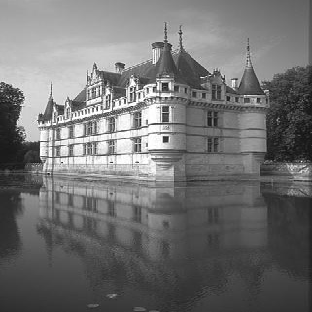}\quad 
\includegraphics[height=2.5cm]{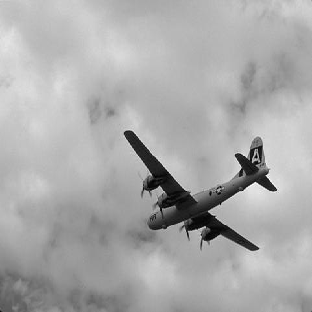}\quad 
\includegraphics[height=2.5cm]{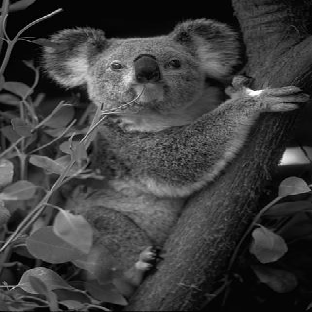}\quad
\includegraphics[height=2.5cm]{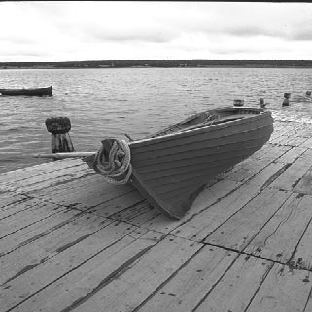}
\caption{Some images from the Berkeley database.}
\label{orig:images}
\end{center}
\end{figure}

\begin{figure}[h!]
\begin{center}
\includegraphics[height=2.5cm]{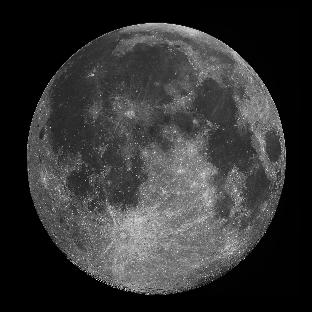}\quad 
\includegraphics[height=2.5cm]{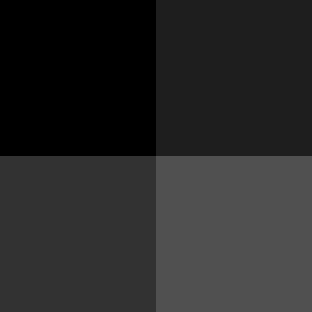}\quad 
\includegraphics[height=2.5cm]{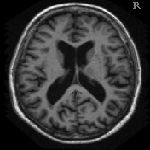}\quad
\includegraphics[height=2.5cm]{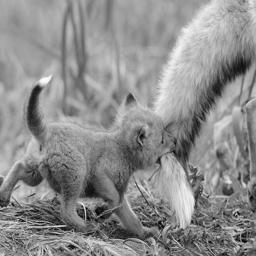}\quad
\includegraphics[height=2.5cm]{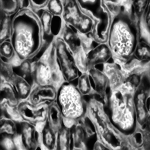}
\caption{Some additional images used for our experiments\protect\footnotemark.}
\label{orig:images2}
\end{center}
\end{figure}

\footnotetext{Moon image: \url{http://commons.wikimedia.org/wiki/File:FullMoon2010.jpg\#/media/File:FullMoon2010.jpg},\url{http://www.oasis-brains.org/}, Brain image: \url{http://meyerinst.com/confocals/tcs-spe/index.htm}}

\subsection{Gaussian-salt \& pepper case}   \label{numres:gaussimp}

We start focusing on the Gaussian-salt \& pepper model considered in Section \ref{subsec:gaussimp}. We want to compute numerically the solution pair of the following minimisation problem
\begin{equation}  \label{gaussimp:numer}
\min_{u,v} \left\{ J_\gamma(u,v):=|Du|_\gamma(\Omega) + \lambda_1~\| v \|_{\ell^1}  + \frac{\lambda_2}{2}~ \| f-u-v\|_{\ell^2}^2 \right\}.
\end{equation}
We observe that the $\ell^1$-term in \eqref{gaussimp:numer} dealing with the sparse component of the noise introduces a further nondifferentiability obstacle in the design of a numerical gradient-based optimisation method solving \eqref{gaussimp:numer}. Therefore, we Huber-regularise this term using \eqref{huberregular} as we did for the TV term and consider the following optimality conditions for the regularised problem:
\begin{equation*}
\begin{cases}
\frac{\partial J_\gamma}{\partial u}&=- \mathrm{div}\,\left(\frac{\gamma\nabla u}{\max(\gamma|\nabla u|, 1)} \right) -\lambda_2 (f-u-v)  = 0,\\
\frac{\partial J_\gamma}{\partial v}&=\lambda_1 \frac{\gamma\ v}{\max(\gamma|v|, 1)} -\lambda_2 (f-u-v) = 0.
\end{cases}
\end{equation*}
The system above can be equivalently written in primal-dual form as:
\begin{equation*}
\begin{cases}
& - \mathrm{div} q -\lambda_2 (f-u-v)  = 0,\\
& \lambda_1 p -\lambda_2 (f-u-v)  = 0, \\
& q= \left(\frac{\gamma\nabla u}{\max(\gamma|\nabla u|, 1)} \right) , \\
& p= \frac{\gamma\ v}{\max(\gamma|v|, 1)}.
\end{cases}
\end{equation*}
\noindent Starting from an appropriate initial guess for $u_0$ and $v_0$ (which for our experiments is the noisy image and a sparse vector, respectively), the SSN iteration reads:
\begin{equation*}
\left\{\begin{aligned}
&- \mathrm{div}\, \delta_q +\lambda_2 \delta_u + \lambda_2 \delta_{v} = - \left( - \mathrm{div}\, q +\lambda_2 (f-u-v) \right), \\
& \lambda_1\delta_p+\lambda_2\delta_u+\lambda_2\delta_{v} = - \left(\lambda_1p -\lambda_2 (f-u-v) \right),  \\
&\delta_q - \frac{\gamma \nabla \delta_u}{\max(1,\gamma|\nabla u|)}+ \chi_{\mathcal U_\gamma}  \gamma^2 \frac{\nabla u^T \nabla \delta_u}{\max(1,\gamma|\nabla u|)^2} \frac{q}{\max(1,|q|)} =-q+ \frac{\gamma \nabla u}{\max(1,\gamma|\nabla u|)},  \\
&\delta_p - \frac{\gamma\ \delta_{v}}{\max(1,\gamma|v|)}+ \chi_{\mathcal V_\gamma}  \gamma^2 \frac{v\ \delta_{v}}{\max(1,\gamma|v|)^2} \frac{p}{\max(1,|p|)} =-p+ \frac{\gamma\  v}{\max(1,\gamma|v|)},
\end{aligned} \right.
\end{equation*}
for the increments $\delta_u, \delta_{v}, \delta_q$ and $\delta_p$ and where the active sets $\mathcal{U}_\gamma$ and $\mathcal{V}_\gamma$ are defined as: $\mathcal{U}_\gamma:=\{ x \in \Omega: \gamma|\nabla u(x)|\geq1 \}$ and $\mathcal{V}_\gamma:=\{ x \in \Omega: \gamma|v(x)|\geq 1 \}$. As in \cite{juancarlos2012,noiselearning,lucasampling,bilevellearning}, the SSN iteration above has been modified using the properties of the solution on the final active sets $\mathcal{U}_\gamma$ and $\mathcal{V}_\gamma$. Namely, on these sets we have that $q=\frac{\nabla u}{|\nabla u|}$ and $p=\frac{v}{|v|}$ together with $q,~ p\leq 1$ a.e. in $\Omega$. The standard Newton iteration can then be modified accordingly, thus obtaining a positive definite Hessian matrix in each iteration which ensures global convergence.

\medskip

Figure \ref{result:impgauss} shows the numerical denoising results for $312\times 312$ pixel images corrupted with a combination of salt \& pepper and Gaussian noise of different intensities. The parameters $\lambda_1$ and $\lambda_2$ have been optimised experimentally with respect to the best Peak Signal to Noise Ratio (PSNR) of the denoised image $u$ in comparison with the corresponding ground truth $u_0$. In all the experiments we observe that the noise is successfully removed from the original image and it is further decomposed in its two noise components constituting salt \& pepper and Gaussian noise corresponding to the $\ell^1$ and $\ell^2$ term in \eqref{gaussimp:numer}, respectively. 
The solution pair $(u^*,v^*)$ is computed jointly, so the noise removal process can be thought of as an iterative process where in each iteration the salt \& pepper component of the noise is extracted from the noisy image and encoded in the component $v$, while the Gaussian noise component is treated with the $L^2$ fidelity of the residuum. We observe that the reconstructed images may suffer a loss of contrast resulting in image structures left in the noise components (cf. fifth column of Figure \ref{result:impgauss}). This is a well-known drawback of TV regularisation \cite{Meyer2001} and can be improved by using higher-order imaging models such as TV-TV$^2$ \cite{kostasTV2} or TGV regularisation \cite{TGV,diff_tens}, or, alternatively, it can be enhanced by solving numerically the TV problem using Bregman iteration \cite{OsherBurger2005}. 

\begin{figure}[!h]
\centering
\begin{subfigure}[b]{0.2\textwidth}
\includegraphics[height=2.9cm,width=2.9cm]{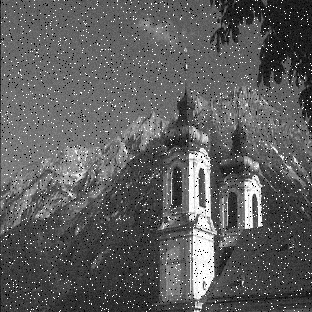}\vspace{0.1cm}
\includegraphics[height=2.9cm,width=2.9cm]{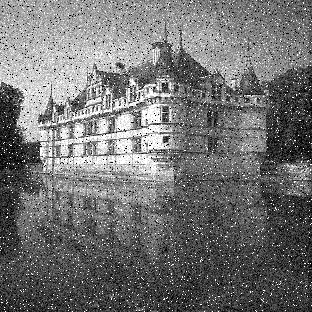}\vspace{0.1cm}
\includegraphics[height=2.9cm,width=2.9cm]{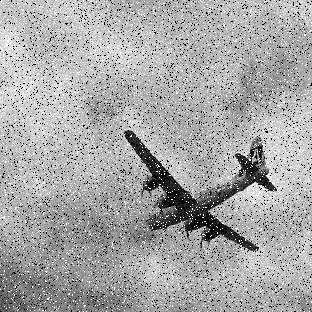}\vspace{0.1cm}
\includegraphics[height=2.9cm,width=2.9cm]{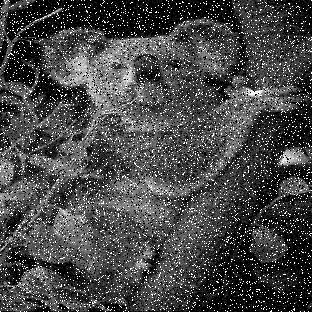}\vspace{0.1cm}
\caption{$f$}
\end{subfigure}
\hspace{-0.3cm}
\begin{subfigure}[b]{0.2\textwidth}
\includegraphics[height=2.9cm,width=2.9cm]{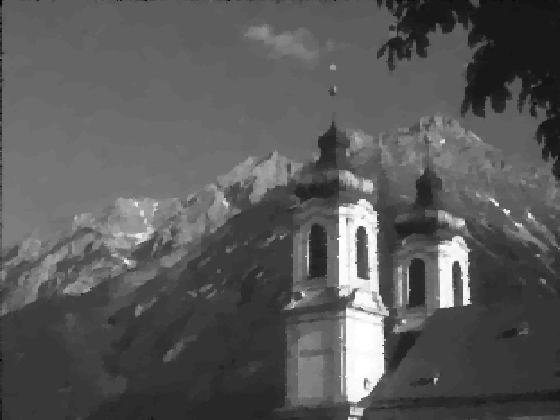}\vspace{0.1cm}
\includegraphics[height=2.9cm,width=2.9cm]{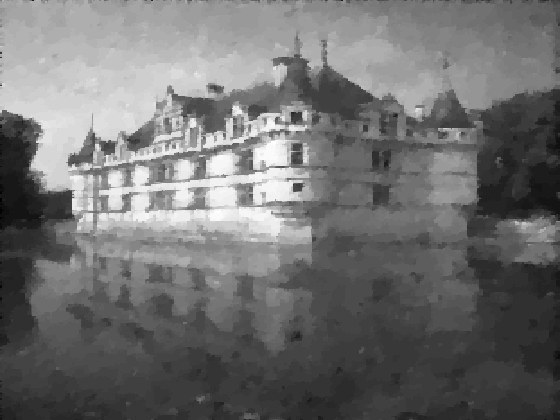}\vspace{0.1cm}
\includegraphics[height=2.9cm,width=2.9cm]{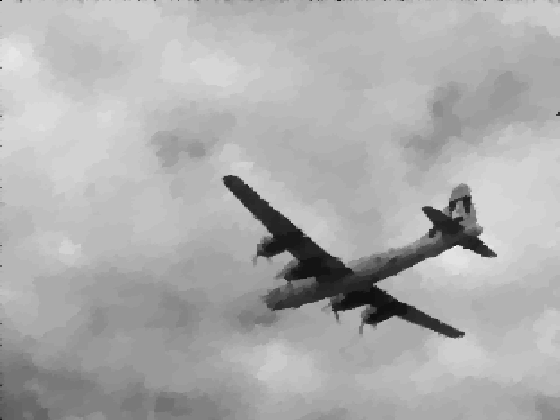}\vspace{0.1cm}
\includegraphics[height=2.9cm,width=2.9cm]{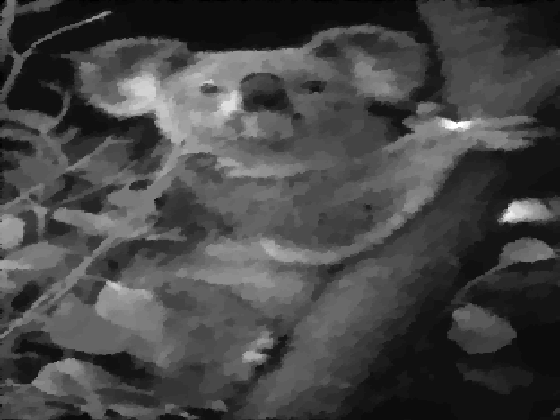}\vspace{0.1cm}
\caption{$u$}
\end{subfigure}
\hspace{-0.3cm}
\begin{subfigure}[b]{0.2\textwidth}
\includegraphics[height=2.9cm,width=2.9cm]{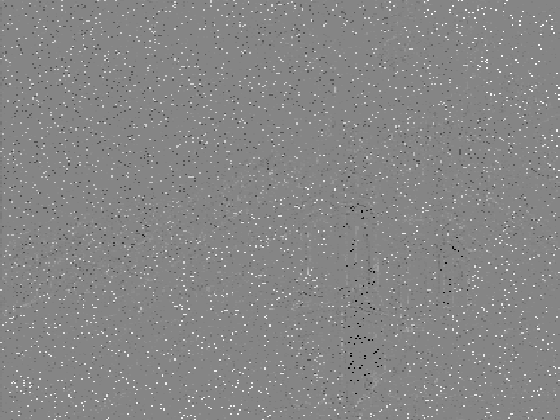}\vspace{0.1cm}
\includegraphics[height=2.9cm,width=2.9cm]{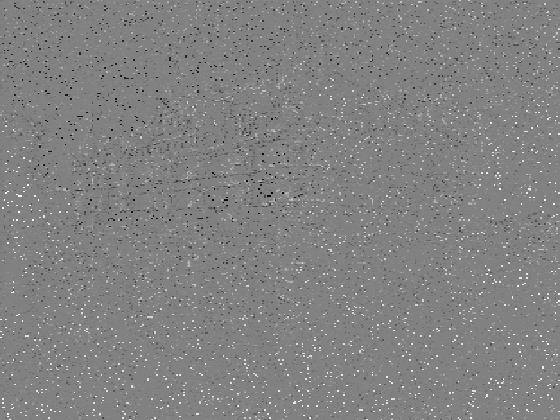}\vspace{0.1cm}
\includegraphics[height=2.9cm,width=2.9cm]{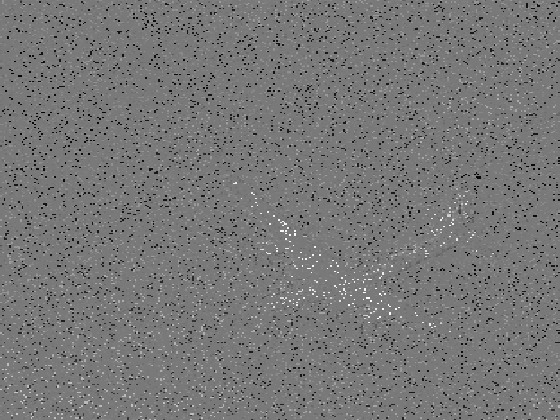}\vspace{0.1cm}
\includegraphics[height=2.9cm,width=2.9cm]{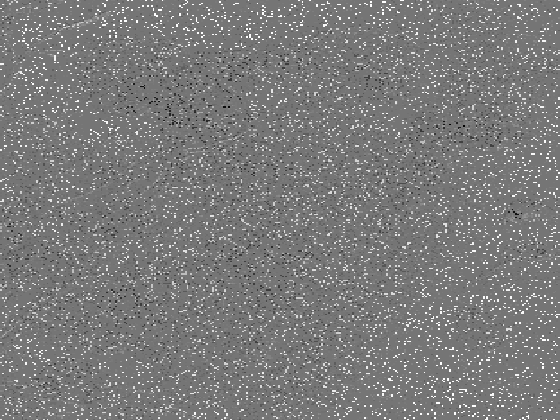}\vspace{0.1cm}
\caption{$v$}
\end{subfigure}
\hspace{-0.3cm}
\begin{subfigure}[b]{0.2\textwidth}
\includegraphics[height=2.9cm,width=2.9cm]{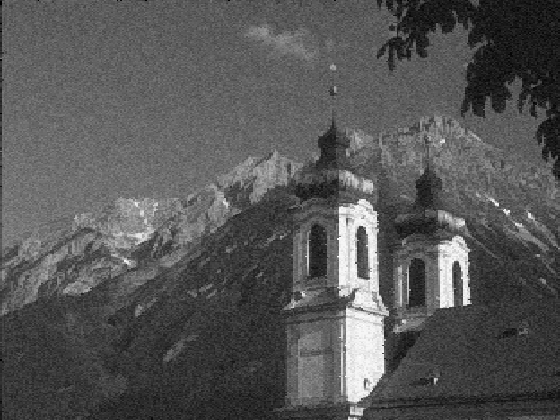}\vspace{0.1cm}
\includegraphics[height=2.9cm,width=2.9cm]{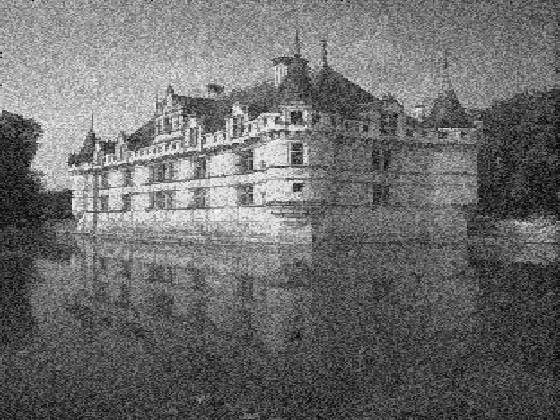}\vspace{0.1cm}
\includegraphics[height=2.9cm,width=2.9cm]{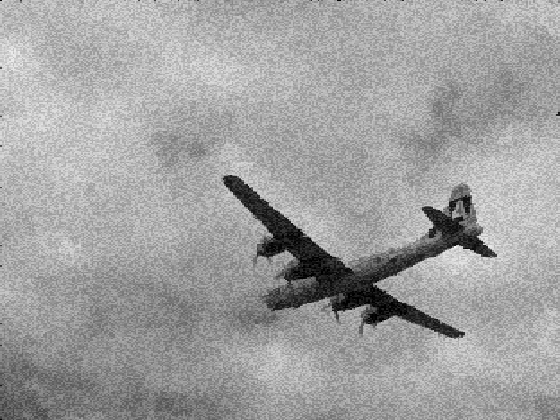}\vspace{0.1cm}
\includegraphics[height=2.9cm,width=2.9cm]{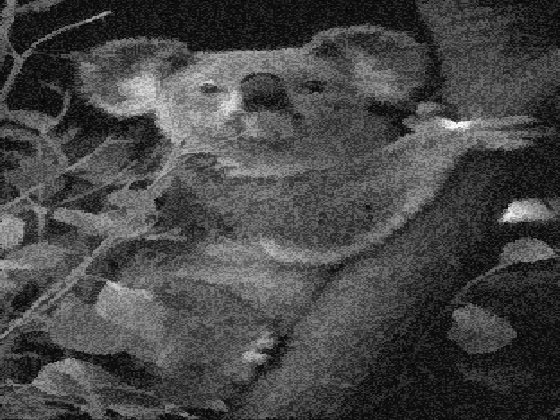}\vspace{0.1cm}
\caption{$f-v$}
\end{subfigure}
\hspace{-0.3cm}
\begin{subfigure}[b]{0.2\textwidth}
\includegraphics[height=2.9cm,width=2.9cm]{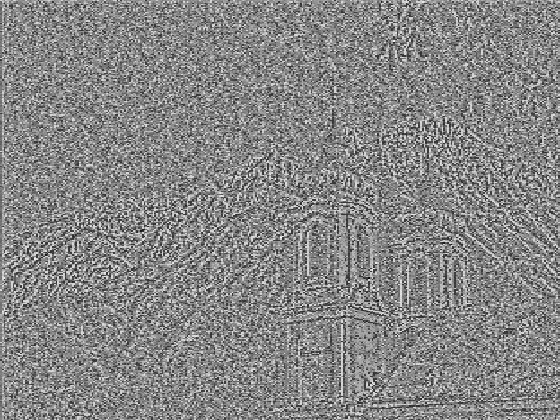}\vspace{0.1cm}
\includegraphics[height=2.9cm,width=2.9cm]{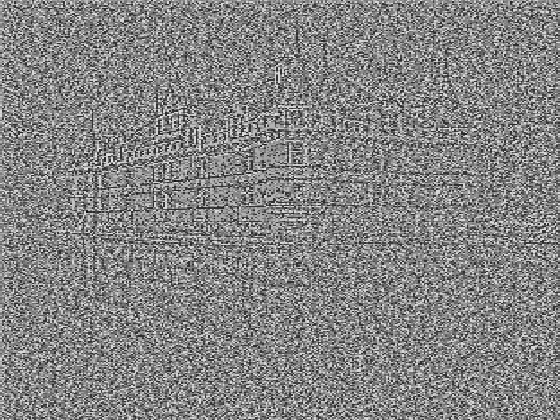}\vspace{0.1cm}
\includegraphics[height=2.9cm,width=2.9cm]{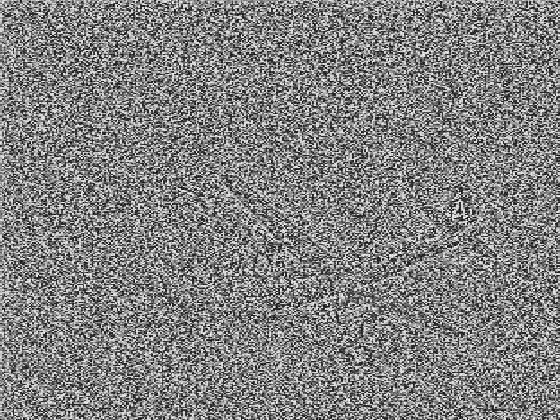}\vspace{0.1cm}
\includegraphics[height=2.9cm,width=2.9cm]{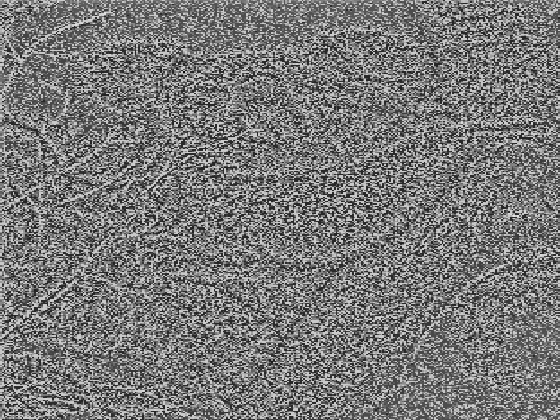}\vspace{0.1cm}
\caption{$f-v-u$}
\end{subfigure}

\caption{\textbf{First column}: Images corrupted with salt \& pepper and zero-mean Gaussian noise of different intensities. \textbf{Second column}: Denoising results. \textbf{Third column}: Salt \& pepper noise component. \textbf{Fourth column}: Gaussian noisy image residuum. \textbf{Fifth column}:  Gaussian noise component.\\
 \textbf{First row}: salt \& pepper noise density $d=5\%$, Gaussian noise variance $\sigma^2=0.001$. Noisy image PSNR=$17.91$ dB.  Denoised version PSNR=$29.03$ dB. Parameters: $\lambda_1=651$, $\lambda_2=6779$.\\  
 \textbf{Second row}: salt \& pepper noise density $d=5\%$, Gaussian noise variance $\sigma^2=0.01$. Noisy image PSNR=$16.11$ dB. Denoised version PSNR=$25.02$ dB. Parameters: $\lambda_1=622$, $\lambda_2=4201$. \\
  \textbf{Third row}: salt \& pepper noise density $d=10\%$, Gaussian noise variance $\sigma^2=0.005$. Noisy image PSNR=$14.63$ dB. Denoised version PSNR=$31.43$ dB. Parameters: $\lambda_1=523$, $\lambda_2=5551$. \\  
  \textbf{Fourth row}: salt \& pepper noise density $d=15\%$, Gaussian noise variance $\sigma^2=0.005$. Noisy image PSNR=$12.66$ dB. Denoised version PSNR=$26.40$ dB. Parameters $\lambda_1=482$, $\lambda_2=5233$. }
\label{result:impgauss}
\end{figure}

\medskip

Figure \ref{fig:asymptoticsimpgauss} confirms the convergence results of Proposition \ref{prop:impgaussasympt}, i.e. the single noise models are recovered asymptotically. The salt \& pepper and the Gaussian noise component of the model are plotted and their convergence to zero is observed as the corresponding weighting parameter goes to infinity. In particular, this behaviour corresponds to an ``asymptotical" convergence of the combined model \eqref{gaussimp:numer} to the classical TV denoising models for single noise removal (i.e. the classical ROF model \cite{rudinosherfatemi} for the  Gaussian noise case and the TV-$L^1$ model \cite{duvaltvL1} for the salt \& pepper case).

\begin{figure}[!h]
\begin{subfigure}[b]{0.45\textwidth}
\begin{center}
\includegraphics[height=4cm,width=6cm]{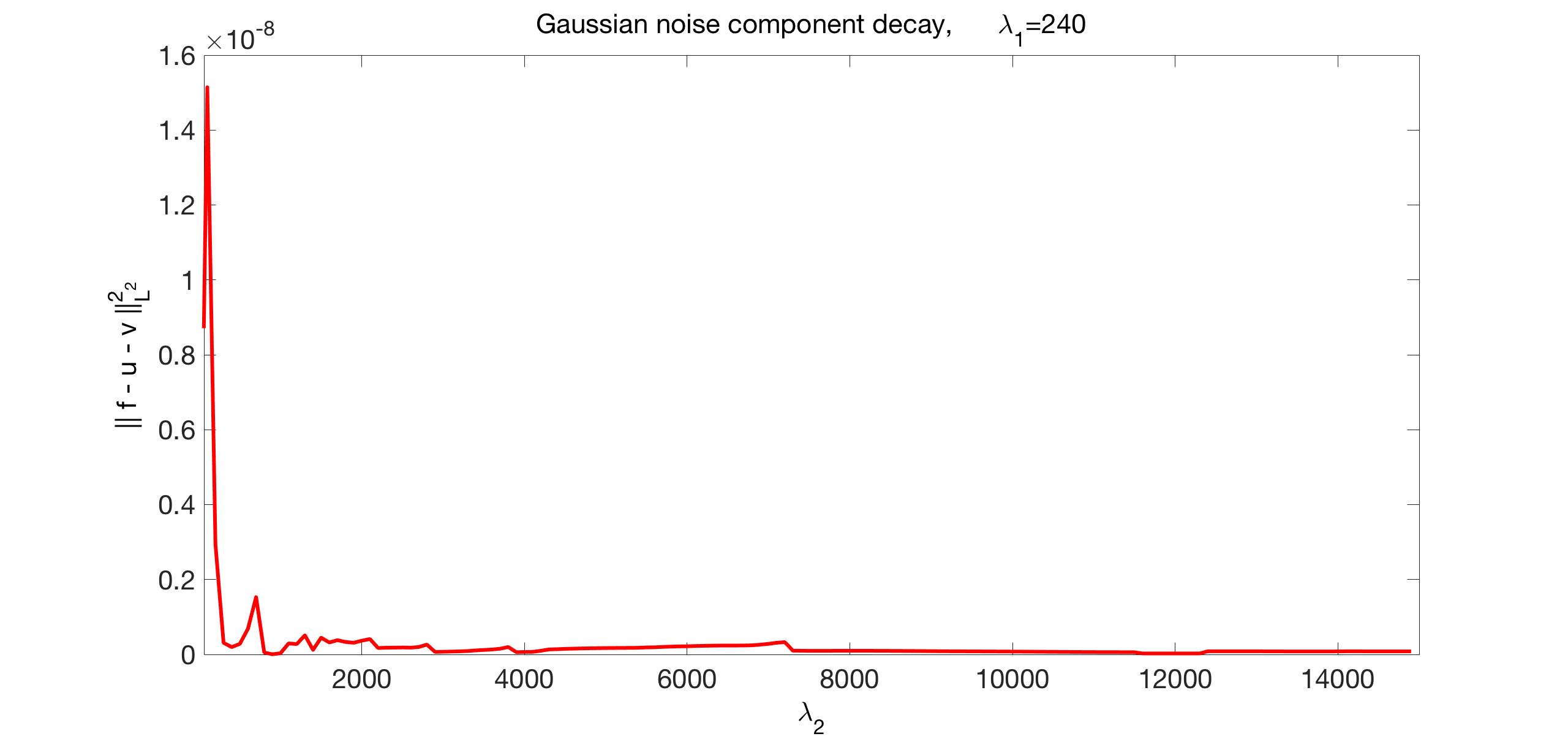}\vspace{0.1cm}
\end{center}
\caption{$\| f-u-v\|_{L^2}^2$ decay as $\lambda_2\to\infty$}
\end{subfigure}
\begin{subfigure}[b]{0.45\textwidth}
\begin{center}
\includegraphics[height=4cm,width=6cm]{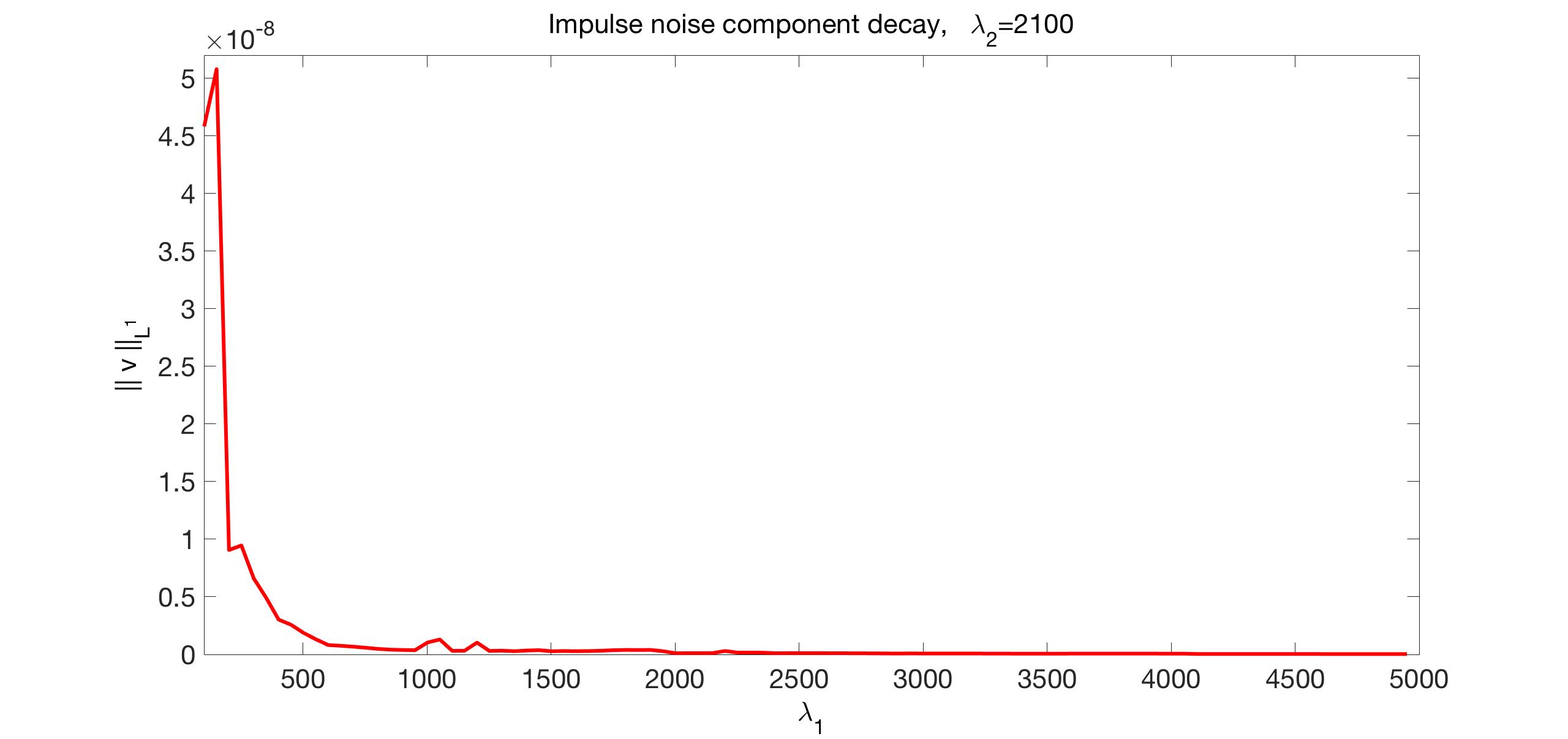}\vspace{0.1cm}
\end{center}
\caption{$\| v\|_{L^1}$ decay as $\lambda_1\to\infty$}
\end{subfigure}
\caption{Noise components behaviour as parameters $\lambda_1, \lambda_2$ of \eqref{gaussimp:numer} go to infinity.}
\label{fig:asymptoticsimpgauss}
\end{figure}

\medskip

To gain insight on the sensitivity of the reconstruction with respect to the choice of the parameters $\lambda_1$ and $\lambda_2$, we compare in Figure \ref{fig:comparisonsolutions} the solutions computed for the TV-IC model \eqref{gaussimp:numer} for different values of $\lambda_1$ and $\lambda_2$. The noisy image considered has been corrupted by a combination of salt \& pepper (density of missing pixels $d=5\%$) and Gaussian noise (with zero mean and $\sigma^2=0.005$). For comparison, we also present the denoising results computed with the standard denoising models TV-$L^1$ \cite{nikolovaoutremov,duvaltvL1}, TV-$L^2$ \cite{rudinosherfatemi} and the additive TV-$L^1$-$L^2$ combination \cite{bilevellearning,noiselearning,langerl1l2}. 
For reference, we recall below the models mentioned:
\begin{align}  
& \min_{u} \left\{ |Du|_\gamma(\Omega) + \lambda_1~\|f-u\|_{L^1(\Omega),\gamma}\right\}, \label{TVL1}\tag{TV-$L^1$}\\
& \min_{u} \left\{ |Du|_\gamma(\Omega) + \frac{\lambda_2}{2}~\|f-u\|^2_{L^2(\Omega)}\right\}, \label{TVL2}\tag{TV-$L^2$}\\
& \min_{u} \left\{ |Du|_\gamma(\Omega) + \lambda_1~\|f-u\|_{L^1(\Omega),\gamma}+ \frac{\lambda_2}{2}~\|f-u\|^2_{L^2(\Omega)}\right\}. \label{TVL1L2}
\tag{TV-$L^1$-$L^2$} \\ \notag
\end{align}

\vspace{-0.7cm}

\noindent For consistency, in our experiments we also Huber-regularise both the TV term \eqref{huberregularTV} and the $L^1$ term. As expected from Proposition \ref{prop:impgaussasympt} and verified numerically in Figure \ref{fig:asymptoticsimpgauss}, we observe that TV-$L^1$ and TV-$L^2$-type solutions can be obtained from \eqref{gaussimp:numer} by considering large weighting parameters $\lambda_2$ or $\lambda_1$, respectively. In these situations, we note that only one component of the noise is smoothed, namely the one corresponding to the active (i.e. non-vanishing) fidelity term in the model. Moreover, we observe that the computed solution of the TV-IC model \eqref{gaussimp:numer} is comparable to the one computed using the TV-$L^1$-$L^2$ denoising model, but with the additional property of a noise decomposition shown in Figure \ref{result:impgauss} above. Finally, as proved in point  \textit{iii)} of Proposition \ref{prop:impgaussasympt}, the noisy image $f$ is completely recovered by taking large parameters $\lambda_1$ and $\lambda_2$. For every model considered, all the parameters have optimised with respect to the best PSNR of the denoised image $u$. When looking at the asymptotics with respect to one parameter weight, one parameter has been set to $1e5$ and the other has been optimised with respect to the best PSNR of $u$. In the case when the joint asymptotics are studied, both parameters have been set to $1e5$.  

\begin{figure}[!h]
\begin{subfigure}[b]{0.22\textwidth}
\begin{center}
\includegraphics[height=3cm,width=3cm]{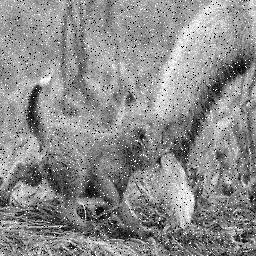}
\caption{Noisy image}
\end{center}
\end{subfigure}\quad
\begin{subfigure}[b]{0.22\textwidth}
\begin{center}
\includegraphics[height=3cm,width=3cm]{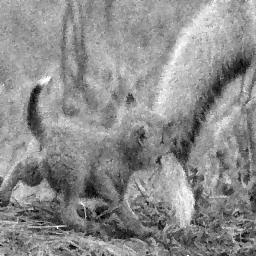}
\caption{TV-$L^1$}
\end{center}
\end{subfigure}
\begin{subfigure}[b]{0.22\textwidth}
\begin{center}
\includegraphics[height=3cm,width=3cm]{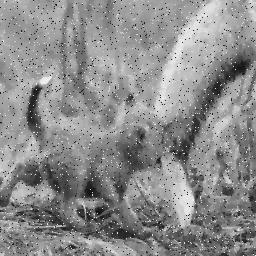}
\caption{TV-$L^2$}
\end{center}
\end{subfigure}\quad
\begin{subfigure}[b]{0.22\textwidth}
\begin{center}
\includegraphics[height=3cm,width=3cm]{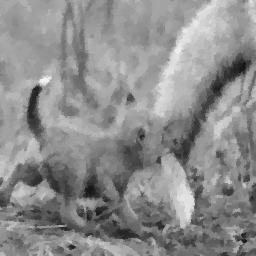}
\caption{TV-$L^1$-$L^2$}
\end{center}
\end{subfigure}\\
\begin{subfigure}[b]{0.22\textwidth}
\begin{center}
\includegraphics[height=3cm,width=3cm]{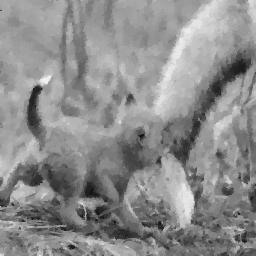}
\caption{TV-IC}
\end{center}
\end{subfigure}  \quad
\begin{subfigure}[b]{0.22\textwidth}
\begin{center}
\includegraphics[height=3cm,width=3cm]{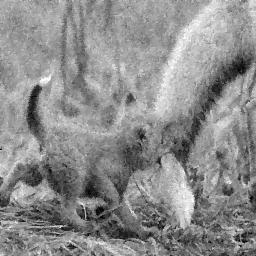}
\caption{TV-IC, $\lambda_2\gg 1$}
\end{center}
\end{subfigure}
\begin{subfigure}[b]{0.22\textwidth}
\begin{center}
\includegraphics[height=3cm,width=3cm]{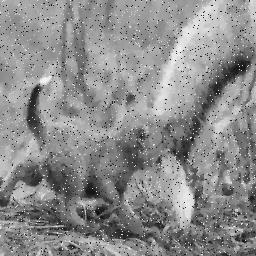}
\caption{TV-IC, $\lambda_1\gg 1$}
\end{center}
\end{subfigure}\quad
\begin{subfigure}[b]{0.22\textwidth}
\begin{center}
\includegraphics[height=3cm,width=3cm]{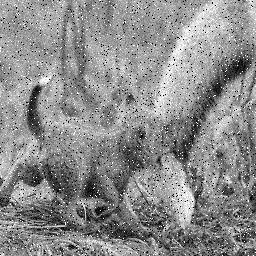}
\caption{TV-IC, $\lambda_1,\lambda_2\gg 1$}
\end{center}
\end{subfigure} 
\caption{Comparison between \eqref{TVL1}, \eqref{TVL2}, \eqref{TVL1L2} and TV-IC \eqref{gaussimp:numer} reconstructions for optimised and large parameters $\lambda_1, \lambda_2$. \\  \textbf{First row}: (a) noisy image corrupted with salt \& pepper noise ($d=5\%$) and Gaussian noise with zero mean and variance $\sigma^2=0.005$, PSNR=$17.18$ dB. (b) TV-$L^1$ solution with $\lambda_1=352$, PSNR=$25.53$ dB. (c) TV-$L^2$ solution with $\lambda_2=2121$, PSNR=$20.26$ dB. (d) TV-$L^1$-$L^2$ solution with $\lambda_1=351, \lambda_2=258$, PSNR=$25.62$ dB. \\ \textbf{Second row}: (e) TV-IC solution with $\lambda_1=352, \lambda_2=2121$, PSNR=$26.51$ dB (f) TV-IC solution with $\lambda_1=352, \lambda_2=1e5$, PSNR= $25.61$ dB. (g) TV-IC solution with $\lambda_1=1e5, \lambda_2=2121$, PSNR=$20.13$ dB. (e) TV-IC solution with $\lambda_1=\lambda_2=1e5$, PSNR=$17.17$ dB.}
\label{fig:comparisonsolutions}
\end{figure}

\subsection{Gaussian-Poisson case: numerical results}   \label{numres:gausspoiss}

For the numerical solution of the mixed Gaussian-Poisson model presented in Section \ref{subsec:gausspoiss}, we relax the unilateral constraints on $u$ and $v$ by adding two standard penalty terms as follows:
\begin{align}   
\min_{u,v} \Big\{ J_\gamma(u,v):=|Du|_\gamma(\Omega) + \frac{\lambda_1}{2}~\| v \|_{\ell^2}^2  +  \lambda_2~ \int_\Omega \left(u - (f-v) +(f-v)\log \left(\frac{f-v}{u}\right)\right)~dx   \notag \\ 
+ \frac{\gamma_1}{2}\|\min(u,0)\|_{\ell^2}^2 + \frac{\gamma_2}{2}  \|\min(f-v,0)\|_{\ell^2}^2 \Big\}. \label{gausspois:numer}
\end{align}
In the following numerical experiments, we start from initial values $\gamma_1^0=10$ and $\gamma_2^0=100$ and increase them throughout the iterations.

The optimality conditions for \eqref{gausspois:numer} read:
\begin{align*}
\frac{\partial J_\gamma}{\partial u}&=- \mathrm{div}\,\left(\frac{\gamma\nabla u}{\max(\gamma|\nabla u|, 1)} \right) +\lambda_2\left(1-\frac{f-v}{u}\right) + \gamma_1~ \chi_{\mathscr{I}_u} u  = 0,\\
\frac{\partial J_\gamma}{\partial v}&= \lambda_1~v -\lambda_2~\log\left( \frac{f-v}{u} \right) +\gamma_2~\chi_{\mathscr{I}_v} (v-f)= 0
\end{align*}
where $\chi_{\mathscr{I}_u}$ and $\chi_{\mathscr{I}_v}$ are the characteristic functions of the sets $\mathscr{I}_u=\left\{ x\in\Omega: u(x)<0\right\}$ and $\mathscr{I}_v=\left\{ x\in\Omega: v(x)>f(x)\right\}$, respectively.

Similarly as before, we express the system above in primal-dual form and write the modified SSN iteration for the increments $\delta_u, \delta_q, \delta_v$ which reads:
\begin{equation*}
\left\{
\begin{aligned}
&- \mathrm{div}\, \delta_q +\lambda_2\left(\frac{f-v}{u^2}\right) \delta_u + \frac{\lambda_2}{u} \delta_{v} +\gamma_1~\chi_{\mathscr{I}_u}~\delta_u =   \mathrm{div}\, q - \lambda_2\left(1-\frac{f-v}{u}\right) - \gamma_1~ \chi_{\mathscr{I}_u} u , \\
&\delta_q - \frac{\gamma \nabla \delta_u}{\max(1,\gamma|\nabla u|)}+ \chi_{\mathcal U_\gamma}  \gamma^2 \frac{\nabla u^T \nabla \delta_u}{\max(1,\gamma|\nabla u|)^2} \frac{q}{\max(1,|q|)} =-q+ \frac{\gamma \nabla u}{\max(1,\gamma|\nabla u|)},  \\
& \lambda_1~\delta_v+\frac{\lambda_2}{u}\delta_u+\frac{\lambda_2}{f-v}~\delta_v+\gamma_2~\chi_{\mathscr{I}_v}~\delta_v=-\left(  \lambda_1~v -\lambda_2~\log\left( \frac{f-v}{u} \right) +\gamma_2~\chi_{\mathscr{I}_v} (v-f) \right), 
\end{aligned}\right.
\end{equation*}
where the set $\mathcal{U}_\gamma$ is the same as the one defined in the previous section.

In Figure \ref{fig:poissgaussresults} we report the denoising results for the mixed Gaussian-Poisson model solved via the SSN iteration above. In order to generated the noisy data, we do as follows: at each pixel $(x_i,x_j)$ of  the image domain $\Omega$, the Poisson noise component is distributed with Poisson distribution \eqref{poissondistribution} with parameter $u_{ij}=u(x_i,x_j)$, whereas the Gaussian noise component of the examples has zero mean and different intensities (variance) specified in each case. Also in this case  we observe that the noise components are decomposed as expected, with the Gaussian one being distributed over the whole image domain and the Poisson one depending on the intensity of the image itself. This means that low intensity (darker) areas in the image will be corrupted by a smaller amount of Poisson noise than the high intensity (brighter) regions (compare, for instance the black sky background in the moon image or the synthetic image in the last row), whereas the Gaussian noise component is independent of the image intensity. As above, we observe a loss of structure in the TV reconstructed image which are captured in the noise components. 

\begin{figure}[!h]
\centering
\begin{subfigure}[b]{0.2\textwidth}
\includegraphics[height=2.9cm,width=2.9cm]{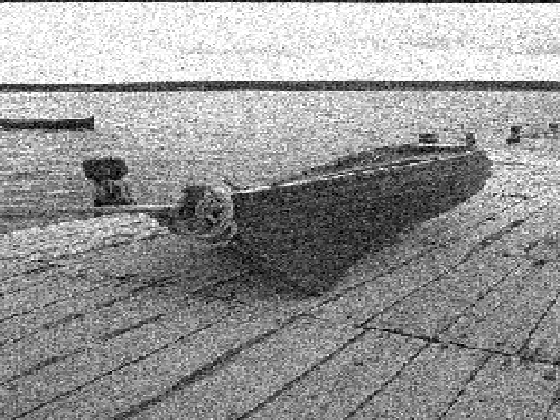}\vspace{0.1cm}
\includegraphics[height=2.9cm,width=2.9cm]{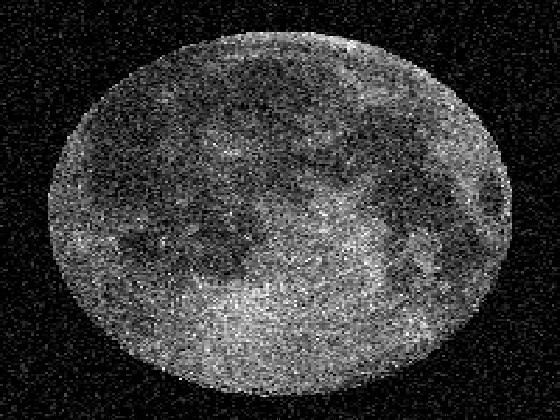}\vspace{0.1cm}
\includegraphics[height=2.9cm,width=2.9cm]{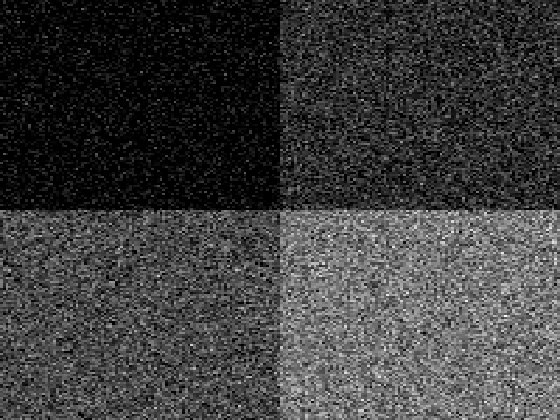}\vspace{0.1cm}
\includegraphics[height=2.9cm,width=2.9cm]{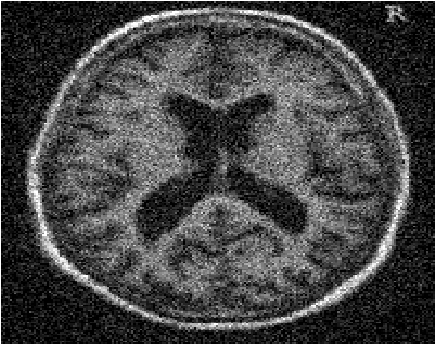}\vspace{0.1cm}
\caption{$f$}
\end{subfigure}
\hspace{-0.3cm}
\begin{subfigure}[b]{0.2\textwidth}
\includegraphics[height=2.9cm,width=2.9cm]{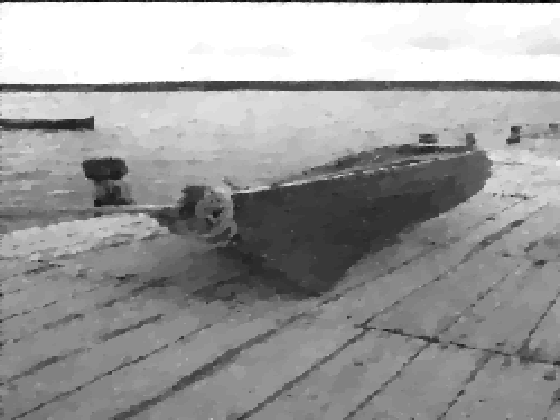}\vspace{0.1cm}
\includegraphics[height=2.9cm,width=2.9cm]{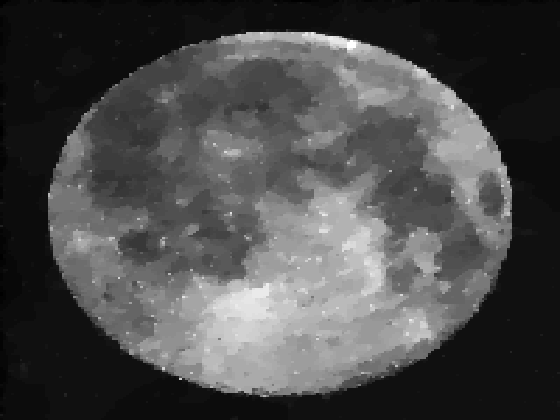}\vspace{0.1cm}
\includegraphics[height=2.9cm,width=2.9cm]{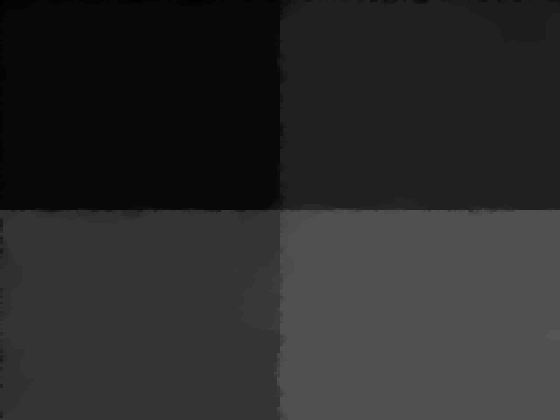}\vspace{0.1cm}
\includegraphics[height=2.9cm,width=2.9cm]{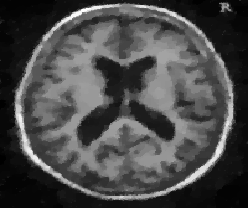}\vspace{0.1cm}
\caption{$u$}
\end{subfigure}
\hspace{-0.3cm}
\begin{subfigure}[b]{0.2\textwidth}
\includegraphics[height=2.9cm,width=2.9cm]{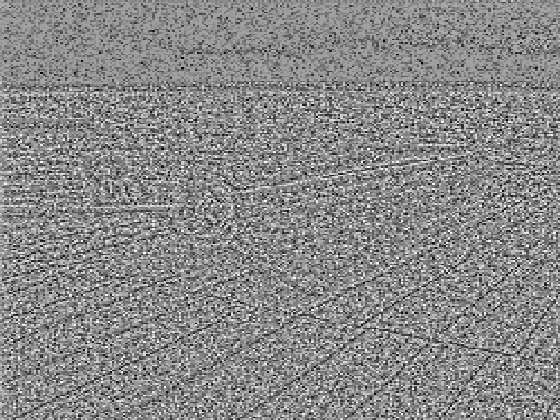}\vspace{0.1cm}
\includegraphics[height=2.9cm,width=2.9cm]{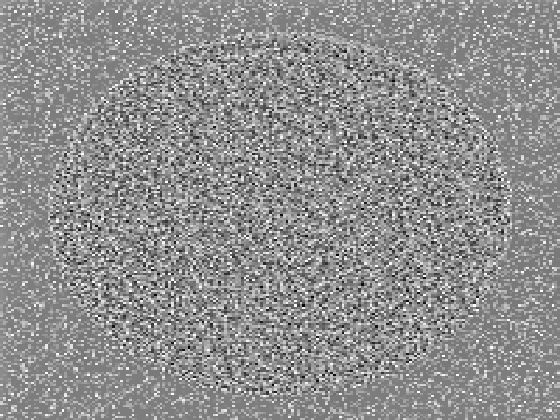}\vspace{0.1cm}
\includegraphics[height=2.9cm,width=2.9cm]{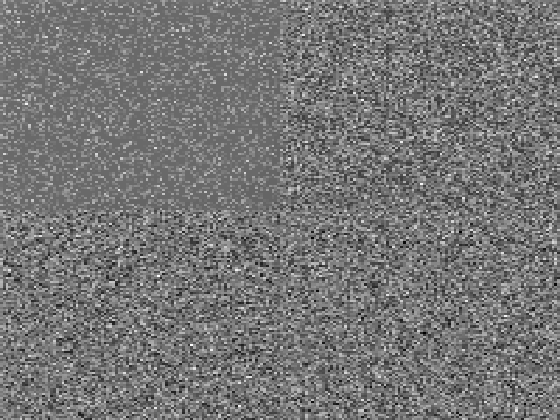}\vspace{0.1cm}
\includegraphics[height=2.9cm,width=2.9cm]{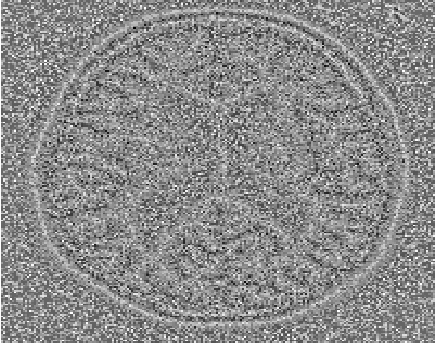}\vspace{0.1cm}
\caption{$v$}
\end{subfigure}
\hspace{-0.3cm}
\begin{subfigure}[b]{0.2\textwidth}
\includegraphics[height=2.9cm,width=2.9cm]{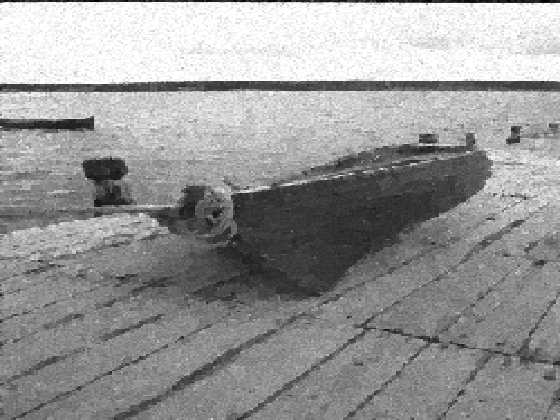}\vspace{0.1cm}
\includegraphics[height=2.9cm,width=2.9cm]{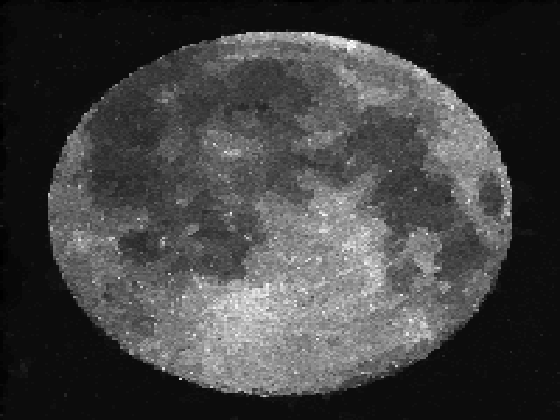}\vspace{0.1cm}
\includegraphics[height=2.9cm,width=2.9cm]{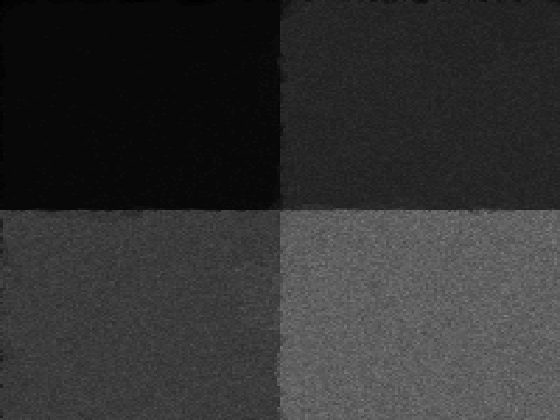}\vspace{0.1cm}
\includegraphics[height=2.9cm,width=2.9cm]{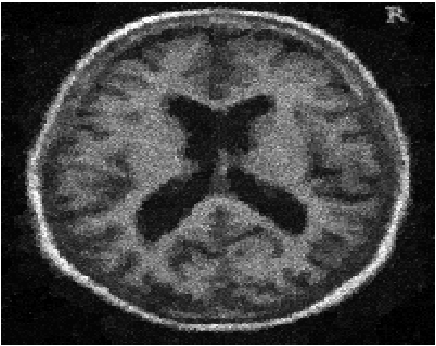}\vspace{0.1cm}
\caption{$f-v$}
\end{subfigure}
\hspace{-0.3cm}
\begin{subfigure}[b]{0.2\textwidth}
\includegraphics[height=2.9cm,width=2.9cm]{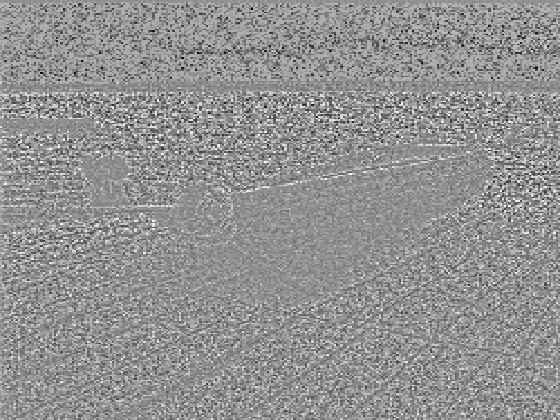}\vspace{0.1cm}
\includegraphics[height=2.9cm,width=2.9cm]{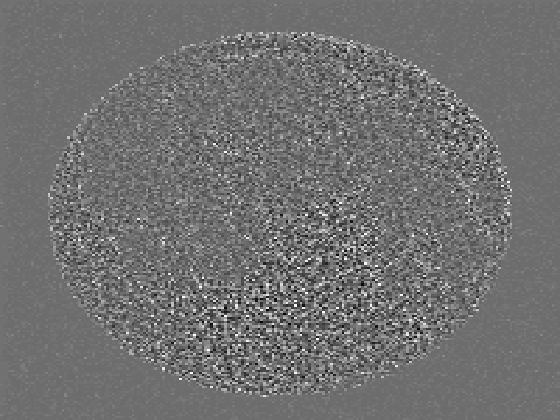}\vspace{0.1cm}
\includegraphics[height=2.9cm,width=2.9cm]{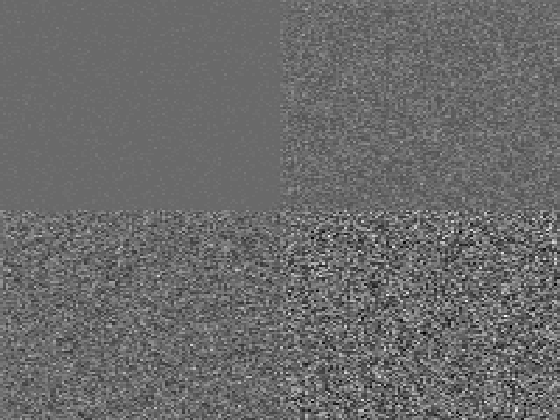}\vspace{0.1cm}
\includegraphics[height=2.9cm,width=2.9cm]{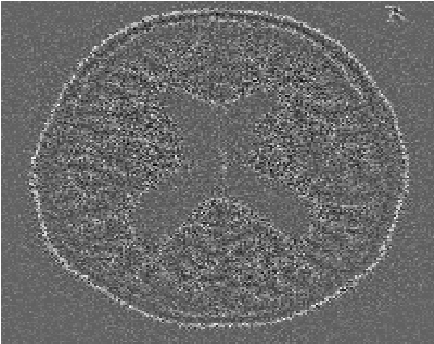}\vspace{0.1cm}
\caption{$f-v-u$}
\end{subfigure}

\caption{\textbf{First column}: Images corrupted with Poisson and Gaussian noise. \textbf{Second column}: Denoising result. \textbf{Third column}: Gaussian noise component. \textbf{Fourth column}: Poisson residuum. \textbf{Fifth column}:  Poisson noise component. \\\textbf{First row}: Gaussian noise variance $\sigma^2=0.005$. Noisy image PSNR=$19.66$ dB.   Denoised version PSNR=$26.03$ dB. Parameters: $\lambda_1=4289$, $\lambda_2=5503$. \\
 \textbf{Second row}: Gaussian noise variance $\sigma^2=0.01$. Noisy image PSNR=$19.46$ dB.  Denoised version PSNR=$26.19$ dB. Parameters: $\lambda_1=2903$, $\lambda_2=2107$. \\
 \textbf{Third row}: Gaussian noise variance $\sigma^2=0.005$. Noisy image PSNR=$22.39$ dB. Denoised version PSNR=$33.04$ dB. Parameters: $\lambda_1=2105$, $\lambda_2=1896$.\\
  \textbf{Fourth row}: Gaussian noise variance $\sigma^2=0.05$. Noisy image PSNR=$18.62$ dB. Denoised version PSNR=$23.87$ dB. Parameters: $\lambda_1=809$, $\lambda_2=712$.}
\label{fig:poissgaussresults}
\end{figure}

\medskip

In Figure \ref{fig:gausspoisscomparison} we compare the reconstructions obtained using TV denoising models with single $L^2$ \cite{rudinosherfatemi}, $KL$ \cite{alexTVpoisson,lechartrandTVpoisson} and the sum of the two data fidelities as in \cite{noiselearning,bilevellearning} as well as with the exact log-likelihood derived in \cite{Jezierska2012,poissongauss2013}. Namely, we compare our method with Huber-regularised versions of the \eqref{TVL2} model, and with
\begin{align}
& \min_u \left\{ |Du|_\gamma(\Omega) + \lambda_2 \int_\Omega \left(u - f~\log u\right)~dx \right\},\label{TVKL}\tag{TV-KL} \\ 
& \min_u \left\{ |Du|_\gamma(\Omega) + \frac{\lambda_1}{2}\| f-u \|_{L^2(\Omega)}^2 + \lambda_2 \int_\Omega \left(u - f~\log u\right)~dx \right\}, \label{TVL2KL}\tag{TV-$L^2$-KL}\\
& \min_u \left\{ |Du|_\gamma(\Omega) -\int_\Omega \log \left(\sum_{n=0}^{+\infty} \frac{u^n e^{-u}}{n!}
\frac{e^{-\frac{(u-n)^2}{2\sigma^2}}}{\sqrt{2\pi\sigma^2}}\right)~dx \right\}. \label{TV_exactPG}\tag{TV-GP}
\end{align}

\begin{figure}[!h]
\begin{subfigure}[b]{0.28\textwidth}
\begin{center}
\includegraphics[height=3.3cm,width=3.3cm]{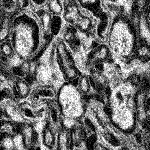}
\caption{Noisy image}
\end{center}
\end{subfigure}\quad
\begin{subfigure}[b]{0.28\textwidth}
\begin{center}
\includegraphics[height=3.3cm,width=3.3cm]{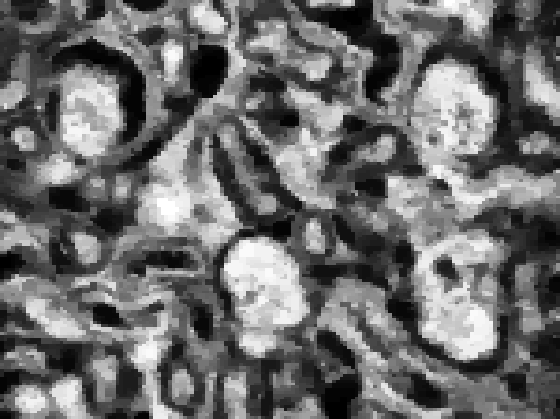}
\caption{TV-$L^2$}  \label{TVL2recon}
\end{center}
\end{subfigure}\quad
\begin{subfigure}[b]{0.28\textwidth}
\begin{center}
\includegraphics[height=3.3cm,width=3.3cm]{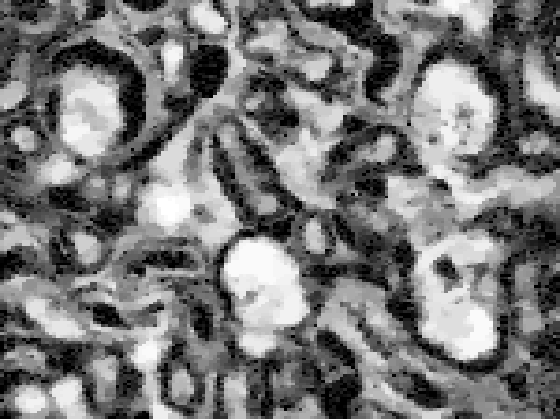}
\caption{TV-KL} \label{TVKLrecon}
\end{center}
\end{subfigure} \\
\begin{subfigure}[b]{0.28\textwidth}
\begin{center}
\includegraphics[height=3.3cm,width=3.3cm]{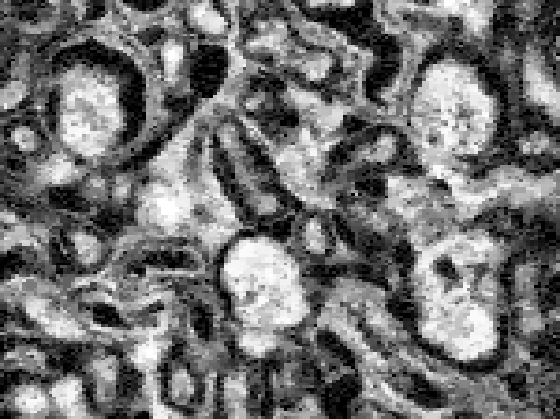}
\caption{TV-$L^2$-KL}  \label{TVL2KLrecon}
\end{center}
\end{subfigure} \quad
\begin{subfigure}[b]{0.28\textwidth}
\begin{center}
\includegraphics[height=3.3cm,width=3.3cm]{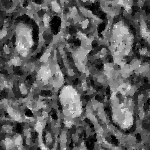}
\caption{TV-GP} 
\end{center}
\end{subfigure}\quad
\begin{subfigure}[b]{0.28\textwidth}
\begin{center}
\includegraphics[height=3.3cm,width=3.3cm]{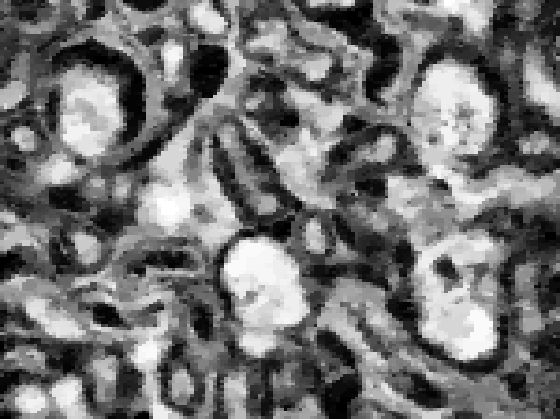}
\caption{TV-IC}  \label{TVICrecon}
\end{center}
\end{subfigure}
\caption{Comparison between solutions of TV-IC model \eqref{fig:poissgaussresults} and solutions of  \eqref{TVL2}, \eqref{TVKL}, \eqref{TVL2KL} and \eqref{TV_exactPG} models. \textbf{First row}: (a) noisy image corrupted with Gaussian noise with zero mean and variance $\sigma^2=0.005$ and Poisson noise with parameter $u$, PSNR=$18.81$dB. (b) TV-$L^2$ solution with $\lambda_1=1800$, PSNR=$22.97$dB. (c) TV-KL solution with $\lambda_2=1200$, PSNR=$21.19$dB. \textbf{Second row}: (d) TV-$L^2$-KL solution with $\lambda_1=520, \lambda_2=1100$. PSNR=$22.24$dB.  (e) Solution of the TV-GP model \eqref{TV_exactPG} model. PSNR=$19.65$dB. (f) Solution of TV-IC model \eqref{gausspois:numer} with $\lambda_1=1200, \lambda_2=1800$, PSNR=$22.52$dB.}
\label{fig:gausspoisscomparison}
\end{figure}

Finally, in Figure \ref{fig:gausspoisscomparisonlargeparams} we report the results of the TV-IC model for large values of the parameters $\lambda_1$ and $\lambda_2$ to show how the reconstruction changes for different choices of the parameters. As shown in Section \ref{subsec:asymptpoissgauss} such choices correspond to enforcing the \eqref{TVKL} and the \eqref{TVL2KL} model, respectively. 

\begin{figure}[!h]
\begin{subfigure}[b]{0.22\textwidth}
\begin{center}
\includegraphics[height=3cm,width=3cm]{brain_gauss001pois_noise}
\caption{Noisy image}
\end{center}
\end{subfigure}
\begin{subfigure}[b]{0.22\textwidth}
\begin{center}
\includegraphics[height=3cm,width=3cm]{brain_gauss001pois_den}
\caption{TV-IC}   \label{TVICbrain}
\end{center}
\end{subfigure}\quad
\begin{subfigure}[b]{0.22\textwidth}
\begin{center}
\includegraphics[height=3cm,width=3cm]{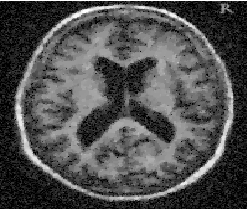}
\caption{TV-IC, $\lambda_1\gg1$}  \label{TVIClargelambda1}
\end{center}
\end{subfigure}\quad
\begin{subfigure}[b]{0.22\textwidth}
\begin{center}
\includegraphics[height=3cm,width=3cm]{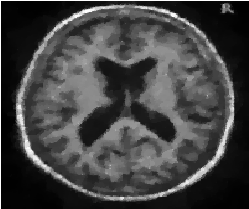}
\caption{TV-IC, $\lambda_2 \gg 1$} \label{TVIClargelambda2}
\end{center}
\end{subfigure}
\caption{Comparison between solutions of TV-IC model \eqref{fig:poissgaussresults} for large values of $\lambda_1$ and $\lambda_2$. \\
(a) Noisy image corrupted with Gaussian noise with zero mean and variance $\sigma^2=0.05$ and Poisson noise, PSNR=$18.62$ dB. (b) TV-IC solution with $\lambda_1=809$, $\lambda_2=712$, PSNR=$23.87$ dB. (c) TV-IC solution with large $\lambda_1$, $\lambda_2=712$, PSNR=$19.83$ dB. (d) TV-IC solution with large $\lambda_2$, $\lambda_1=809$, PSNR=$22.71$ dB.}
\label{fig:gausspoisscomparisonlargeparams}
\end{figure}

To conclude, we report in Figure \ref{fig:asymptoticspoissgauss} numerical tests on the asymptotical behaviour of the model \eqref{gausspois:numer} as the fidelity weights go to infinity. As shown in Proposition \ref{prop:gausspoissasympt}, both the TV-$L^2$ model for Gaussian noise removal \eqref{TVL2} (whose solution is denoted by $u_{TV-L^2}$) and the TV-KL one \eqref{TVKL} (whose solution is denoted by $u_{TV-KL}$) are recovered asymptotically under appropriate norms.

\begin{figure}[!h]
\centering
\begin{subfigure}[b]{0.45\textwidth}
\includegraphics[height=4cm,width=7cm]{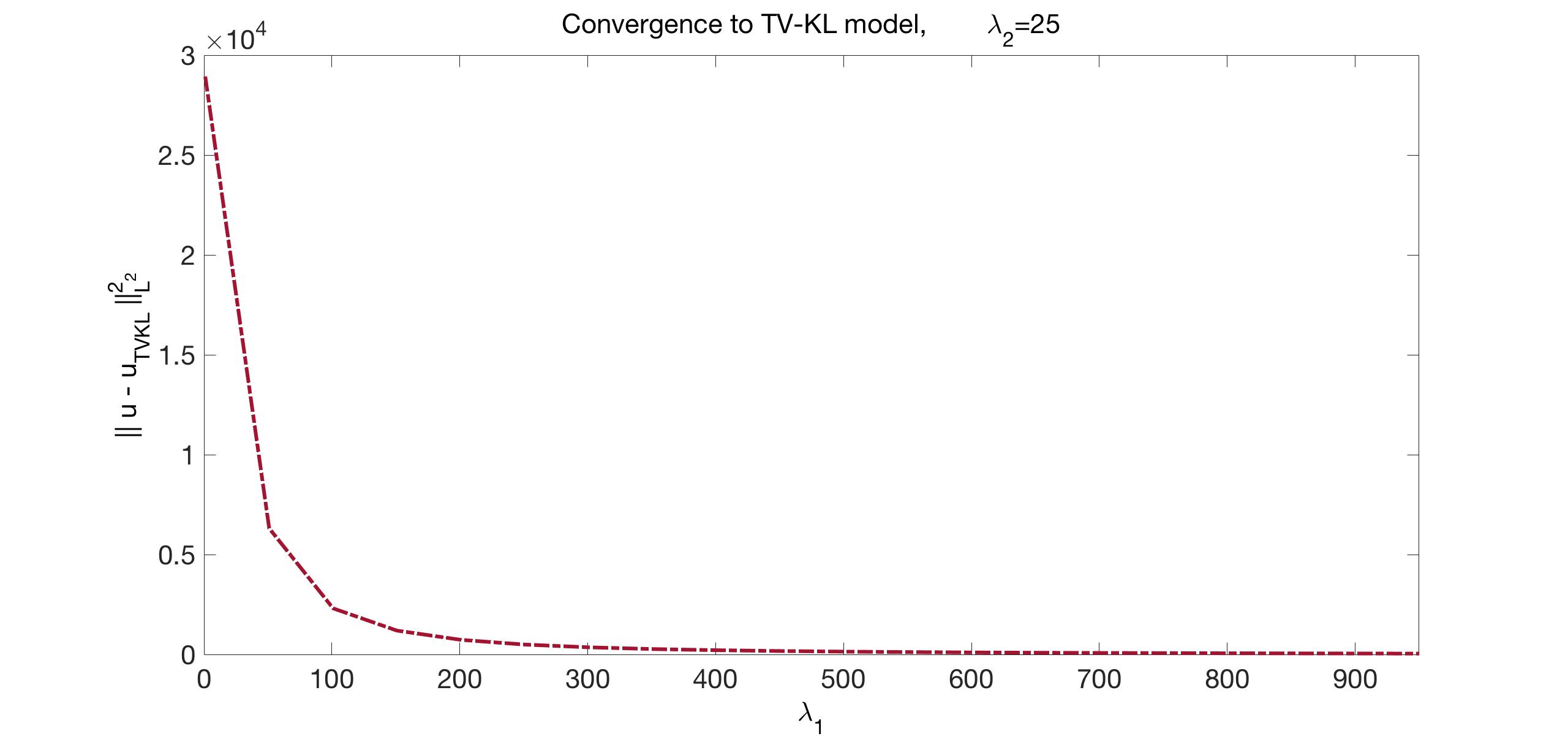}\vspace{0.1cm}
\caption{$\|u-u_{TV-KL}\|_{L^2}^2$ decay}
\end{subfigure}
\begin{subfigure}[b]{0.45\textwidth}
\includegraphics[height=4cm,width=7cm]{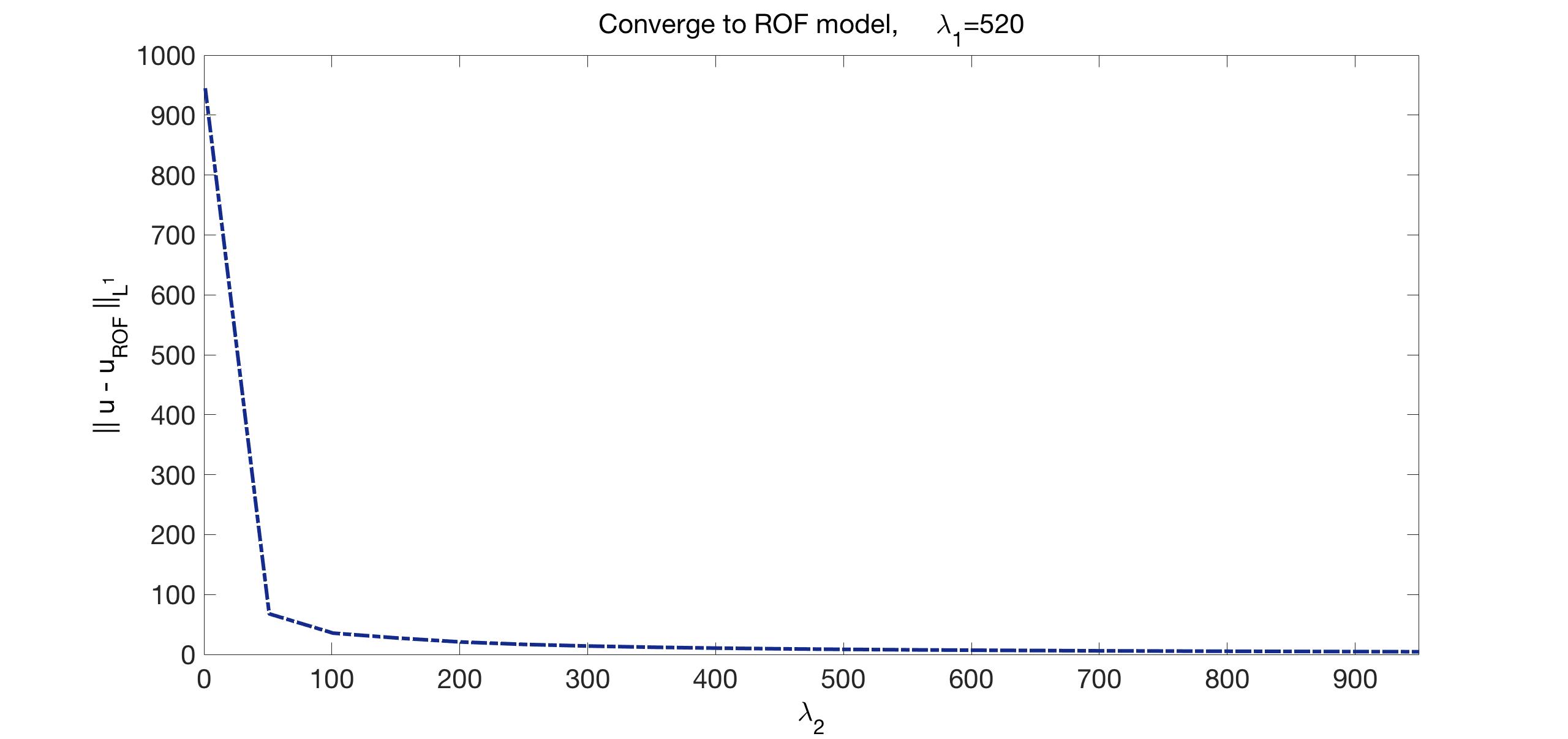}\vspace{0.1cm}
\caption{$\|u-u_{TV-L^2}\|_{L^1}$ decay}
\end{subfigure}
\caption{Convergence to single noise models as parameters $\lambda_1, \lambda_2$ of \eqref{gausspois:numer} go to infinity.}
\label{fig:asymptoticspoissgauss}
\end{figure}

\section{Conclusions and outlook}

In this paper we presented a novel variational approach for mixed noise removal. Mixed noise occurs in many applications where different acquisition and/or transmission sources may create interferences of different statistical nature in the image. Here we focused on two cases of mixed noise, namely salt \& pepper noise mixed with Gaussian noise and Gaussian noise combined with Poisson noise. Our variational model, which we call TV-IC, constitutes an infimal convolution combination of standard data fidelities classically associated to one single noise distribution and a total variation (TV) regularisation is used as regularising energy. The well-posedness of the model is studied and it is shown how single noise models can be recovered from the combined model ``asymptotically'', i.e. by letting the weighting parameters of the model become infinitely large. We also gave a statistical motivation for our model, modelling the noise as probability distributions with negative exponential structure (such as Laplace, Gaussian and Poisson). For our numerical experiments we used a semi-smooth Newton (SSN) second-order method to solve efficiently a Huber-regularised version of the problem. In several numerical results and comparisons with existing methods the properties of the proposed TV-IC model are discussed. 

Our numerical results show the property of the model of decomposing the noise present in the data in its different single-noise components. This is achieved by the particular modelling of data fidelities considered, which allows for a splitting of the noise in its constituting parts.  Comparisons with state-of-the-art models dealing with the combined case are reported. From a computational point of view, the use of a second-order SSN scheme allows for an efficient computation of the numerical solution of the TV-IC model.


The novel modelling of mixed noise distributions introduced in this paper offers several interesting problems for future research. Among those we believe it would be interesting to study:
\begin{itemize} 
\item Parameter learning for TV-IC as outlined in Section \ref{sec:learning_mot}.
\item The design of a more general model which could feature a combination of more noise distributions. 
\item The use of higher-order regularisations such as ICTV \cite{chambollelions1997} and TGV \cite{TGV,diff_tens} for the reduction of the problem of loss of structures encoded in the noise residuals, compare Figure \ref{result:impgauss} and \ref{fig:poissgaussresults}.
\item The derivation of a quality measure which is especially designed to assess optimality with respect to the splitting of the noise into its components.
\item Characterisation of solutions of the TV-IC model, starting from the one-dimensional case.
\end{itemize}

Despite these open problems, we believe that the method presented in this paper is a sensible and efficient alternative to state-of-the-art data fidelity modelling of mixed noise distributions due to its statistical coherence, simple numerical realisation and new noise decomposition feature.

\subsection{Outlook on parameter learning}  \label{sec:learning_mot}
In the spirit of recent developments in the context of learning the optimal noise model from examples \cite{noiselearning,lucasampling,bilevellearning,kunisch2013bilevel,interiorpaper2015}, we consider the TV-IC model and give a preliminary discussion on the selection of optimal parameters $\lambda_1$ and $\lambda_2$ in the mixed noise case. Denoting by $u_{\lambda_1,\lambda_2}$ the TV-IC reconstructed image and by $\tilde{u}$ the corresponding ground-truth, standard cost functionals in the literature assessing optimality are the $L^2$ cost functional 
\begin{equation}   \label{costfunctL2}
F_{L^2}(u_{\lambda_1,\lambda_2}):=\| u_{\lambda_1,\lambda_2} - \tilde{u} \|_{L^2(\Omega)}^2
\end{equation}
and the Huber-regularised TV cost functional
\begin{equation}    \label{huberisedTVcost}
F_{L^1_\gamma\nabla}(u_{\lambda_1,\lambda_2}):=\| D(u_{\lambda_1,\lambda_2} - \tilde{u}) \|_{L^1(\Omega),\gamma}
\end{equation}
where the $L^1$ term has been Huber-regularised as in \eqref{huberregular}-\eqref{huberregularTV}. This choice has been proposed in \cite{interiorpaper2015} and, despite being different from classical quality measures such as PSNR and SSIM, has been shown to represent visually pleasant results.

We consider the case of the brain image corrupted only with salt \& pepper noise with a percentage of missing pixels $d=5\%$, see Figure \ref{fig:learningmotivation}. In Figure \ref{fig:learning_impgasuss_onlyimp} we plot the cost functional \eqref{huberisedTVcost} against the parameters $\lambda_1$ and $\lambda_2$, indicating with a red cross the minimum of the cost functional within the tested range of parameters. For comparison, we consider the TV-IC model \eqref{gaussimp:numer} and the TV-$L^1$-$L^2$ model \eqref{TVL1L2} used previously in \cite{HintermuellerLanger2013,noiselearning}.  Both models accommodate salt \& pepper and Gaussian noise, whenever the parameters are positive and finite. However, in the particular case considered we expect to select optimal parameters $\lambda_1^*$ and $\lambda_2^*$ enforcing a TV-$L^1$ type model \eqref{TVL1} which is the appropriate one in the case of only salt \& pepper noise \cite{nikolovaoutremov,duvaltvL1}. Our plot shows that in both cases, the optimal value for \eqref{huberisedTVcost} within the tested range of parameters is achieved in correspondence with an optimal pair $(\lambda_1^*,\lambda_2^*)$ which approximates the TV-$L^1$ type model. In particular, the TV-IC plot in Figure \ref{optimalparams1} shows that the optimal choice of parameters corresponds to $\lambda_1^*=281$ and a very large $\lambda_2^*$. As shown in Section \ref{subsec:asymptgaussimp} and confirmed numerically in Figure \ref{fig:asymptoticsimpgauss}, in this case such choice approximates the TV-$L^1$ model, as the Gaussian noise residuum decays to zero as $\lambda_2\to \infty$. Similarly, for the TV-$L^1$-$L^2$ model in Figure \ref{optimalparams2}, the same optimal parameter $\lambda_1^*$ is found with $\lambda_2^*=0$, thus enforcing similarly a TV-$L^1$ denoising model.

\begin{figure}[!h]
\centering
\begin{subfigure}[b]{0.25\textwidth}
\includegraphics[height=3cm,width=4cm]{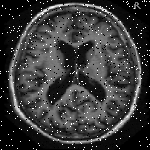}\vspace{0.1cm}
\caption{Noisy image}
\label{noisyimageimponly}
\end{subfigure}\qquad
\begin{subfigure}[b]{0.25\textwidth}
\includegraphics[height=3cm,width=4cm]{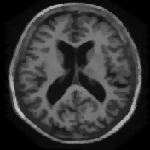}\vspace{0.1cm}
\caption{Optimal denoised }
\label{optimalbraindenoiseim}
\end{subfigure}\qquad
\begin{subfigure}[b]{0.3\textwidth}
\includegraphics[height=3.13cm,width=4cm]{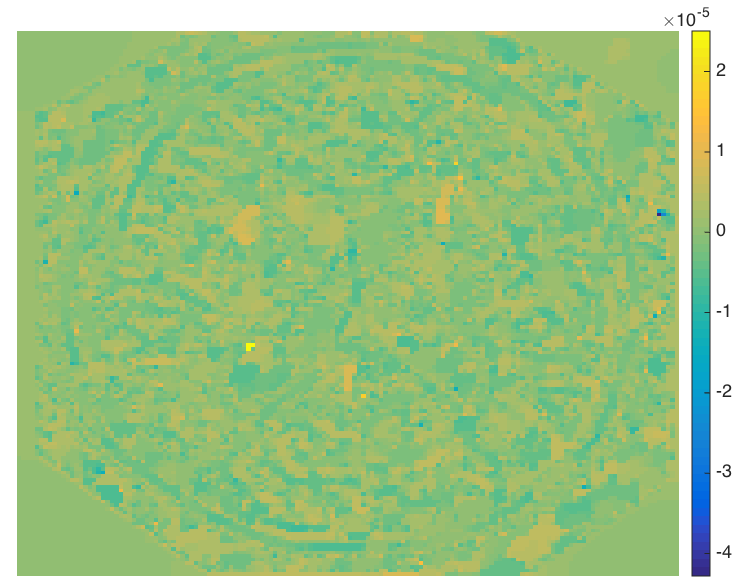}\vspace{0.1cm}
\caption{TV-IC/$L^1$-$L^2$ discrepancy}
\label{optimalbraindenoisediff}
\end{subfigure}
\caption{TV-IC denoising with optimal parameters for single salt \& pepper denoising. For both the TV-IC \eqref{gaussimp:numer} and  TV-$L^1$-$L^2$ denoising model \eqref{TVL1L2}, the optimal parameters $(\lambda_1^*, \lambda_2^*)$ selected enforce a TV-$L^1$ type model. Figure \ref{optimalbraindenoisediff} plots the difference between TV-IC and TV-$L^1$-$L^2$ solutions computed in correspondence with the optimal parameters: the maximum discrepancy between the two has absolute value of $4.32\cdot 10^{-5}$.}
\label{fig:learningmotivation}
\end{figure}

\begin{figure}[!h]
\centering
\begin{subfigure}[b]{0.45\textwidth}
\includegraphics[height=4cm,width=7cm]{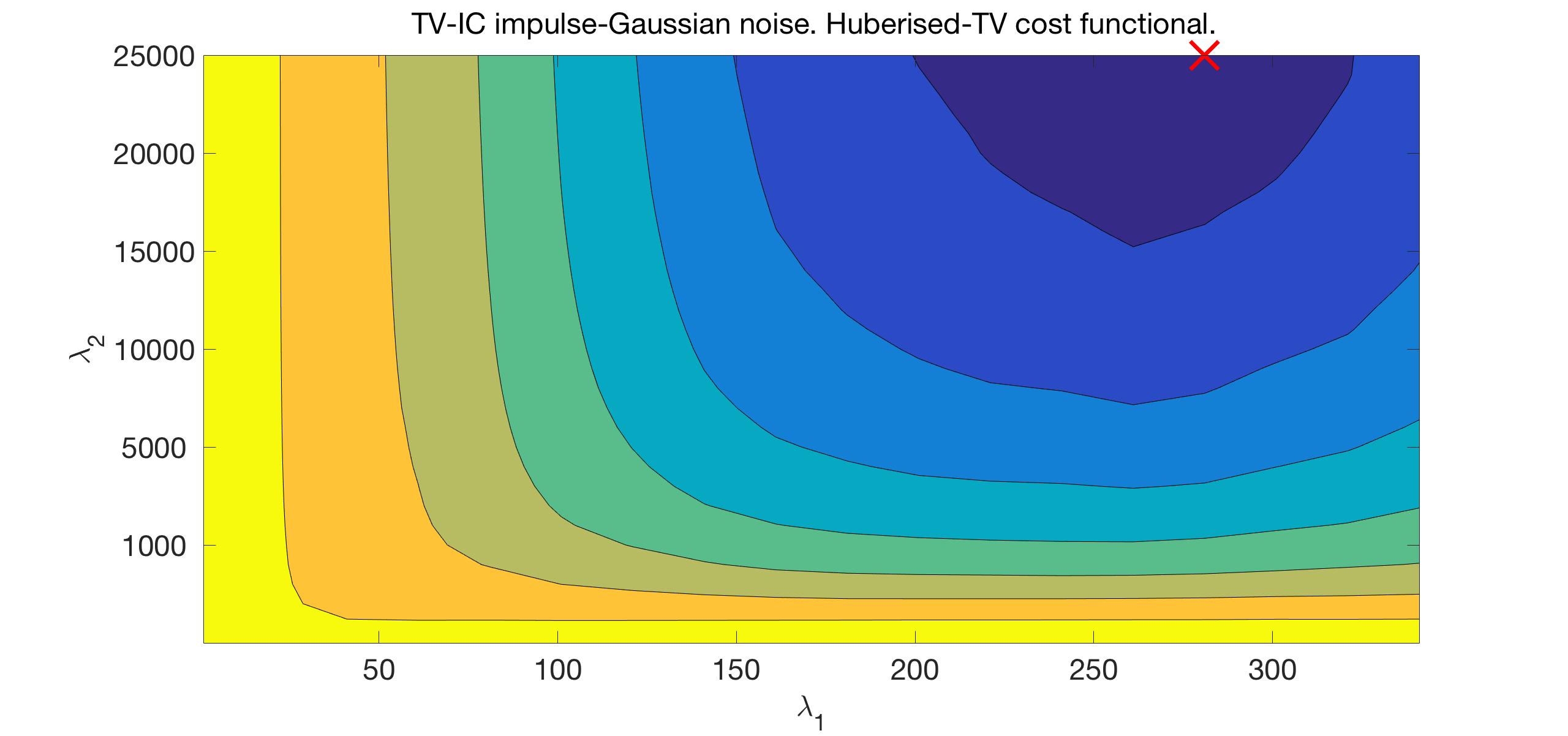}\vspace{0.1cm}
\caption{Optimal $\lambda_1,\lambda_2$ for TV-IC \eqref{gaussimp:numer}}
\label{optimalparams1}
\end{subfigure}\qquad
\begin{subfigure}[b]{0.45\textwidth}
\includegraphics[height=4cm,width=7cm]{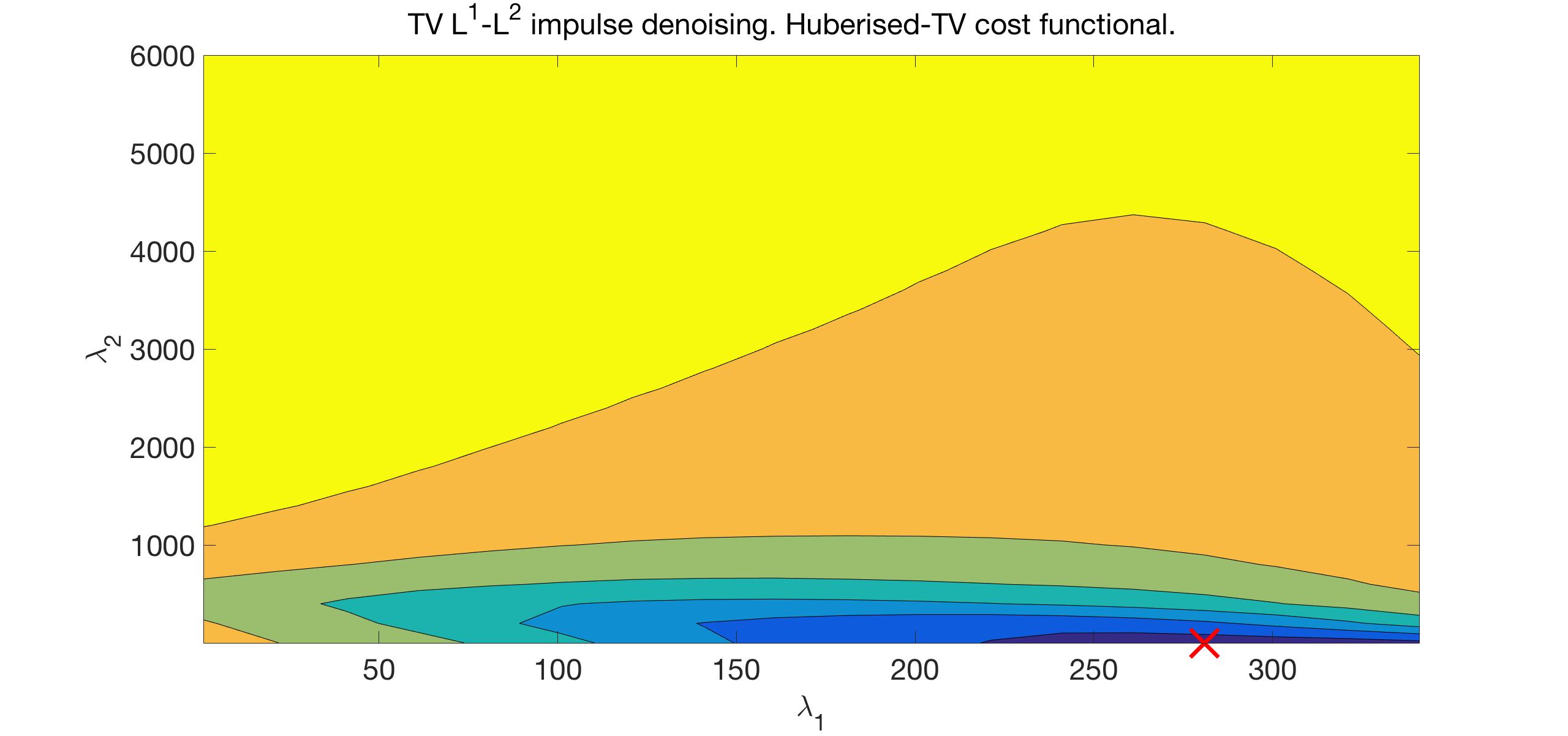}\vspace{0.1cm}
\caption{Optimal $\lambda_1,\lambda_2$ for \eqref{TVL1L2}}
\label{optimalparams2}
\end{subfigure}
\caption{Optimal value of parameters $\lambda_1^*$ and $\lambda_2^*$ for TV-IC \eqref{gaussimp:numer} and TV-$L^1$-$L^2$ denoising model \eqref{TVL1L2} applied to the brain image corrupted with salt \& pepper noise only, Figure \ref{noisyimageimponly} ($d=5\%$). Optimality is measured with respect to the Huberised-TV quality measure \eqref{huberisedTVcost}, compare \cite{interiorpaper2015}. In both cases, the combined model enforces the Gaussian component to be zero with the choice of the optimal $L^1$ weight $\lambda_1^*=281$. The optimum is depicted with a red cross.}
\label{fig:learning_impgasuss_onlyimp}
\end{figure}

Motivated by these preliminary results, let us next consider the formal treatment of the bilevel learning problem in case of a combination of salt \& pepper and Gaussian noise. We start from the general lower level Huber-regularised minimisation problem \label{jointminimisation_huber}. As in \cite{noiselearning}, we now introduce an elliptic-type regularisation for the derivation of the optimality system in an Hilbert space framework and consider:
\begin{equation} \label{regprob}
\min_{\substack{ u\in H^1_0(\Omega)\\ v\in L^2(\Omega)}} \left\{ \frac{\varepsilon}{2} \| \nabla u \|_{L^2}^2+|Du|_\gamma(\Omega) + \lambda_1 \|v\|_{L^1(\Omega),\gamma} + \frac{\lambda_2}{2} \|f-v-u\|^2_{L^2(\Omega)} \right\}.
\end{equation}

Then reduced bilevel optimisation problem, under the regularised constraint \eqref{regprob} reads
\begin{equation}   \label{costfunct}
\min_{\lambda_1\lambda_2\geq 0}~ \mathcal{F}(u)
\end{equation}
subject to \eqref{regprob}. The cost functional $\mathcal{F}$ is here assumed to be differentiable. For example, we can think about it as the squared $L^2$ \eqref{costfunctL2} or the Huberised TV cost functional \eqref{huberisedTVcost}. 
A necessary and sufficient optimality condition for \eqref{regprob} is then given by the following system:\begin{subequations} \label{regularoptimal}
\begin{equation}   \label{optimalityu}
\varepsilon(D\bar{u},Dw)_{L^2}  + (h_\gamma(D\bar{u}),Dw)_{L^2} - \lambda_2 (f- \bar v-\bar u,w)_{L^2}=0,\quad\text{for all }w\in H^1_0,
\end{equation}
\begin{equation}   \label{optimalityn}
\lambda_1(h_\gamma(\bar{v}),z)_{L^2}-\lambda_2 (f- \bar v-\bar u,z)_{L^2} =0\qquad\qquad\qquad\qquad\qquad\ \text{for all }z\in L^2(\Omega).
\end{equation}
\end{subequations}
where $(\bar{v},\bar{u})$ corresponds to the optimal solution and $h_\gamma(\cdot)$ is a smoother $C^2$ Huber-type regularisation, for exmaple of the form
\begin{equation*}
h_{\gamma}(z):=
\begin{cases}
\frac{z}{|z|} &\text{ if }~\gamma |z|-1 \geq \frac{1}{2\gamma}\\
 \frac{z}{|z|} (1- \frac{\gamma}{2} (1- \gamma |z|+\frac{1}{2\gamma})^2) &\text{ if }~\gamma |z|-1 \in (-\frac{1}{2\gamma}, \frac{1}{2\gamma})\\
\gamma z &\text{ if }~\gamma |z|-1 \leq -\frac{1}{2\gamma},
\end{cases}
\end{equation*}
similar to \cite[Eq.(3.11)]{noiselearning}. In strong form, the equations \eqref{regularoptimal} read as follows:
\begin{align*}
-&\varepsilon\Delta \bar{u}- \text{div}\left( h_\gamma(D\bar{u})\right)- \lambda_2 (f- \bar v-\bar u)=0  \\
& \lambda_1 h_\gamma(\bar{v})-\lambda_2 (f- \bar v-\bar u)=0.\quad
\end{align*}

We now use the Lagrangian formalism to get an insight in the structure of the optimality system. In order to do that, we write (\textbf{formally}) the Lagrangian functional of the problem \eqref{costfunct} subject to \eqref{regprob} as:
\begin{multline}  
\mathcal{L}(\lambda_1,\lambda_2,u,v,p_1,p_2):=\mathcal{F}(u)- (\varepsilon D u, D p_1) -\left( h_\gamma(Du), D p_1 \right)\\+\lambda_2 (f-v-u, p_1) \label{lagrangian}
- \lambda_1 (h_\gamma(v), p_2)+\lambda_2 (f-v-u, p_2).
\end{multline}

In \eqref{lagrangian} the adjoint states $p_1$ and $p_1$ need to be properly defined.
Now, calling $\Lambda$, $S$ and $P$ the first, the second and the third pair of arguments for $\mathcal{L}$ corresponding to the pair of control, state and adjoint variables, respectively, we know that in correspondence of an optimal solution $(\lambda^\gamma_1,\lambda^\gamma_2,u_\gamma,v_\gamma)$ (where now we stress the dependence on the regularising parameter $\gamma$), we have that the following two relations hold true:
\begin{align}
& \mathcal{L}_S(\lambda^\gamma_1,\lambda^\gamma_2,u_\gamma,v_\gamma,p^\gamma_1,p^\gamma_2)=0,  \label{lagrangian1}\\
& \mathcal{L}_\Lambda(\lambda^\gamma_1,\lambda^\gamma_2,u_\gamma,v_\gamma,p^\gamma_1,p^\gamma_2)((\alpha,\beta)^T-(\lambda_1^\gamma,\lambda_2^\gamma)^T)\geq 0,\quad\text{for every }\alpha,\beta\geq0. \label{lagrangian2}
\end{align}

So let us compute (formally) such derivatives in correspondence of the optimal solution:
\begin{align*}
&\frac{\partial\mathcal{L}}{\partial u} (\lambda^\gamma_1,\lambda^\gamma_2,u_\gamma,v_\gamma,p^\gamma_1,p^\gamma_2)[w_1] =  (\mathcal{F}'(u_\gamma),w_1)_{L^2} -\varepsilon (D p^\gamma_1, D w_1)_{L^2} - (h'_\gamma(Du_\gamma) Dw_1, D p^\gamma_1)_{L^2}  \\
& \hspace{2cm}-\lambda_2 ( p_1^\gamma,w_1 )_{L^2}-\lambda_2 (p_2^\gamma,w_1 )_{L^2}=0,\qquad\qquad\text{for all }w_1\in H^1_0,\\
& \frac{\partial\mathcal{L}}{\partial v} (\lambda^\gamma_1,\lambda^\gamma_2,u_\gamma,v_\gamma,p^\gamma_1,p^\gamma_2)[w_2] =-\lambda_2( p_1^\gamma, w_2 )_{L^2} - \lambda_1 (p_2^\gamma, h_\gamma' (v) w_2 )_{L^2} \\
& \hspace{2cm} -\lambda_2 (p_2^\gamma,w_2 )_{L^2}=0,\qquad\qquad\qquad\qquad\qquad\ \text{for all }w_2\in L^2 (\Omega),
\end{align*}
which we can rewrite as:
\begin{subequations}   \label{optimality2}
\begin{align*}
\varepsilon(Dp_1^\gamma,Dw_1)_{L^2} & + (h'_\gamma(Du_\gamma)^*Dp^\gamma_1, Dw_1)_{L^2} + \\ & -\lambda_2 (p_1^\gamma+ p_1^\gamma,w_1 )_{L^2}=(g'(u_\gamma),w_1)_{L^2}, &&\text{for all }w_1\in H^1_0,\\
& -\lambda_1 h_\gamma'(v)^* p_2^\gamma+ \lambda_2 (p_1^\gamma+p_2^\gamma)=0 && \text{a.e. in }\Omega.   \notag
\end{align*}
\end{subequations}

Proceeding in an analogous way in order to get the optimality condition \eqref{lagrangian2} we have:
\begin{subequations}  
\begin{align}
&\frac{\partial\mathcal{L}}{\partial\Lambda_1}(\lambda^\gamma_1,\lambda^\gamma_2,u_\gamma,v_\gamma,p^\gamma_1,p^\gamma_2)(\alpha-\lambda_1^\gamma)=\Bigl(\int_\Omega h_\gamma'(v_\gamma)p^\gamma_2\Bigr)(\lambda_1^\gamma-\alpha)\geq 0, \notag\\
&\frac{\partial\mathcal{L}}{\partial\Lambda_2}(\lambda^\gamma_1,\lambda^\gamma_2,u_\gamma,v_\gamma,p^\gamma_1,p^\gamma_2)(\beta-\lambda_2^\gamma)=\left( \int_\Omega  (f-v_\gamma-u_\gamma) (p_1^\gamma+p_2^\gamma) \right) (\lambda_2^\gamma-\beta)\geq 0 \notag
\end{align}
\end{subequations}
for every $\alpha, \beta\geq 0$. Introducing the multipliers $\mu_1:= - \int_\Omega h'_\gamma (v_\gamma)p_2^\gamma$ and $\mu_2:= \int_\Omega  (f-v_\gamma-u_\gamma) (p_1^\gamma+p_2^\gamma)$ it the follows that
\begin{equation*}
\lambda_i^\gamma \geq 0, \qquad \mu_i \geq 0, \qquad \mu_i \lambda_i=0, \qquad \text{ for }i=1,2.
\end{equation*}

Altogether, for an optimal quadruplet $(\lambda_1^\gamma,\lambda_2^\gamma,u_\gamma, v_\gamma)\in\R_{\geq 0}\times \R_{\geq 0}\times H^1_0(\Omega)\times (L^2(\Omega))$ there exist $(p_1^\gamma, p_2^\gamma)\in H^1_0(\Omega)\times( L^2(\Omega))$ such that the following optimality system holds (in strong form):
\begin{align*}
-&\varepsilon \Delta u_\gamma - \text{div }h_\gamma(D u_\gamma)- \lambda_2(f-v_\gamma-u_\gamma)=0 &&\text{in }\Omega\\
& u_\gamma=0 && \text{on } \Gamma\\
& \lambda_1 h_\gamma(v_\gamma) - \lambda_2 (f-v_\gamma-u_\gamma)=0 &&\text{a.e. in }\Omega\\
-&\varepsilon \Delta p_1^\gamma - \text{div }\left[ h_\gamma'(D u_\gamma)^* D p_1^\gamma \right] + \lambda_2(p_1^\gamma+p_2^\gamma)=g'(u_\gamma) &&\text{in }\Omega\\
& p_1^\gamma=0 && \text{on } \Gamma\\
& \lambda_1 h_\gamma'(v_\gamma)^* p_2^\gamma - \lambda_2 (p_1^\gamma+p_2^\gamma)=0 &&\text{a.e. in }\Omega\\
& \mu_1= - \int_\Omega h'_\gamma (v_\gamma)p_2^\gamma\\
& \mu_2= \int_\Omega  (f-v_\gamma-u_\gamma) (p_1^\gamma+p_2^\gamma)\\
& \lambda_i^\gamma \geq 0, \qquad \mu_i \geq 0, \qquad \mu_i \lambda_i=0 && \text{for }i=1,2.
\end{align*}

\medskip

\begin{remark}
For the Gaussian and Poisson framework described in Section \ref{subsec:gausspoiss} additional difficulties have to be tackled. In this case the admissible sets $\mathcal{A}$ and $\mathcal{B}$ in \eqref{admissiblesetsgausspoiss} are required to guarantee convexity of $\phi$ in $u$. However, since this involves pointwise bounds on a state variable, the existence of Lagrange multipliers (even of low regularity) has to be carefully justified and the numerical solution of the problem becomes challenging.
\end{remark}

Another possible research direction is the study of the optimality system derived above in the limit as $\gamma\to\infty$ in order to show connections with the non-smooth original problem.

\section*{Acknowledgments}
The authors would like to thank Yi Yu, Sara Sommariva and Evangelos Papoutsellis for discussions on the statistical and analytical interpretation of the model. LC is grateful to Anna Jezierska for providing the code used in \cite{Jezierska2012,poissongauss2013,AnnaThesis2013} for the comparisons of results in Figure \ref{fig:gausspoisscomparison}. LC acknowledges the UK Engineering and Physical Sciences Research Council (EPSRC) grant Nr. EP/H023348/1 of the University of Cambridge Centre for Doctoral Training, the Cambridge Centre for Analysis (CCA) and the joint ANR/FWF Project ``Efficient Algorithms for Nonsmooth Optimization in Imaging" (EANOI) FWF n. I1148 / ANR-12-IS01-0003. CBS acknowledges support from Leverhulme Trust project on Breaking the non-convexity barrier, EPSRC grant Nr. EP/M00483X/1, the EPSRC Centre Nr. EP/N014588/1 and the Cantab Capital Institute for the Mathematics of Information. This research was partially funded by the Escuela Politécnica Nacional de Ecuador under award PIJ-15-22.

\begin{appendices}

\section{Proofs}   \label{append:proofs}

In this Appendix we include the proofs of Propositions \ref{wellposednimpgauss} and \ref{wellposednessgausspoiss}.

\begin{proof}[Proof of Proposition \ref{wellposednimpgauss}]
As observed in Remark \ref{remark:prox_map}, the functional $\Phi^{\lambda_1,\lambda_2}$ is the proximal map of the $L^1$ norm functional in the Hilbert space $L^2(\Omega)$ weighted by the parameter $\lambda_1/\lambda_2$ and evaluated in the point $f-u$  \cite[Section 12.4]{bauschkecombettes}. Moreover, in \cite{VaggelisIC2015} the authors show that such combination is nothing but the Huber regularisation of the $L^1$ norm. Existence and uniqueness properties can be shown directly by standard lower semicontinuity and strict convexity properties of $\mathscr{F}^{\lambda_1,\lambda_2}(f,u,\cdot)$ in $L^2(\Omega)$.
\end{proof}

\begin{proof}[Proof of Proposition \ref{wellposednessgausspoiss}]
Since the  $D_{KL}$ functional is positive  (see Remark \ref{positivityKL}), the functional $\mathscr{F}^{\lambda_1,\lambda_2}$ is bounded from below.  
Let us consider then a minimising sequence $\left\{ v_n \right\} \subset L^2(\Omega)\cap\mathcal{B}$ for the functional $\mathscr{F}^{\lambda_1,\lambda_2}(f,u,\cdot)$. Using the positivity of $D_{KL}$, we have that:
\begin{equation*}
\frac{\lambda_1}{2}\| v_n \|_{L^2(\Omega)}^2\leq \mathscr{F}^{\lambda_1,\lambda_2}(f,u,v_n)\leq C,\quad\text{ for all }n\geq 1.
\end{equation*}
Hence, we can extract a non-relabelled subsequence $\{ v_n \}$ weakly converging to $v$ in $L^2(\Omega)$. As $\Omega$ is bounded and $L^2(\Omega)$ is continuously embedded $L^1(\Omega)$, $\{ v_n \}$ is also converging weakly to $v$ in $L^1(\Omega)$. Due the continuity of the $L^2$ norm and the weak lower semicontinuity of $D_{KL}(\cdot, u)$ for a fixed nonnegative $u\in L^2(\Omega)$ with respect to the weak topology of $L^1(\Omega)$ (compare Proposition \ref{propos:KLproperties}), we have:
\begin{multline} 
 \mathscr{F}^{\lambda_1,\lambda_2}(f,u,v)=\frac{\lambda_1}{2}\| v \|_{L^2(\Omega)}^2 + \lambda_2 D_{KL}(f-v,u) \notag \\ \leq \liminf_{n\to\infty}~ \frac{\lambda_1}{2}\| v_n \|_{L^2(\Omega)}^2+D_{KL}(f-v_n,u)=~\liminf_{n\to\infty}~ \ \mathscr{F}^{\lambda_1,\lambda_2}(f,u,v_n). \notag
\end{multline}
Hence, $\mathscr{F}^{\lambda_1,\lambda_2}(f,u,\cdot)$ is weakly lower semicontinuous in $L^2(\Omega)$. To show that $v$ is an element of $\mathcal{B}$ is sufficient to observe that $\mathcal{B}$ is convex and closed in $L^2(\Omega)$ and hence weakly closed by Mazur's lemma. Uniqueness of the minimiser follows by strict convexity of $\mathscr{F}^{\lambda_1,\lambda_2}(f,u,\cdot)$.
\end{proof}

\section{The Kullback-Leibler functional}    \label{append:KLfunct}

We recall some general definitions and results on the Kullback-Leibler functional which have been studied in \cite{eggermont1993,resmerita2007,borwein1991,Saw11} and used in this work for the analysis of the model \eqref{gausspoissinfconvfidelity}. 

\begin{definition}[Kullback-Leibler functional] Let $\Omega\subset\R^d,~ d>1$ be a regular domain and $\mu$ a measure on $\Omega$. The Kullback-Leibler (KL) functional is the function: $D_{KL}: L^1(\Omega)\times L^1(\Omega)\to \R^+\cup{+\infty}$ defined by:
\begin{equation}   \label{generaldefKL}
D_{KL}(\varphi,\psi)=\int_\Omega \left( \varphi\log\left( \frac{\varphi}{\psi} \right) - \varphi+\psi\right) d\mu\quad\text{ for every } \varphi,~\psi\geq 0~\text{ a. e. }
\end{equation}
\end{definition}

\begin{remark}    \label{positivityKL}
Here and throughout the paper, we have uses the convention $0\log 0=0$ by which we have that the integrand function in \eqref{generaldefKL} is nonnegative (by convexity of the function $f(x)=x\log x$) and vanishes if and only if $\varphi=\psi$.
\end{remark}

\smallskip

The following Proposition and Corollary collect some results from \cite[Section 3.4]{resmerita2007} and \cite[Lemma 4.6.3]{Saw11} on convexity and weak lower semicontinuity  properties of the functional $D_{KL}$ and are formulated to be adapted to our study. 

\begin{proposition}   \label{propos:KLproperties}
Let $D_{KL}$ be defined as in \eqref{generaldefKL}. The following properties hold:
\begin{enumerate}
\item The function $(\varphi,\psi)\mapsto D_{KL}(\varphi,\psi)$ is convex. 
\item For any fixed nonnegative $\varphi\in L^1(\Omega)$, the function $D_{KL}(\varphi,\cdot)$ is lower semicontinuous with respect to the weak topology of $L^1(\Omega)$.
\item For any fixed nonnegative and bounded $\psi\in L^1(\Omega)$,  the function $D_{KL}(\cdot,\psi)$ is lower semicontinuous with respect to the weak topology of $L^1(\Omega)$.
\item For any nonnegative functions $\varphi, \psi\in L^1(\Omega)$ the following estimate holds:
\begin{equation}   \label{KLestimate}
\| \varphi-\psi \|_{L^1(\Omega)}^2 \leq \left( \frac{2}{3} \| \varphi \|_{L^1(\Omega)} +  \frac{4}{3} \| \psi \|_{L^1(\Omega)} \right) D_{KL}(\varphi,\psi).
\end{equation}
\end{enumerate}
\end{proposition}

\begin{corollary}   \label{corollaryKL}
Let $\left\{ \varphi_n \right\}$ and  $\left\{ \psi_n \right\}$ are bounded sequences in $L^1(\Omega)$. Then, by \eqref{KLestimate}
\begin{equation}
\lim_{n\to \infty} D_{KL} (\varphi_n,\psi_n) = 0 \quad\Rightarrow\quad \lim_{n\to\infty} \| \varphi_n-\psi_n \|_{L^1(\Omega)} = 0.
\end{equation}
\end{corollary}

\end{appendices}

\bibliographystyle{plain} 
\bibliography{infconvbib}

\end{document}